\documentclass[a4paper]{article}
\usepackage{authblk}
\usepackage{amsmath}
\usepackage{amsfonts}
\usepackage{amssymb}
\usepackage{graphicx}
\usepackage{xcolor}
\usepackage{natbib}
\usepackage[top=2cm,bottom=2cm,right=2cm,left=2cm]{geometry}
\usepackage{hyperref}

\newtheorem{theorem}{Theorem}

\newtheorem{corollary}[theorem]{Corollary}
\newtheorem{definition}[theorem]{Definition}

\newtheorem{lemma}[theorem]{Lemma}
\newtheorem{proposition}[theorem]{Proposition}
\newtheorem{remark}[theorem]{Remark}

\graphicspath{{Images}}

\begin{document}

\title{Minimax prediction in the Functional Autoregressive Model}
\author[1]{André Mas}
\author[2]{Angelina Roche}
\affil[1]{IMAG, Univ. Montpellier, CNRS}
\affil[2]{Université Paris Cité, CNRS, MAP5, F-75006 Paris, France}
\date{}
\maketitle

\begin{abstract}
The Functional Autoregressive Model (FAR) generalizes the multivariate AR(1) model in Time Series Analysis to functional data. It serves as a historical foundational point in the study of functional time series and remains a fundamental and widely used model for dependent functional data. The process—the observed data—generated by the FAR model forms a Hilbert-valued Markov chain. This paper investigates the non-asymptotic prediction mean square error and derives a lower bound. This lower bound is established in the specific context of non-i.i.d. data and depends on the mixed smoothness of the functional time series and of the unknown correlation operator driving the FAR model. Instead of the standard functional PCA regularization, a ridge-type estimator is proposed, which avoids the preliminary estimation of the spectrum of the covariance sequence associated with the process. A non-asymptotic upper bound is derived for this estimate, which matches the lower bound up to multiplicative constants. Furthermore, a detailed study of the estimate's bias reveals connections between functional smoothness parameters and regularly/rapidly varying functions, which are common in extreme value theory. Simulation results corroborate the theoretical main theorems.
\end{abstract}

\section{Introduction}

Functional data analysis is an established area of modern statistics. Pioneering works dating back to the 1970s and 1980s focused on the inference of first- or second-order moments \cite{Kleffe1973,Dauxois1982}. From its inception, FDA branched out into several directions corresponding to mainstream classical statistics, leading to fruitful research areas such as exploratory analysis, regression analysis, spatial functional data, and Bayesian analysis for functional data. A series of monographs and state-of-the-art articles provide comprehensive overviews of both theoretical and applied aspects (\cite{Ramsay2005,FerratyVieu2006,Kokoszka2012-vp,Cuevas2014,Hsing2015}).

Functional Time Series Analysis (FTSA) represents a highly active subdomain. The underlying principle consists in extending popular multivariate time series models to functional data. A comprehensive introduction to FTSA is provided by \cite{Aue2015}. FTSA typically addresses dependent data and frequently relies on tools from stochastic analysis (see, e.g., \cite{Merlevede1996,Bosq2007,hormann2010,kühnert2025}).

The functional autoregressive process examined in this work is the foundational model introduced in FTSA, originally proposed by Bosq in \cite{Bosq1991}. By contrast, the linear regression model appeared only a few years later in \cite{CardotFLM99}.

A functional (Hilbertian) autoregressive process generalizes the scalar AR(1) and multivariate VAR(1) models to functional data. It is defined as a stationary solution to the equation:
\begin{equation*}
X_{n}-\boldsymbol\mu=\rho\left( X_{n-1}-\boldsymbol\mu\right) +\varepsilon_{n},\quad n\in\mathbb{Z}
\end{equation*}
where $X_{n}$ is the $n$-th observation in a sample of dependent functional data, $\boldsymbol\mu=\mathbb{E}X_1$, $\rho$ is a bounded linear operator, and $\left(\varepsilon_{n}\right)_{n\in\mathbb{Z}}$ is a Hilbertian white noise. The theoretical properties of this model have been extensively studied in the monograph \cite{Bosq2000}. Various inference strategies have been proposed in the literature, including sieves \cite{BENSMAIN2001}, wavelets \cite{ANTONIADIS2003}, RKHS \cite{Mokhtari2003}, and low-rank approximation \cite{KARGIN2008}. Additionally, an R-package was developed by Damon and Guillas \cite{Damon2005}.

Further significant contributions to the model are detailed in Chapters 13, 14, and 15 of \cite{Kokoszka2012-vp}, which address change-point problems and order identification. An extensive study on prediction is conducted in \cite{Didericksen2011}.

Several works have focused on the theoretical question of the convergence rates of estimators or predictors, typically within an asymptotic framework (e.g., \cite{Mas2007,CAPONERA2022}). Various extensions have also been proposed: in \cite{RuizMedina2011}, the model is adapted to spatial data; in \cite{Pumo1998,RuizMedina2019}, to Banach-valued data; and in \cite{Caponera2021,Caponera2021bis}, to random functions on the sphere.

Finally, numerous applications have been implemented across diverse fields such as climatology (\cite{Besse2000, Damon02, AlvarezLibana2019}), finance (\cite{Cai2018}), econometrics (\cite{franchi_paruolo_2019}), and electricity consumption. The model is well-suited for any framework involving dependent functional data with an underlying Markovian structure.

Most existing estimation strategies rely on functional PCA as a dimension reduction technique for the nonparametric estimation of $\rho$. This approach requires estimating the spectrum (eigenvalues and eigenvectors) of the covariance operator of the stationary process $X_{n}$, denoted by $C$. While eigenvalue estimation is relatively stable, approximating eigenfunctions is more challenging and prone to numerical perturbations. This limitation motivates the alternative strategies explored in this work.

This paper represents an initial attempt to establish optimality results in the minimax sense for a dependent functional data model and to provide non-asymptotic guarantees. The remaining sections are organized as follows. Section 2 introduces the fundamental concepts and tools of functional analysis. Section 3 describes the FAR model, presents the proposed estimator, and details the core assumptions regarding moments and regularity. Note that the term \textit{regularity} is preferred to \textit{smoothness} since the unknown parameters of interest are linear operators rather than classical functions. Section 4 presents the main results, establishing first a lower bound and subsequently deriving the upper bound for the estimator. The proofs are collected in Section 7, complemented by an Appendix.

\section{Functional framework}
This section introduces the core concepts of functional analysis required for the subsequent developments. The presentation is restricted to the essential prerequisites. For a comprehensive overview of Hilbert spaces, linear operators, or Gaussian analysis, the reader is referred to \cite{Ramsay2005}, \cite{Lifshits2010}, and \cite{Hsing2015}.

\subsection{Hilbert spaces and operators}

Let $\mathbf{H}$ be a real separable Hilbert space of functions equipped with the inner product $\left\langle \cdot,\cdot\right\rangle $ and the norm $\left\Vert\cdot\right\Vert $. Typically, $\mathbf{H}$ can be chosen as $L^{2}\left(\Omega\right)$ for a compact set $\Omega$, or as a Sobolev space $H^{2,m}\left(\Omega\right)$ if specific smoothness is required: 
\begin{equation}
H^{2,m}\left( \Omega\right) =\left\{ f\in L^{2}\left( \Omega\right)
:f^{\left( m\right) }\in L^{2}\left( \Omega\right) \right\} .  \label{sob}
\end{equation}
In the statistical model under investigation, the unknown parameter is a bounded linear operator on $\mathbf{H}$. The Banach space of bounded linear operators from $\mathbf{H}$ into $\mathbf{H}$ is denoted by $\mathcal{L}\left(\mathbf{H}\right)$ and is equipped with the standard operator norm $\left\Vert \cdot\right\Vert _{\infty}$ defined for all $T\in\mathcal{L}\left( \mathbf{H}\right) $ by $\left\Vert T\right\Vert _{\infty}=\sup_{\left\Vert x\right\Vert =1}\left\Vert Tx\right\Vert $. 

The space $\mathcal{L}\left( \mathbf{H}\right) $ is too general for our purposes and lacks separability. The space of compact operators $\mathcal{L} _{0}\left( \mathbf{H}\right) \subset\mathcal{L}\left( \mathbf{H}\right) $ provides a more suitable framework. An operator $T\in\mathcal{L}\left(\mathbf{H}\right)$ is compact if the image of the unit ball of $\mathbf{H}$ under $T$ is relatively compact. Let $\otimes$ denote the outer product of two vectors $x, y \in \mathbf{H}$, where $x\otimes y$ is a rank-one operator defined for all $h \in \mathbf{H}$ by $\left(x\otimes y\right) \left( h\right) =\left\langle y,h\right\rangle x$. By the Spectral Theorem, a self-adjoint compact operator $S$ admits the decomposition: $S=\sum_{k\in\mathbb{N}}s_{k}\left\langle \phi_{k},\cdot\right\rangle \phi_{k}=\sum_{k\in\mathbb{N}}s_{k}\phi_{k} \otimes\phi_{k}$, where $(s_{k},\phi_{k})\in \mathbb{R}^{+}\times \mathbf{H}$ for all $k$, $\left(s_{k}\right)_{k\in\mathbb{N}}\in c_{0}$ (i.e., $\lim s_{k}=0$), and $\left(\phi_{k}\right)_{k\in\mathbb{N}}$ forms an orthonormal system. The square root of a non-negative compact symmetric operator $S=\sum_{k\in\mathbb{N}}s_k\phi_k\otimes\phi_k$ is defined as $S^{1/2}=\sum_{k\in\mathbb{N}}s_{k}^{1/2}\left( \phi_{k}\otimes\phi_{k}\right) $. 

To define important classes of linear operators that arise naturally in this context, we introduce the Schatten $p$-class. For $p>0$, the Schatten $p$-class $\mathcal{S}_{p}$ is a Banach space defined by: 
\begin{equation*}
\mathcal{S}_{p}=\left\{ T\in\mathcal{L}_{0}\left( \mathbf{H}\right) :\left\Vert
T\right\Vert _{p}=\left(\mathrm{tr}\left( U^{p}\right)\right)^{1/p} <+\infty,\quad U=\left(
T^{\ast}T\right) ^{1/2}\right\}
\end{equation*}
where $\mathrm{tr}$ denotes the operator trace and $\left( T^{\ast}T\right)^{1/2}$ is the scaling operator of $T$. The trace functional is defined by $\mathrm{tr}\left( U\right) =\sum_{k}\left\langle U\phi_{k},\phi_{k}\right\rangle $, where $(\phi_{k})$ is an arbitrary orthonormal basis of $\mathbf{H}$. The definition of $\mathrm{tr}\left( U\right) $ (and consequently of $\left\Vert T\right\Vert _{p}$) is independent of the choice of basis. We note that $\mathrm{tr}(T)=\mathrm{tr}(T^*)$ and $\mathrm{tr}(TS) = \mathrm{tr}(ST)$. Alternatively, the Schatten class can be characterized as $\mathcal{S}_{p}=\left\{ T\in \mathcal{L}_{0}\left( \mathbf{H}\right) :\left\Vert T\right\Vert _{p}=\left(\sum_{k\in\mathbb{N}}s_{k}^{p}\right)^{1/p}<+\infty\right\} $. An essential property is that if $T\in\mathcal{L}\left( \mathbf{H}\right) $ and $S\in\mathcal{S}_{p}$, then both $TS$ and $ST$ belong to $\mathcal{S}_{p}$.

Two specific classes play a crucial role in this paper: $\mathcal{S}_{2}$ and $\mathcal{S}_{1}$. The space $\mathcal{S}_{1}$ is the space of trace-class (or nuclear) operators, which are of primary interest in FDA due to their connection with covariance operators. The space $\mathcal{S}_{2}$ comprises Hilbert-Schmidt operators and serves as the preferred framework for studying the random operators introduced below. It is a separable Hilbert space equipped with the inner product $\left\langle T,S\right\rangle _{2}=\mathrm{tr}\left( T^{\ast}S\right) $ and the norm $\left\Vert T\right\Vert _{2}^{2}=\sum_{k\in\mathbb{N}}\left\Vert T\phi_{k}\right\Vert ^{2}=\mathrm{tr}\left( T^{\ast}T\right) $. The following standard inequalities hold: 
\begin{align*}
\left\vert \mathrm{tr}\left( TS\right) \right\vert & \leq\left\Vert
T\right\Vert _{2}\left\Vert S\right\Vert _{2}\quad \text{for } T,S\in\mathcal{S}_{2}, \\
\left\vert \mathrm{tr}\left( UV\right) \right\vert & \leq\left\Vert
U\right\Vert _{1}\left\Vert V\right\Vert _{\infty}\quad \text{for } U\in\mathcal{S}_{1}, V\in\mathcal{L}\left( \mathbf{H}\right) .
\end{align*}
When $T$ is compact, symmetric, and positive-definite, it possesses a denumerable sequence of positive eigenvalues, denoted by $\left( \lambda_{k}\right)_{k\in\mathbb{N}}$. In this case, $\left\Vert T\right\Vert _{2}^{2}=\sum_{k\in\mathbb{N}}\lambda_{k}^{2}$ and $\left\Vert T\right\Vert _{1}=\sum_{k\in\mathbb{N}}\lambda_{k}$. 

Unbounded operators on Hilbert spaces will also be encountered in this work. While the use of unbounded operators is less common in statistics, differential operators defined on appropriate domains provide typical examples. For instance, consider $T:H^{2,1}\left( \left[ 0,1\right] \right)\rightarrow L^{2}\left( \left[ 0,1\right] \right)$ defined by $Tf=f^{\prime}$ for $f\in L^{2}\left( \left[ 0,1\right] \right) $. It is straightforward to verify that $T$ is unbounded on $\mathbf{H}=L^{2}\left( \left[ 0,1\right] \right) $ (e.g., consider $f\left( t\right) =t^{-1/4}$). However, $T$ is bounded on $H^{2,1}\left( \left[ 0,1\right] \right) $ as defined in (\ref{sob}), and $H^{2,1}\left( \left[ 0,1\right] \right) $ is dense in $L^{2}\left( \left[ 0,1\right] \right) $.

Differential operators can be viewed, to some extent, as the inverses of integral operators. More generally, an important class of unbounded operators can be defined by inverting injective compact operators. Let $T$ be a compact self-adjoint operator admitting the spectral decomposition $T=\sum_{k\in\mathbb{N}}s_{k}\left\langle \phi_{k},\cdot\right\rangle \phi_{k}=\sum_{k\in\mathbb{N}}s_{k}\left( \phi_{k}\otimes\phi_{k}\right) $, where $\left( s_{k}\right)_{k\in\mathbb{N}}$ is a positive real sequence converging to zero. Typically, $T$ represents the covariance operator of a random function. If $s_{k}>0$ for all $k$, $T$ is injective, allowing the formal definition of the inverse operator $T^{-1}=\sum_{k\in\mathbb{N}}s_{k}^{-1}\left\langle \phi_{k},\cdot\right\rangle \phi_{k}.$ However, for an arbitrary $x\in\mathbf{H}$, the series $\left\Vert T^{-1}x\right\Vert ^{2}=\sum_{k\in\mathbb{N}}s_{k}^{-2}\left\langle \phi_{k},x\right\rangle ^{2} $ may diverge. Consequently, $T^{-1}$ is unbounded and discontinuous. Its domain is defined as: 
\begin{equation*}
\mathcal{D}\left( T^{-1}\right) =\left\{ x\in\mathbf{H:}\sum_{k\in \mathbb{N}
}\frac{\left\langle \phi_{k},x\right\rangle ^{2}}{s_{k}^{2}}<+\infty\right\}
.
\end{equation*}
The domain $\mathcal{D}\left( T^{-1}\right) $ has an empty interior but is dense in $\mathbf{H}$, and $T^{-1}$ defines a measurable mapping from $\mathcal{D}\left( T^{-1}\right) $ into $\mathbf{H}$. Note that $\mathcal{D}\left( T^{-1}\right)$ corresponds to the RKHS associated with the kernel $K(s,t) = \sum_{k\geq 1}s_k\phi_k(s)\phi_k(t)$. In the sequel, all encountered unbounded linear operators are assumed to have dense domains. The general theory of closed unbounded operators with dense domains falls outside the scope of this paper; only the basic properties summarized in Section 7.1 of the Appendix will be utilized.

\subsection{Functional data}

\subsubsection{General functional data}

Let $X$ be a random function, modeled as an $\mathbf{H}$-valued random element. The distribution of $X$ defines a probability measure on $\mathbf{H}$. The expectation $\mathbb{E}X$ is a deterministic element of $\mathbf{H}$. The covariance structure of $X$ is characterized by a bounded linear operator $C:\mathbf{H\rightarrow H}$, defined for all $x\in \mathbf{H}$ by $C\left(x\right) =\mathbb{E}\left[ \left\langle X-\mathbb{E}X,x\right\rangle (X-\mathbb{E}X)\right] $. The operator $C$ generalizes the classical covariance matrix to the infinite-dimensional setting and is positive and self-adjoint by construction. Under the standard assumption that $\mathbb{E}\left\Vert X\right\Vert ^{2}<+\infty $, the domain of $C$ is the entire space $\mathbf{H}$. Furthermore, since $\mathbb E[\|X\|^2]=\mathrm{tr}(C)+\|\mathbb E X\|^2$, it follows that $C\in \mathcal{S}_{1}$. 

Let $\mathbf{L}_{0}^{m}\left( \mathbf{H}\right) $ denote the set of $\mathbf{H}$-valued centered random variables $Z$ satisfying $\mathbb{E}\left\Vert Z\right\Vert^{m}<+\infty $.

As a compact operator, $C$ possesses a pure point spectrum. Let $\left(\lambda_{k},e_{k}\right)_{k\in \mathbb{N}^{\ast }}$ denote the eigenvalue-eigenvector pairs of $C$, which define the functional PCA of $X$. The eigenvalues are positive and summable. We assume that $C$ is injective, that the eigenvalues are simple and sorted in strictly decreasing order: $\lambda_{1}>\dots>\lambda_{k}>\lambda_{k+1}>\dots>0$ for all $k$, and normalized such that $\max_k \lbrace k\lambda_k\rbrace \leq 1$. Multiple eigenvalues could be accommodated with minor modifications to the notation and results. The operator $C^{-1}$ is referred to as the precision operator of $X$. In most applications, covariance operators are integral operators with continuous kernels, and their corresponding precision operators behave similarly to differential operators. The Karhunen-Loève expansion of $X$ is given by: 
\begin{equation}
X=\mathbb{E}X+\sum_{k=1}^{+\infty }\sqrt{\lambda _{k}}\eta _{k}e_{k}
\label{Karhunen-L}
\end{equation}
where the real random variables $\eta _{k}=\left\langle X-\mathbb{E}X,e_{k}\right\rangle /\sqrt{\lambda _{k}}$ are centered with unit variance and are mutually uncorrelated. This decomposition generalizes the classical Fourier series expansion. For any $T\in \mathcal{L}\left(\mathbf{H}\right)$, the following identities hold:
\begin{align*}
\mathbb{E}\left\langle TX,X\right\rangle & =\sum_{k\geq 1}\mathbb E[ \left\langle X,T^\ast e_{k}\right\rangle \left\langle e_{k},X\right\rangle ]=\mathrm{tr}\left(TC\right),\\
\mathbb{E}\left\Vert TX\right\Vert ^{2}& =\mathbb{E}\left\langle T^\ast TX,X\right\rangle =\mathrm{tr}\left( T^{\ast }TC\right) =\mathrm{tr}((TC^{1/2})^\ast (TC^{1/2}))=\left\Vert C^{1/2}T\right\Vert _{2}^{2}.
\end{align*}

\subsubsection{Gaussian functional data\label{subsec-gauss-func-data}}

A Gaussian random function is a Hilbert-valued random element $X$ such that $\left\langle X,h\right\rangle$ is a real Gaussian random variable for all $h\in\mathbf{H}$. The distribution of a Gaussian element $X$ is uniquely determined by its mean $m$ and its covariance operator $C$, with $\left\langle X,h\right\rangle \sim\mathcal{N}\left( \mu_{h},\sigma_{h}^{2}\right) $, where $\mu _{h}=\mathbb{E}\left\langle X,h\right\rangle =\left\langle m,h\right\rangle $ and $\sigma_{h}^{2}=\mathbb{E}\left\langle X,h\right\rangle ^{2}=\left\langle Ch,h\right\rangle $. When $X$ is Gaussian, the random variables $\eta_{k}$ in (\ref{Karhunen-L}) are mutually independent. For a detailed treatment of the theory of Gaussian random functions, see \cite{Lifshits2010}. 

The Feldman-Hájek Theorem provides necessary and sufficient conditions for the equivalence or mutual singularity of two Gaussian measures on a Hilbert space. Following the formulation in \cite{DPZ}, this result will be used to establish the minimax lower bound.

\begin{theorem}
\label{origFeldHayek}(Feldman--H\'{a}jek 1958) Let $\mu=\mathcal{N}\left(m_{1},\Sigma_{1}\right) $ and $\nu=\mathcal{N}\left( m_{2},\Sigma_{2}\right) $ be two Gaussian measures on $\mathbf{H}$. Then $\mu$ and $\nu$ are either mutually singular or equivalent. They are equivalent if and only if the following three conditions hold: (i) $\mathrm{Im}\Sigma_{1}^{1/2}=\mathrm{Im}\Sigma_{2}^{1/2}=H_{0}$, (ii) $m_{1}-m_{2}\in H_{0}$, and (iii) the operator $\left( \Sigma_{1}^{1/2}\Sigma_{2}^{-1/2}\right) \left( \Sigma_{1}^{1/2}\Sigma_{2}^{-1/2}\right) ^{\ast}-\mathbf{I}_{\mathbf{H}}$ is Hilbert-Schmidt on $\overline{H}_{0}$.
\end{theorem}

\section{Model\label{section-model}}

\subsection{Estimation procedure}

Without loss of generality, the autoregressive process is assumed to be centered:
\begin{equation}
X_{n}=\rho\left( X_{n-1}\right) +\varepsilon_{n}\quad n\in\mathbb{Z},
\label{model}
\end{equation}
where $\left( \varepsilon_{n}\right)_{n\in \mathbb{Z}}$ is an i.i.d. centered sequence satisfying $\mathbb{E}\left\Vert \varepsilon_1 \right\Vert^2<+\infty$. We denote the following covariance operators:
\begin{equation*}
C=\mathbb{E}\left( X_{1}\otimes X_{1}\right) ,\quad D=\mathbb{E}\left(X_{1}\otimes X_{2}\right), \quad  C_{\varepsilon}=\mathbb{E}\left( \varepsilon_{1}\otimes\varepsilon_{1}\right).
\end{equation*}
As established by Bosq (\cite{Bosq2000}, Theorem 3.1), a unique stationary solution to (\ref{model}) exists provided that:
\begin{equation} \label{A0}
\left\Vert \rho\right\Vert _{\infty}<1.
\end{equation}
Under this condition, $X_0 = \sum_{n=-\infty}^0 \rho^{-n}(\varepsilon_n)$. Taking the outer product on both sides of (\ref{model}) yields the fundamental lyceum equation linking $C$ and $C_{\varepsilon}$:
\begin{equation}
C=\rho C\rho^{\ast}+C_{\varepsilon}.  \label{cov}
\end{equation}
The identification of $\rho$ relies on second-order moments. The corresponding normal equation is analogous to the finite-dimensional case:
\begin{equation}
D=\rho C.  \label{mom}
\end{equation}
While equation (\ref{mom}) is common to general linear functional regression models, equation (\ref{cov}) is specific to the autoregressive structure (\ref{model}). The following proposition summarizes the identifiability properties of the model (see, e.g., \cite{crambes2013}, Section 1.3).

\begin{proposition}
If $\ker C\neq\left\{ 0\right\} $, the parameter $\rho$ is not identifiable. If $\ker C=\left\{ 0\right\} $, which is assumed throughout, a solution to equation (\ref{mom}) uniquely defines $\rho$ only on the dense subspace $\mathrm{Im}C$. The extension of $\rho_{|\mathrm{Im}C}$ to the entire space $\mathbf{H}$ follows from the Hahn-Banach theorem.
\end{proposition}

Equation (\ref{mom}) implies that $\rho$ cannot be directly recovered via $DC^{-1}$ since the inverse operator $C^{-1}$ is unbounded and discontinuous everywhere. Because the true covariance operator $C$ is unknown, estimating $C^{-1}$ constitutes an ill-posed inverse problem that requires regularization.

A standard ordinary least squares framework minimizes the objective:
\begin{equation*}
\min_{\rho}\sum_{k=1}^{n-1}\left\Vert X_{k+1}-\rho\left( X_{k}\right)\right\Vert ^{2}.
\end{equation*}
The first-order optimality condition leads to the empirical sample equation $D_{n}=\widehat{\rho}_{n}C_{n}$, where:
\begin{equation*}
C_{n}=\dfrac{1}{n}\sum_{k=1}^{n-1}X_{k}\otimes X_{k},\quad D_{n}=\dfrac{1}{n}\sum_{k=1}^{n-1}X_{k+1}\otimes X_{k}.
\end{equation*}
Since the empirical covariance operator $C_{n}$ is random and has finite rank, it cannot be inverted directly. A widely used regularized approach replaces $C_{n}$ with a low-dimensional spectral truncation based on functional PCA. Let $\left( \widehat{\lambda}_{k},\widehat{e}_{k}\right)$ denote the eigenvalue-eigenvector pairs of $C_{n}$. A generalized inverse can be defined as:
\begin{equation*}
C_{n}^{\natural}=\sum_{k=1}^{K_{n}}\frac{1}{\widehat{\lambda}_{k}}\widehat{e}_{k}\otimes\widehat{e}_{k}
\end{equation*}
where the truncation parameter $K_{n}$ must be appropriately calibrated as the sample size $n$ tends to infinity. An alternative approach based on low-rank approximations was proposed by \cite{KARGIN2008}.

To address the ill-posed nature of the inversion, Tikhonov regularization provides an alternative to spectral truncation (\cite{Engl2000}). Incorporating a ridge penalty on the Hilbert-Schmidt norm $\mathcal{S}_{2}$ leads to the following program:

\begin{proposition}
For a regularization parameter $\alpha>0$, the Tikhonov optimization problem:
\begin{equation}
\min_{\rho \in \mathcal{S}_{2}}\left\{ \sum_{k=1}^{n-1}\left\Vert
X_{k+1}-\rho \left( X_{k}\right) \right\Vert ^{2}+\alpha \left\Vert \rho
\right\Vert _{2}^{2}\right\}  \label{Tikho}
\end{equation}
is strictly convex and admits a unique solution given by $\widehat{\rho}_{n}=D_{n}\left( C_{n}+\alpha \mathbf{I}\right)^{-1}$.
\end{proposition}

\begin{remark}\label{rem:general_penatly}
The objective function can be generalized to $\min_{\rho \in \mathcal{S}_{2}}\left\{\sum_{k=1}^{n-1}\left\Vert X_{k+1}-\rho \left( X_{k}\right) \right\Vert^{2}+\left\Vert T\rho \right\Vert _{2}^{2}\right\} $, where $T$ is a positive-definite and coercive operator, such as a differential operator.
\end{remark}

Evaluating $\widehat{\rho}_{n}$ does not require computing the random empirical project eigenvectors of $C_{n}$. The functional observations $X_{n}$ can be projected onto a fixed, deterministic basis (e.g., Fourier, splines, or wavelets).

The parameter $\alpha$ controls the bias-variance trade-off, analogous to the truncation parameter in PCA:
\begin{itemize}
\item As $\alpha \rightarrow 0$, the estimator approaches the unregularized empirical normal equation, and $\left(C_{n}+\alpha \mathbf{I}\right)^{-1}$ becomes ill-conditioned, leading to numerical instability since its operator norm is bounded by $\|D_n\|_\infty/\alpha$.
\item Conversely, large values of $\alpha$ introduce a severe shrinkage bias, driving the norm $\|\widehat \rho_n\|_2$ toward zero.
\end{itemize}
This behavior is illustrated in the numerical studies in Section 5. Selecting an optimal $\alpha$ is therefore essential to achieve an optimal bias-variance compromise.

\subsection{Assumptions}
In functional data analysis, theoretical guarantees typically require three types of assumptions: moment conditions, the smoothness of the underlying processes, and the smoothness of the operator to be estimated. Because linear operators do not commute, regularity conditions must distinguish between left- and right-regularity.

\subsubsection{Regularity of the autocorrelation operator}

We introduce the following structural condition:
\begin{equation} \label{A2}
\left\Vert C_{\varepsilon }^{-1/2}\rho \right\Vert _{\infty }<+\infty,
\end{equation}
which characterizes the left-regularity of the statistical model. This condition is crucial for establishing the minimax lower bound. Under the assumption of Gaussian innovations, condition (\ref{A2}) is necessary and sufficient for the conditional distribution of $X_{n}$ given $X_{n-1}$ to be equivalent to the distribution of $\varepsilon_{n}$. This absolute continuity is required to evaluate statistical distances, such as the Kullback-Leibler divergence. Note that (\ref{A2}) implies that $\rho \in \mathcal{S}_{2}$.

In non-parametric functional estimation, regularity classes are commonly defined in terms of ellipsoids. For the ridge estimator, it is convenient to introduce a specific regularity class defined via the eigenvalues of $C$. Let $\tau :\mathbb{R}^{+}\rightarrow \mathbb{R}^{+}$ be a non-decreasing function. We say that $\rho \in \mathbf{S}\left(\tau\right)$ if:
\begin{equation}
\forall N\in \mathbb{N}^{\ast }:\max \left\{ \sum_{k=1}^{N}\frac{\left\Vert
\rho e_{k}\right\Vert ^{2}}{\lambda _{k}},\frac{1}{\lambda _{N}}
\sum_{k=N+1}^{+\infty }\left\Vert \rho e_{k}\right\Vert ^{2}\right\} \leq
\tau \left( N\right) .  \label{A2ter}
\end{equation}
The function $\tau$ quantifies the alignment between the operator $\rho$ and the principal components of $C$, controlling the behavior of the bias terms. Section 4.3 provides a detailed discussion of this regularity class.

\begin{remark}
The $\lambda _{k}$'s are unknown but the ratio $\left\Vert \rho
e_{k}\right\Vert ^{2}/\lambda _{k}$ may be easily interpreted in terms of
\textquotedblleft amount of information" on the $k$-th dimension of the
functional PCA. And $\left\Vert \rho e_{k}\right\Vert ^{2}$ relates in
Fourier sense to the energy of the $k$-th harmonics of $\rho $ with respect
to the eigenbasis of $C$. Finally function $\tau $ plays a role similar to
function $\varphi $ in \cite{crambes2013} p.2633 and quantifies the
regularity required to address this problem.
\end{remark}

Combining conditions (\ref{A0}), (\ref{A2}), and (\ref{A2ter}), we state the global assumption $\mathbf{A}_{1}$:
\begin{equation} \label{A1}
\mathbf{A}_{1}:\left\Vert \rho\right\Vert _{\infty}<1, \quad \left\Vert C_{\varepsilon }^{-1/2}\rho \right\Vert _{\infty
}<+\infty ,\quad \rho \in \mathbf{S}\left( \tau \right). 
\end{equation}
We denote the associated constants as follows:
\begin{align*}
\left\Vert \rho\right\Vert _{\infty}=a_{1},\quad
\left\Vert C_{\varepsilon }^{-1/2}\rho \right\Vert _{\infty }=a_{2},\quad
\mathrm{tr}C_{\varepsilon} & =\sigma_{\varepsilon}^{2}=\mathbb{E}\left\Vert
\varepsilon_{1}\right\Vert ^{2}.
\end{align*}
As shown in Proposition \ref{closed-op} of the Appendix, under $\mathbf{A}_{1}$, the operator $C_{\varepsilon}^{-1/2}C^{1/2}$ is bounded and invertible. We additionally define:
\begin{align*}
\left\Vert C^{-1/2}\rho C^{1/2} \right\Vert_{\infty} =a_{1}^{\prime}<1,\quad
\left\Vert C^{-1/2}\rho\right\Vert_{\infty} =a_{2}^{\prime},\quad
\left\Vert C_{\varepsilon}^{-1/2}C^{1/2}\right\Vert _{\infty}=a_{3}>1,\quad\left\Vert
C^{-1/2}C_{\varepsilon}^{1/2}\right\Vert _{\infty}=a_{3}^{\prime}<1.
\end{align*}

\subsubsection{Moments}

The following condition controls the higher-order moments of the principal components. For $\mathbf{s}\in\mathbb{N}^{\ast}$, let:
\begin{equation}
\mathbf{A}_{2}=\mathbf{A}_{2}(\mathbf{s}):\sup_{k}\frac{\mathbb{E}\left\langle
X_{1},e_{k}\right\rangle ^{2\mathbf{s}}}{\lambda_{k}^{\mathbf{s}}}=\sup_{k}
\frac{\mathbb{E}\left\langle X_{1},e_{k}\right\rangle ^{2\mathbf{s}}}{\left( 
\mathbb{E}\left\langle X_{1},e_{k}\right\rangle ^{2}\right) ^{\mathbf{s}}}
\leq\xi_{\mathbf{s}}.  \label{A3m}
\end{equation}
Assumption $\mathbf{A}_{2}$ ensures that the projections onto the eigenfunctions possess bounded moments up to order $2\mathbf{s}$, scaled by the corresponding eigenvalues. If $\mathbf{A}_{2}$ holds, then $\sup_{k}\frac{\mathbb{E}\left\langle X_{1},e_{k}\right\rangle ^{2p}}{\lambda_{k}^{p}}\leq\xi_{\mathbf{s}}^{p/\mathbf{s}}$ for all $p\leq\mathbf{s}$.

\begin{remark}
A stronger condition requires this property to hold for all orders. Let $\Gamma$ be a positive constant:
\begin{equation}
\mathbf{A}_{2,\infty}\left( \Gamma\right) :\forall p\geq1\sup_{k}\frac{
\mathbb{E}\left\langle X_{1},e_{k}\right\rangle ^{2p}}{\lambda_{k}^{p}}
=\xi_{p}\leq\frac{\left( 2p\right) !}{p!}\Gamma^{p-1}.  \label{A3infty}
\end{equation}
Condition $\mathbf{A}_{2,\infty}$ is satisfied by Gaussian processes (with $\Gamma=2^{-1}$) and by random elements with compact support.
\end{remark}

\subsubsection{Smoothness of the data}

Sharp control of the bias requires quantifying the decay rate of the tail series of the sequences $(\lambda_{k})_{k\in\mathbb{N}}$ and $(\|\rho e_k\|^2)_{k\in\mathbb{N}}$. We introduce a specific subspace of $l^1(\mathbb{N})$:

\begin{definition}
For $\eta > 0$, let:
\begin{equation*}
l_{\eta}^{1}=\left\{ \left( u_{k}\right) _{k\in\mathbb{N}}\in l^{1}:\left(
k^{1+\eta}u_{k}\right) _{k\in\mathbb{N}}\ \mathrm{is\ decreasing}\right\} .
\end{equation*}
\end{definition}

The smoothness assumption $\mathbf{A}_{3}$ is formulated as follows for $\eta,\xi>0$:
\begin{equation*}\label{A3}
\mathbf{A}_{3}:\left( \lambda_{k}\right) _{k\in\mathbb{N}}\in
l_{\eta}^{1},\quad\left( \left\Vert \rho e_{k}\right\Vert ^{2}\right) _{k\in 
\mathbb{N}}\in l_{\xi}^{1}.
\end{equation*}
Assumption $\mathbf{A}_{3}$ specifies a polynomial decay condition on the eigenvalues and operator coefficients, which is standard for processes such as the Wiener process (where $\lambda_k \asymp k^{-2}$).
\begin{remark} \label{rem.aprime3}
Assumption $\mathbf{A}_{3}$ could be replaced by assumption $\mathbf{A}'_{3}$ below. Pick $\eta$ and $\xi$ in $\left(  0, +\infty \right] $.

\begin{center}
$\mathbf{A}'_{3}$ : The sequence $\left( \lambda_{k}\right) _{k\in\mathbb{N}}$ (resp. $( \left\Vert \rho e_{k}\right\Vert ^{2}) _{k\in \mathbb{N}}$) is regularly varying with index $-1-\eta$ (resp $-1-\xi$).
\end{center}

The notion of regularly varying sequence was introduced in \cite{Galambos1973} and is strongly related to regularly varying functions. We refer to subsection \ref{bias-study} for references and developments about this topic and to the proof of Lemma \ref{series} in the Appendix deriving equivalent results under $\mathbf{A}_{3}$ and $\mathbf{A}'_{3}$. The case $\eta=+\infty$ (or $\xi=+\infty$) refers to rapidly varying sequences. Here it is considered as a particular case of regularly varying sequence only because it does not change the final results. In all the sequel we mention only $\mathbf{A}_{3}$ for brevity but all our result hold with $\mathbf{A}'_{3}$ instead of $\mathbf{A}_{3}$. 

\end{remark}

\section{Main results \label{section-main-results}}

A lower bound is derived first then we turn to obtaining an upper bound in
the next subsection. Set:
\begin{align*}
\mathfrak{M}_{\mathbf{s}} & = \lbrace (\rho,X,\varepsilon),\mathbf{A_{1-3}} \text{ hold} \rbrace, \\
 \mathfrak{M}_{G} & =\lbrace (\rho,X,\varepsilon),\mathbf{A_{1},\mathbf{A_{3} } \text{ hold and } (X,\varepsilon)} \text{ is Gaussian} \rbrace.
\end{align*}

\subsection{Lower bound}

For a lower bound in the framework of model (\ref{model}) two main issues appear. First it is not possible here to take advantage of the simplifications due to independence. Second the output data are of functional nature and the use of likelihood based methods is usually not possible. However the likelihood ratio of two Gaussian Hilbert-valued random elements may be evaluated by Feldman--H\'{a}jek's theorem that generalizes Cameron-Martin's theorem on a Hilbert space to any affine transformation. The latter will provide a basis to compute the Kullback-Leibler information between Gaussian measures in our setting. This is possible essentially thanks to assumption (\ref{A1}).

 \bigskip

Let $\mathbf{P}_{0}$ be the probability distribution of $\mathbb{U}=\left(
\varepsilon_{n},\varepsilon_{n-1},...,\varepsilon_{1},X_{0}\right) $ in $
\mathbf{H}^{n+1}$. In order to elaborate on the lower bound we assume now
that $\mathbb{U}$ is Gaussian centered. Since the vector $\mathbb{U}$ has
independent components its covariance operator is :
\begin{equation*}
\mathbf{\Sigma_{0}=}\mathrm{diag}\left( C_{\varepsilon},...,C_{\varepsilon
},C\right) \mathbf{\in}\left( \mathcal{S}_{1}\right) ^{n+1}.
\end{equation*}
This restriction to Gaussian distribution is not a problem as far as the
lower bound is under concern. Indeed since $\mathfrak{M}_{G}\subset\mathfrak{
M}_{\mathbf{s}}$ :
\begin{equation*}
\sup_{\mathfrak{M}_{\mathbf{s}}}\mathbb{E}\left\Vert \left( \widehat{\rho }
-\rho\right) X_{n+1}\right\Vert ^{2}\geq\sup_{\mathfrak{M}_{G}}\mathbb{E}
\left\Vert \left( \widehat{\rho}-\rho\right) X_{n+1}\right\Vert ^{2}
\end{equation*}
and a lower bound under $\mathfrak{M}_{G}$ will provide a lower bound under
more general models such as $\mathfrak{M}_{\mathbf{s}}$. Now let $\mathbf{P}_{\rho}$ be the probability distribution of $
\mathbb{X}=\left( X_{n},X_{n-1},...,X_{1},X_{0}\right) $ in the model (\ref
{model}). Here $\mathbb{X}=\mathbf{R}_{\rho}\mathbb{U}$ with :
\begin{equation*}
\mathbf{R}_{\rho}\mathbf{=}\left[ 
\begin{array}{ccccc}
\mathbf{I} & \rho & \rho^{2} & \cdots & \rho^{n} \\ 
0 & \mathbf{I} & \rho & \ddots & \vdots \\ 
\vdots & \ddots & \ddots & \ddots & \vdots \\ 
\vdots & 0 & \ddots & \mathbf{I} & \rho \\ 
0 & \cdots & \cdots & 0 & \mathbf{I}
\end{array}
\right] \in\mathcal{L}\left( \mathbf{H}^{n+1}\right)
\end{equation*}
and $\mathbf{P}_{\rho}$ is Gaussian centered with covariance operator $
\mathbf{\Sigma}_{\rho}=\mathbf{R}_{\rho}\mathbf{\Sigma_{0}R}_{\rho}^{\ast}$.

\subsubsection{Absolute continuity and Kullback-Leibler (KL) divergence}

As a preliminary step we compute the KL divergence between
the distribution of $\mathbf{P}_{\rho}$ and $\mathbf{P}_{0}$. First an
important difference appears between this dependent model and other
regression models in the i.i.d. case. When regressors are independent, the
likelihood of the Gaussian model, conditioned with respect to the input
data, is generally expressed as a shift of a Gaussian distribution. The KL divergence is quite simple to compute thanks to the
Cameron-Martin Theorem. As seen above, the dependent data encountered in our
model lead naturally to a different form for the KL divergence : $\mathbb{X}$
and $\mathbb{U}$ are both centered but differ only by their covariance
operators.

A useful tool is then the Feldman--H\'{a}jek theorem \cite
{feldman58,hajek58} which ensures that $\mathbf{P}_{\rho}$ and $\mathbf{P}
_{0}$ are either equivalent or mutually singular. For the sake of
self-containedness this crucial result was given above in the subsection \ref
{subsec-gauss-func-data} (see Theorem \ref{origFeldHayek}). The problem is that the initial version is not directly applicable in our framework: it involves the computation of the square root of $\mathbf{R}_{\rho }\mathbf{\Sigma_{0}R}_{\rho}^{\ast}$ which is not simple. In Proposition \ref{varFeldHay}, postponed to the Appendix, we derive a
variant of Feldman--H\'{a}jek theorem suited to the specific case when $
\Sigma_{1}=R_{1}\Sigma_{0}R_{1}^{\ast}$ and $\Sigma_{2}=R_{2}\Sigma
_{0}R_{2}^{\ast}$ with $R_1$ and $R_2$ two bounded linear operators. This result may be useful to explore some other dependent linear statistical models. In addition a computation of the KL
divergence is proposed in a specific case, suited to our further purposes.

The next Proposition gives an exact computation of the KL divergence in the
FAR model, then a more practical bound.

\begin{proposition}
\label{Kull-Leib-prop}The measures $\mathbf{P}_{\rho}$ and $\mathbf{P}_{0}$
are equivalent under $\mathbf A_1$. Then when $\ker C_{\varepsilon}=\left\{ 0\right\} $ the Kullback-Leibler
divergence between the distribution of $\mathbf{P}_{\rho}$ and $\mathbf{P}
_{0}$ is : 
\begin{equation}
\mathbf{KL}\left( \mathbf{P}_{\rho}||\mathbf{P}_{0}\right)
=\sum_{k=1}^{n-1}\left( n-k\right) \left\Vert
C_{\varepsilon}^{-1/2}\rho^{k}C_{\varepsilon}^{1/2}\right\Vert
_{2}^{2}+\sum_{k=1}^{n}\left\Vert
C_{\varepsilon}^{-1/2}\rho^{k}C^{1/2}\right\Vert _{2}^{2},  \label{KL-ARH}
\end{equation}
leading to the bound :
\begin{equation}
\mathbf{KL}\left( \mathbf{P}_{\rho}||\mathbf{P}_{0}\right) \leq
nb_{1}  \label{KL-bound}
\end{equation}
with $b_{1}=\frac{a_{3}^2 }{1-(a'_{1})^{2}}$.
\end{proposition}

\subsubsection{A lower bound evaluation}


\begin{theorem}
\label{LowBound}Let function $\tau $, introduced within Definition \ref{thau}, and set :
\begin{equation}
m^{\ast }=\max_{m}\left\{ \frac{1}{n}\leq \lambda _{m}^{2}\frac{\tau \left(
m\right) }{m}\right\} .  \label{alpha.star}
\end{equation}
When $m^{\ast }/\left( n\lambda _{m^{\ast }}\right) \leq 1$ a lower bound holds :
\begin{equation}
\mathfrak{E}_{n}=\min_{\widehat{\rho }}\sup_{\rho \in \mathfrak{M}_{G}}
\mathbb{E}\left\Vert \left( \widehat{\rho }-\rho \right) X_{n+1}\right\Vert
^{2}\geq  \underline{\mathfrak{c}} \frac{\lambda_1}{a_3^2} \frac{m^{\ast }}{n}  \label{LB}
\end{equation}
 with $
\underline{\mathfrak{c}}$ a universal constant.
\end{theorem}

\begin{remark}
It is plain from the definition of $\tau$ that the sequence $\left(\lambda_m^2\tau(m)/m\right)_{m\in\mathbb N^*}$ is non-increasing and tends to zero. This implies that (\ref{alpha.star}) has always a unique solution.
\end{remark}

The lower bound (\ref{LB}) above may be explicited in some situations,
typically when additional information is given about the eigenvalue sequence
and the crucial series $\sum_{p=1}^{m}\frac{\left\Vert \rho e_{p}\right\Vert
^{2}}{\lambda_{p}}$ upper bounded by $\tau\left( m\right) $. The next
Corollary collects in a table these values of $
m^{\ast }$ and $m^{\ast}/n$ for different scenarii. We also display the values of $\alpha^{\ast}=\lambda_{m^{\ast}}$ which will be important later. The proof stems from relatively standard calculations based on Riemann sums and from Proposition \ref{AB12}. This Proposition estimates function $\tau$  and also enligths the function $\Lambda_{\gamma,\delta}$ used in the Corollary.  The symbol $
\asymp$ is classical and $\sim$ below means that the order of magnitude holds up to a residual factor that is not shown but may be easily computed however. The two tables below treat two cases :
first $\lambda_{p}\asymp p^{-1-\eta}$ and second $\lambda_{p}\asymp\exp
\left( -\eta^{\prime}p\right) $ for two different shapes of $\left\Vert \rho
e_{p}\right\Vert ^{2}$ (in rows).

\begin{corollary}
\label{cor-examples}Let $\eta,\eta^{\prime},\beta,\beta^{\prime},\gamma
,\gamma^{\prime}$ be strictly positive constants. Denote $\Lambda
_{\gamma,\gamma^{\prime}}:\mathbb{R}^{+}\rightarrow\mathbb{R}^{+}$ the
function defined by $\Lambda_{\gamma,\gamma^{\prime}}\left( t\right)
=t^{\gamma}\log^{-\gamma^{\prime}}\left( 1/t\right) $ which is increasing on 
$\left] 0,1\right[ $. The tables below evaluate, up to constants, $m^{\ast}$ and $\mathfrak{E}_{n}$ defined in (\ref{alpha.star}) and (\ref{LB}) and $\alpha^{\ast}=\lambda_{m^{\ast}}$ (for further purpose) in different and classical cases.   
%

\begin{equation*}
\begin{tabular}{|c|c|c|c|}
\hline
$\lambda_{p}\asymp p^{-1-\eta}$ & \multicolumn{1}{||c|}{$\alpha^{\ast}$} & $
m^{\ast}$ & $\mathfrak{E}_{n}$ \\ \hline\hline
\multicolumn{1}{|c|}{$\left\Vert \rho e_{p}\right\Vert ^{2}\asymp
p^{-1-\beta }, \quad\beta\leq 1+\eta$} & $n^{-\frac{1+\eta}{2+\eta+\beta}}$ & $n^{\frac{1}{
2+\eta+\beta}}$ & $n^{-\frac{1+\eta+\beta}{2+\eta+\beta}}$ \\
\hline
\multicolumn{1}{|c|}{$\left\Vert \rho e_{p}\right\Vert ^{2}\asymp
p^{-1-\beta }, \quad\beta= 1+\eta$} & $\sim (n \log n)^{-\frac{1+\eta}{2+\eta}}$ & $\Lambda^{-1}_{2+\eta,1}(n)\sim (n \log n)^{\frac{1}{2+\eta}}$&$\sim n^{-\frac{1+\eta}{2+\eta}} (\log n)^{\frac{1}{2+\eta}}$  \\
\hline

\multicolumn{1}{|c|}{$\left\Vert \rho e_{p}\right\Vert ^{2}\asymp
p^{-1-\beta }, \quad\beta> 1+\eta$} & $n^{-\frac{1+\eta}{3+2\eta}}$ & $n^{\frac{1}{
3+2\eta}}$ & $n^{-\frac{2+2\eta}{3+2\eta}}$ \\ \hline
\multicolumn{1}{|c|}{$\left\Vert \rho e_{p}\right\Vert ^{2}\asymp\exp\left(
-\beta^{\prime}p\right) $} & $n^{-\frac{1+\eta}{3+2\eta}}$ & $n^{\frac {1}{
3+2\eta}}$ & $n^{-\frac{2+2\eta}{3+2\eta}}$ \\ \hline
\end{tabular}
\end{equation*}

\begin{equation*}
\begin{tabular}{|c|c|c|c|}
\hline
$\lambda_{p}\asymp\exp\left( -\eta^{\prime}p\right) $ & 
\multicolumn{1}{||c|}{$\alpha^{\ast}$} & $m^{\ast}$ & $\mathfrak{E}_{n}$ \\ 
\hline\hline
$\left\Vert \rho e_{p}\right\Vert ^{2}\asymp p^{-1-\beta}$ & $\Lambda
_{1,1+\beta}^{-1}\left( 1/n\right) \sim\frac{\log^{1+\beta}n}{n}$ & $
\log\left( 1/\alpha^{\ast}\right) $ & $\sim\frac{\log n}{n}$ \\ \hline

$\left\Vert \rho e_{p}\right\Vert ^{2}\asymp\exp\left( -\beta^{\prime
}p\right), \quad \beta^{\prime}=\eta^{\prime} $ & $\frac{1}{\sqrt{n}}$ & $\log\left(n\right)$ & $\frac{\log n}{n}$ \\ \hline

$\left\Vert \rho e_{p}\right\Vert ^{2}\asymp\exp\left( -\beta^{\prime
}p\right), \quad \beta^{\prime} \neq \eta^{\prime}  $ & $\Lambda_{\max\{1+\beta^{\prime}/\eta^{\prime},2\},1}^{-1}\left(
1/n\right) \sim\left( \frac{\log n}{n}\right) ^{\min\left\{\frac{\eta^{\prime}}{
\eta^{\prime}+\beta^{\prime}},\frac12\right\}}$ & $\log\left( 1/\alpha^{\ast}\right)  $ & $\sim \frac{\log n}{n}$ \\ \hline

\end{tabular}
\ 
\end{equation*}
\end{corollary}

\begin{remark}
The results indicate that the decay rate of the eigenvalues of $C$ is the dominant factor governing the minimax rate. Notably, the negative effects of the ill-posed inverse problem are mitigated in this prediction framework: a faster decay of the eigenvalues leads to a faster convergence rate for the prediction error.
\end{remark}

\subsection{Upper bound} \label{upper.bound.section}

In the sequel we set : $m(\alpha)=argmin_p \lbrace \vert \alpha -\lambda_p \vert \rbrace$ and set with an abuse of notation:
 $\alpha=\alpha(m)=\lambda_{m}$. This should be read the following way : once $\alpha$ is fixed, a specific dimension $m$ plays a crucial role through the previous equation. This dimension is associated to the eigenvalue closest to $\alpha$ and is similar to the truncation in Functional Principal Component Regression or to the concept of intrisic/effective dimension in the statistical literature. 
The upper bounds depend on typical bias and variance terms. With clear
notations set :
\begin{equation}
\mathbf{V}_{n}=\frac{\sigma _{\varepsilon }^{2}}{n}\sum_{p=1}^{+\infty }
\frac{\lambda _{p}}{\lambda _{p}+\alpha }\quad \mathbf{B}_{n}=\alpha
^{2}\sum_{p=1}^{+\infty }\frac{\left\Vert \rho e_{p}\right\Vert ^{2}}{
\lambda _{p}+\alpha }  \label{biais-var}
\end{equation}
where $\mathbf{V}_{n}$ accounts for variance and $\mathbf{B}_{n}$ for bias.
The next Proposition provides a kind of non asymptotic bias-variance
decomposition for $\mathcal{M}_{n}$.


The first result is a general upper bound with explicit constants.

\begin{theorem}
\label{risk-upper-bound}When $m/\left( n\lambda_{m}\right) \leq1$ for all $n$ and $m \geq 4$
we have :
\begin{equation}
\sup_{\left( \rho,\varepsilon,X_{0}\right) \in\mathfrak{M}_{\mathbf{s}}}
\mathbb{E}\left\Vert \left( \widehat{\rho}_{n}-\rho\right) X_{n+1} \right\Vert ^{2} \leq (8+4(a'_{3})^{2}) \mathbf{V}_{n}+12 \cdot\mathbf{B}_{n}+\mathfrak{R}_{n}. \label{risk-main}
\end{equation}

The residual term is $\mathfrak{R}_{n}$ :
\begin{align*}
0 & \leq\mathfrak{R}_{n}\leq\frac{\mathcal{R}'_{n}}{n}+\mathcal{R}_{\mathbf s}^{\prime\prime
}\left( \mathbf{V}_{n}+\mathbf{B}_{n}\right) m\left( \frac{m^{2}}{n}\right)
^{\left( \mathbf{s}-2\right) /3}
\end{align*}

where $\mathcal{R}'_n$ and $\mathcal{R}''_{\mathbf s}$ are bounded by explicit constants.
\end{theorem}

\begin{remark}
The technical assumption $m/\left( n\lambda_{m}\right) \leq1$ could be
removed and replaced with a looser $m/\left( n\lambda_{m}\right) \leq c$. It
avoids introducing the additional constant $c.$ In view of the next result
about selection of the optimal parameters $\alpha$ and $m$ (see below) this
constraint always holds.
\end{remark}

The term $\mathfrak{R}_{n}$ my be viewed as residual in the asymptotic sense
: if the two first series on the right hand side of (\ref{risk-main}) are
small then $\mathfrak{R}_{n}$ is substantially smaller than both of them.
Obviously for
large $n$ and for $\mathbf{s}>2$, optimizing the mean
square prediction risk comes down to ignoring $\mathfrak{R}_{n}$ and
focusing on the two first series.

\begin{theorem}
\label{minimaxUpperBound}Assume that $\mathbf{s}>2$ in assumption $\mathbf{A}
_{2}$ and $\eta \geq 3/\left( \mathbf s-2\right) $ where $\eta $ was
defined within $\mathbf{A}_{3}$. An optimal choice of $\alpha $ that
minimizes the bias-variance within (\ref{risk-main}) is defined through :
\begin{equation}
m^{\ast }=\max_{m}\left\{ \frac{1}{n}\leq \frac{\lambda _{m}^{2}}{m}\tau
\left( m\right) \right\} ,\alpha ^{\ast }=\lambda _{m^{\ast }}  \label{m_opt}
\end{equation}
then for $\widehat{\rho }_{n}^{\ast }=\widehat{\rho }_{n}\left( \alpha
^{\ast }\right) $ :
\begin{equation*}
\sup_{\left( \rho ,\varepsilon ,X_{0}\right) \in \mathfrak{M}_{\mathbf{s}}}
\mathbb{E}\left\Vert \left( \widehat{\rho }_{n}^{\ast }-\rho \right)
X_{n+1}\right\Vert ^{2}\leq \overline{\mathfrak{c}}\frac{m^{\ast }}{n}
\end{equation*}
where $\overline{\mathfrak{c}}$ is a constant that depends on $\mathcal{R}_{n}^{\prime},\, \mathcal{R}_{s}^{\prime \prime
}, \, \sigma _{\varepsilon }^{2}$ and $\eta $. \end{theorem}

\begin{remark}\label{rem:minimax}
Theorems~\ref{minimaxUpperBound} and Theorem~\ref{LowBound} establish that the ridge estimator $\widehat{\rho }_{n}^{\ast }$ is minimax optimal, up to multiplicative constants, matching the rates derived in the literature for functional linear regression models. Note that
 \cite{Hall2007,Benatia2017} obtain convergence rates for a ridge estimate very close to the one given just above. Their models and assumptions are different from the present ones, the framework is asymptotic but some comparisons are possible. Optimality is derived in \cite{Hall2007} but with restricted assumptions especially on the eigenvalue sequence. These authors consider the MSE for $\widehat{\rho }_{n}$ not the prediction error which is smoother and may lead to faster rates.
\end{remark}

\subsection{The bias of the ridge estimator in functional models - connection with regularly
and exponentially varying functions\label{bias-study}}

To the best of our knowledge, bounding the bias of ridge estimators in functional models has previously been achieved only under the assumption of polynomial decay for both the eigenvalues of the covariance operator and the sequence $(\|\rho e_p\|)_{p\geq 1}$ (see, e.g., \cite{Hall2007,PP18}). However, to derive more general results, a natural connection with the theory of regularly and exponentially varying functions emerges. This section aims to elucidate these connections and derive the corresponding convergence rates. As introduced in the preceding section, the principal bias component is given by:
\begin{equation*}
\mathbf{B}_{n}=\alpha ^{2}\left[ \sum_{p=1}^{m}
\frac{\left\Vert \rho e_{p}\right\Vert ^{2}}{\lambda _{p}+\alpha }
+\sum_{p=m+1}^{+\infty }\frac{\left\Vert \rho e_{p}\right\Vert ^{2}}{\lambda
_{p}+\alpha }\right]. 
\end{equation*}
As stated in Section \ref{upper.bound.section}, setting $\alpha=\lambda_m$ yields:
\begin{eqnarray*}
\frac{\alpha^2}{2} (\mathbf{b}_{1}+\mathbf{b}_{2}) \leq \mathbf{B}_{n}\leq \alpha^2 (\mathbf{b}_{1}+\mathbf{b}_{2}) \\
\mathbf{b}_{1}\left( n,m\right)=\sum_{p=1}^{m}\frac{\left\Vert \rho
e_{p}\right\Vert ^{2}}{\lambda _{p}}, \quad \mathbf{b}_{2}\left( n,m\right) =\frac{1}{\lambda _{m}}
\sum_{p=m+1}^{+\infty }\left\Vert \rho e_{p}\right\Vert ^{2}.
\end{eqnarray*}

Because the upper bound is expressed in terms of $\max \left( \mathbf{b}
_{1}\left( n,m\right) ,\mathbf{b}_{2}\left( n,m\right) \right) $ via $\tau$ in (\ref{m_opt}), it is necessary to determine which term dominates. The subsequent proposition identifies the dominant term—either $\mathbf{b}_{1}\left( n,m\right) $ or $\mathbf{b}_{2}\left( n,m\right) $—in the bias decomposition of the ridge estimator. We demonstrate that the theory of regularly varying functions and their extensions provides a comprehensive framework to resolve this question. Note that this proposition holds beyond the FAR model and can be applied to general function-on-function regression models. 

Hereafter, the discrete sequences are continuously interpolated into functions. This framework is analytically more tractable and permits the use of calculus techniques that simplify the proof of the forthcoming proposition. This continuous embedding is formally justified by Corollary p.115 in \cite{Galambos1973} and Proposition 10.3 in \cite{Granata2016b}.

\begin{definition}
\label{thau}Let $\varrho $ (respectively, $\lambda $) denote a $
C^{1}\left( \left[ 1,+\infty \right) ,\mathbb{R}^{+}\right) $ function
interpolating the sequence $\left\Vert \rho e_{p}\right\Vert ^{2}$ such that $\varrho
\left( p\right) =\left\Vert \rho e_{p}\right\Vert ^{2}$ (respectively, interpolating $\lambda_{p}$ such that $\lambda \left( p\right) =\lambda _{p}$). Define:
\begin{equation}
g\left( x\right) =\int_{1}^{x}\frac{\varrho \left( s\right) }{\lambda \left(
s\right) }ds.  \label{A2bis}
\end{equation}
\end{definition}

\begin{remark}
The above definition yields $g^{\prime }\left( x\right) =\varrho \left(
x\right) /\lambda \left( x\right) $, indicating that $g$ is strictly increasing. Furthermore, both integrals $
\int_{1}^{+\infty }\lambda \left( s\right) ds$ and $\int_{1}^{+\infty
}\varrho \left( s\right) ds$ are convergent.
\end{remark}

Several classes of real functions, prominently featured in extreme value theory, are instrumental here and will be linked to the functions $\varrho$ and $\lambda$. A brief overview of regularly varying functions is therefore warranted. A comprehensive reference on this subject is \cite{Bingham1987}. Let $\mathbf{RV}_{\alpha}$ denote the class of regularly varying functions with index $\alpha$, and let $\mathbf{RV}_{+}$
(respectively, $\mathbf{RV}_{-}$, $\mathbf{RV}_{0}$) represent the classes of functions that are regularly varying at infinity with strictly positive (respectively, strictly negative, zero) indices. The notation $\mathbf{RV}_{\alpha}$ is standard in extreme value theory; it is retained in this subsection despite the prior use of the parameter $\alpha$ elsewhere in the paper. The class $\mathbf{RV}_{0}$ corresponds to slowly varying functions at infinity. Examples of functions within $\mathbf{RV}_{\alpha}$ are provided below. Finally, let $\mathbf{EV}_{c}$ denote the class of functions with exponential variation at infinity with index $c$, $\mathbf{EV}_{0}$ the class of functions with hypoexponential variation, and $\mathbf{RV}_{\infty }$ the class of rapidly varying functions. 

The families $\mathbf{EV}_{c}$ and $\mathbf{EV}_{0}$ are both contained within $\mathbf{RV}_{\infty }$, but they are disjoint from $\mathbf{RV}_{\alpha }$. Furthermore, the nesting $\mathbf{RV}_{0}\subset \mathbf{RV}_{\alpha
}\subset \mathbf{EV}_{0}\subset \mathbf{EV}_{c}$ holds. These classes are rigorously defined and analyzed in \cite{Granata2016a} and \cite{Granata2016b}. To clarify these functional classes, we provide explicit examples:

\begin{itemize}
\item $\mathbf{RV}_{0}$ contains constant functions and functions possessing a finite, non-zero limit at infinity. Additionally:
\begin{equation*}
\left\{ \log ^{q_{1}}x,\log ^{q_{2}}(\log x),\exp \left( \log
^{q_{3}}x\right) \right\} \subset \mathbf{RV}_{0}
\end{equation*}
for any $q_{1}, q_{2} \in \mathbb{R}$ and $0<q_{3}<1$. 

\item $\mathbf{RV}_{\alpha }$ consists precisely of functions of the form $
x^{\alpha } \ell(x)$, where $\ell \in \mathbf{RV}_{0}$ and $
\alpha \in \mathbb{R}^{\ast }$. Furthermore, $\mathbf{RV}_{+}=\bigcup_{\alpha >0} \mathbf{RV}
_{\alpha }$ and $\mathbf{RV}_{-}=\bigcup _{\alpha <0}\mathbf{RV}_{\alpha }$. 

\item $\mathbf{EV}_{0}$ contains $\mathbf{RV}_{\alpha }$ as well as functions of the form:
\begin{equation*}
\left\{ \exp \left( \log ^{q_{4}}x\right) ,\exp \left( x^{q_{5}}\right) \ : \
q_{4}>1, \ 0<q_{5}<1\right\}
\end{equation*}
along with their products with functions from $\mathbf{RV}_{\alpha }$.

\item $\mathbf{EV}_{c}$ consists precisely of functions of the form $h\left(
x\right) \exp \left( cx\right) $, where $h\in \mathbf{EV}_{0}$ and $c\neq 0$.
\end{itemize}

While other families of rapidly varying functions exist—such as exponentially hypervarying functions, which include $\exp \left(
x^{q_{6}}\right) $ for $q_{6}>1$ (see \cite{Granata2016b})—we restrict our analysis to the aforementioned classes, as they sufficiently cover a broad spectrum of practical scenarios.

\begin{remark}
Consider the function:
\begin{equation*}
x^{-1-\alpha }\exp \left( -c_0 x\right) \log ^{q_{1}}x\log ^{q_{2}}(\log x)\exp
\left( \log ^{q_{3}}x\right) \exp \left(-c_4 \log ^{q_{4}}x\right) \exp \left(-c_5 x^{q_{5}}\right)
\end{equation*}
This function belongs to $\mathbf{EV}_{c}$ if the constant $c_0$ is non-zero; it belongs to $\mathbf{EV}_{0}$ if $c_0=0$ and either $c_4$ or $c_5$ is non-zero (with $q_{4}>1$ and/or $0<q_{5}<1$). It reduces to a function in $\mathbf{RV}_{\alpha }$ when $c_0=c_4=c_5=0$. 
\end{remark}

To proceed, we classify $\varrho $ and $\lambda $ into three strictly partitioned cases: 
$\lambda \in \mathbf{EV}_{c}$ (i.e., $\lambda \in \mathbf{EV}_{c}\setminus \mathbf{EV}_{0}$), $\lambda
\in \mathbf{EV}_{0}$ (i.e., $\lambda \in \mathbf{EV}_{0}\setminus \mathbf{RV}_{\alpha }$), and $\lambda
\in \mathbf{RV}_{\alpha }$. Note that neither $\lambda $ nor $\varrho $ can belong
to $\mathbf{RV}_{0}$ because they must be integrable, which implies that $\alpha < -1$. Table \ref{table:gprime} catalogs the nine possible configurations for the ratio $g^{\prime }=\rho /\lambda$:

\begin{table}[h]
\centering
\begin{tabular}{c|c|c|c}
 $g^{\prime }=\rho /\lambda $ & $\lambda \in \mathbf{EV}_{c}$ & $\lambda \in \mathbf{EV}_{0}$ & $\lambda \in \mathbf{RV}_{\alpha }$ \\ 
\hline $\rho \in \mathbf{EV}_{c}$ & $a.1$ & $a.2$ & $a.3$ \\ 
\hline  $\rho \in \mathbf{EV}_{0}$ & $b.1$ & $b.2$ & $b.3$ \\ 
\hline  $\rho \in \mathbf{RV}_{\alpha }$ & $c.1$ & $c.2$ & $c.3$ \\
\end{tabular}
\caption{\label{table:gprime}}
\end{table}

These nine configurations are summarized in Proposition \ref{prop.reg-var.bias}. For each case, the corresponding class of $g^{\prime }$ and the dominant bias component are identified. We define the following integral criteria:
\begin{equation*}
\mathcal{B}_{1}:\frac{\int_{x}^{+\infty }\varrho \left( s\right) ds}{\lambda
\left( x\right) }\leq M\int_{1}^{x}\frac{\varrho \left( s\right) }{\lambda
\left( s\right) }ds\quad \mathrm{and}\quad \mathcal{B}_{2}:\int_{1}^{x}\frac{
\varrho \left( s\right) }{\lambda \left( s\right) }ds\leq M^{\prime }\frac{
\int_{x}^{+\infty }\varrho \left( s\right) ds}{\lambda \left( x\right) }.
\end{equation*}

Condition $\mathcal{B}_{1}$ implies that $\mathbf{b}_{1}$ is the dominant term. The notation $\mathbf{b}_{1} \cap \mathbf{b}_{2}$ indicates that both $\mathcal{B}_{1}$ and $\mathcal{B}_{2}$ are satisfied; the constants $M$ and $M^{\prime }$ may differ, but the asymptotic rates coincide up to a multiplicative constant. In case $b.2$, the notation $\mathbf{b}_{1} \cup \mathbf{b}_{2}$ indicates that neither term systematically dominates. Depending on the specific choices of $\varrho $ and $
\lambda $, either $\mathcal{B}_{1}$ or $\mathcal{B}_{2}$ may hold exclusively. Examples illustrating these behaviors are provided in the proofs.

\begin{proposition} The nine configurations detailed in Table \ref{table:gprime} are directly linked to assumption $\mathbf{A}'_{3}$ introduced in Remark \ref{rem.aprime3}. This classification determines the functional class of $g'=\rho / \lambda$ and identifies the dominant term in the bias decomposition, as summarized in the following table:
\label{prop.reg-var.bias} 
\begin{equation*}
\begin{tabular}{c|c|c|c|c|c|c|c|c|c}
Case & $a.1$ & $a.2$ & $a.3$ & $b.1$ & $b.2$ & $b.3$ & $c.1$ & $c.2$ & $c.3$
\\ 
\hline Class for $\rho /\lambda $ & $\mathbf{EV}_{c}\cup \mathbf{EV}_{0}$ & $\mathbf{EV}_{c}$ & $\mathbf{EV}_{c}$ & $
\mathbf{EV}_{c}$ & $\mathbf{EV}_{0}\cup \mathbf{RV}_{\alpha }$ & $\mathbf{EV}_{0}$ & $\mathbf{EV}_{c}$ & $\mathbf{EV}_{0}$ & $
\mathbf{RV}_{\alpha }$ \\ 
\hline Dominant term & $\mathbf{b}_{1}$ & $\mathbf{b}_{1}\cap \mathbf{b}_{2}$ & $
\mathbf{b}_{1}\cap \mathbf{b}_{2}$ & $\mathbf{b}_{2}$ & $\mathbf{b}
_{1}\cup \mathbf{b}_{2}$ & $\mathbf{b}_{1}\cap \mathbf{b}_{2}$ & $
\mathbf{b}_{2}$ & $\mathbf{b}_{2}$ & $\mathbf{b}_{1}\cap \mathbf{b}_{2}$
\end{tabular}
\end{equation*}
\end{proposition}

\section{Simulations}\label{sec:simus}
\subsection{Simulation method}

We first simulate 
$$
X_0(t) = \sum_{j=1}^J \sqrt{\lambda_j}\eta_je_j(t), 
$$
with $\psi_j(t) = \sqrt2\sin(2\pi(j-0.5)t)$, $(\eta_1,\hdots,\eta_J)\sim_{i.i.d.}\mathcal N(0,1)$, $(\lambda_j)_{j\geq 1}$ a summable sequence of positive real numbers. $J$ is chosen sufficiently large so that $X_0$ is approximately equal in distribution to a centered Gaussian process with covariance kernel $K(s,t)=\sum_{j\geq 1}\lambda_je_j(s)e_j(t)$. In the following we set $J=1100$.  

The process $\rho$ is chosen as follows 
\[
\rho : f \in\mathbf H\mapsto \sum_{j=1}^J (-1)^j \rho_j \langle f,e_j\rangle e_j, 
\]
with $(\rho_j)_{j\geq 1}$ a square summable sequence of real numbers in the interval $[0,1)$. Remark that, the construction of $\rho$  verifies, for $j\leq J$, $\|\rho e_j\| = \rho_j$. 

Finally, the noise process is generated as follows
\[
\varepsilon_n = \sum_{j=1}^J\sqrt{\mu_j}\xi_j^{(n)}e_j, \text{with } \mu_j= \lambda_j(1-\rho_j^2), 
\]
and $(\xi_j^{(n)})_{j=1,\hdots,J; n\geq 1}\sim_{i.i.d}\mathcal U(-\sqrt{3},\sqrt{3})$. 
This choice ensures that \eqref{cov} is satisfied and that the process $(X_n)_{n\in\mathbb N}$ is stationary. 

For illustrative purposes, examples of trajectories of the resulting process $(X_n)_{n\geq 1}$ are plotted in Figure~\ref{fig:process_trajectory} for different choices of the sequences $(\lambda_j)_{j=1,\hdots,J}$ and $(\rho_j)_{j=1,\hdots,J}$.

\begin{figure}
\begin{tabular}{c|c}
\multicolumn{2}{c}{$\lambda_j=j^{-2}$, $\rho_j=\rho_0j^{-3/2}$}\\
\includegraphics[width=0.45\textwidth]{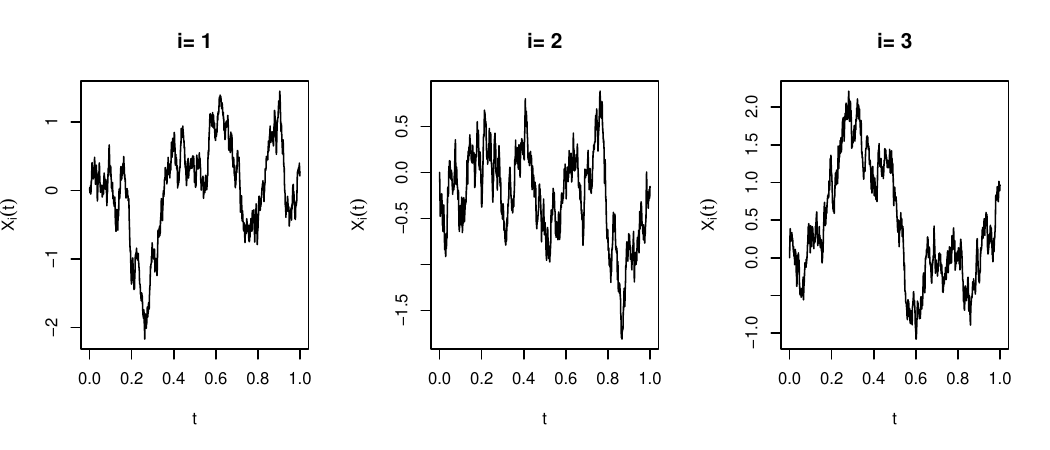}&\includegraphics[width=0.45\textwidth]{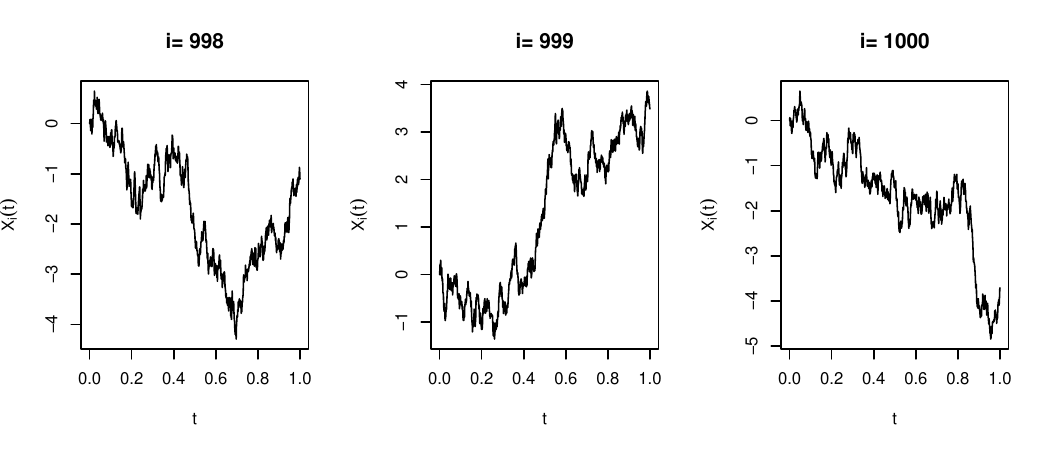}\\
\multicolumn{2}{c}{$\lambda_j=j^{-2}$, $\rho_j=\rho_0e^{-(j-1)/2}$}\\
\includegraphics[width=0.45\textwidth]{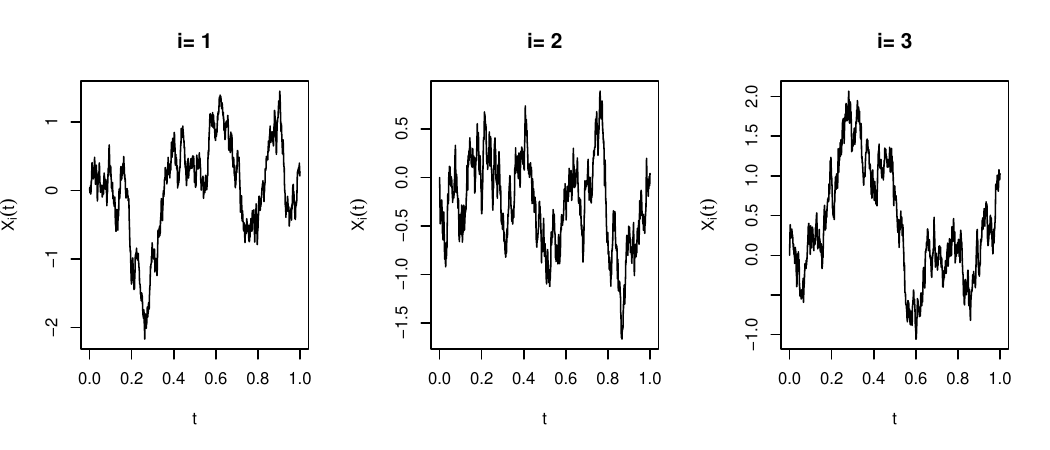}&\includegraphics[width=0.45\textwidth]{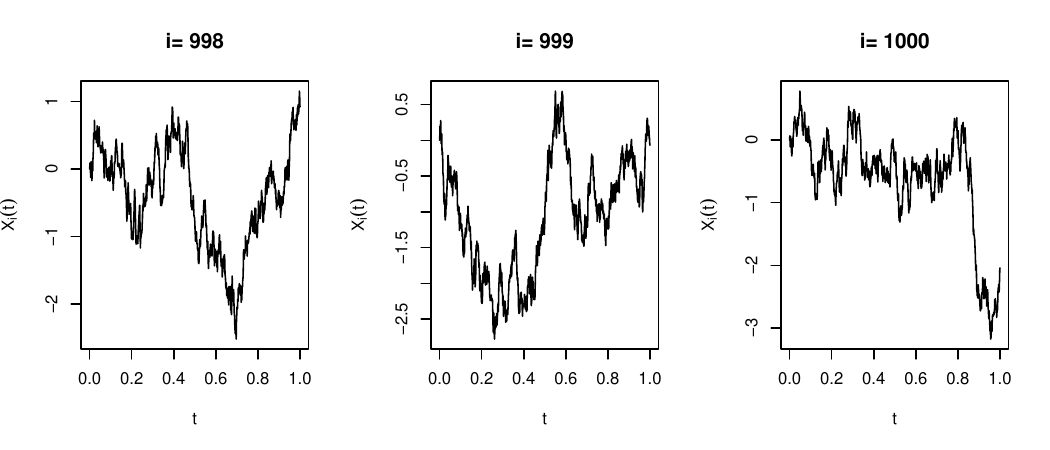}\\
\multicolumn{2}{c}{$\lambda_j=e^{-j}$, $\rho_j=\rho_0j^{-3/2}$}\\
\includegraphics[width=0.45\textwidth]{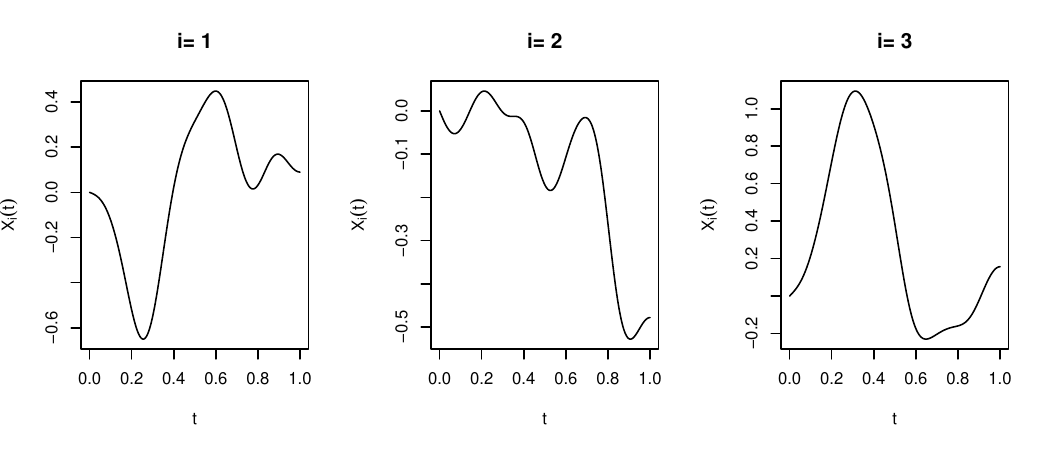}&\includegraphics[width=0.45\textwidth]{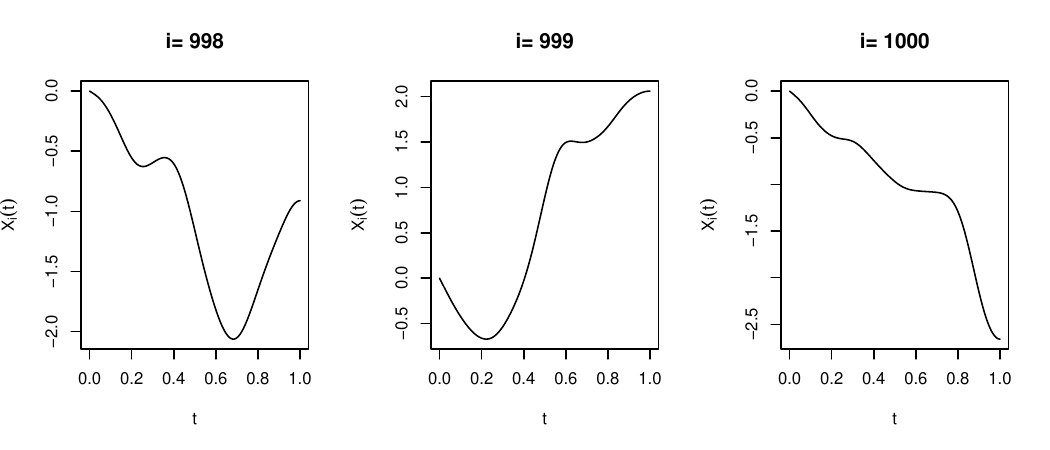}\\
\multicolumn{2}{c}{$\lambda_j=e^{-j}$, $\rho_j=\rho_0e^{-(j-1)/2}$}\\
\includegraphics[width=0.45\textwidth]{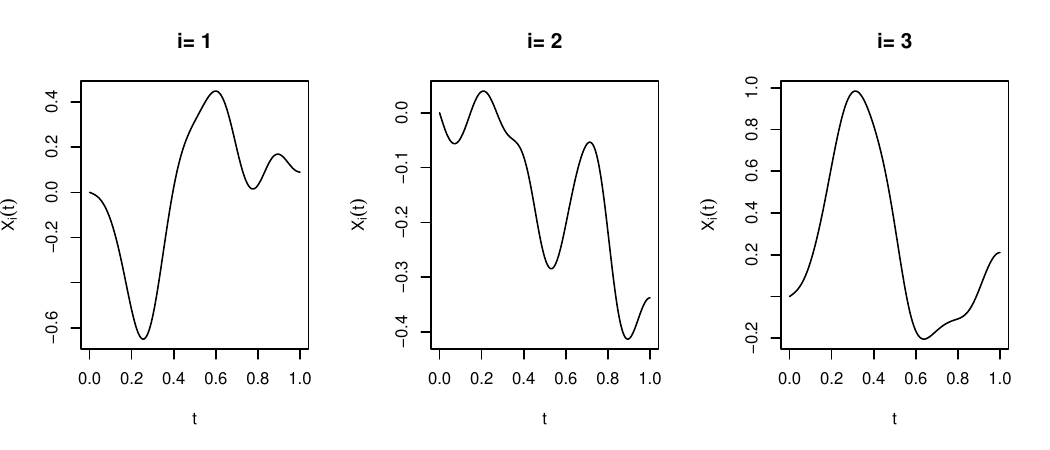}&\includegraphics[width=0.45\textwidth]{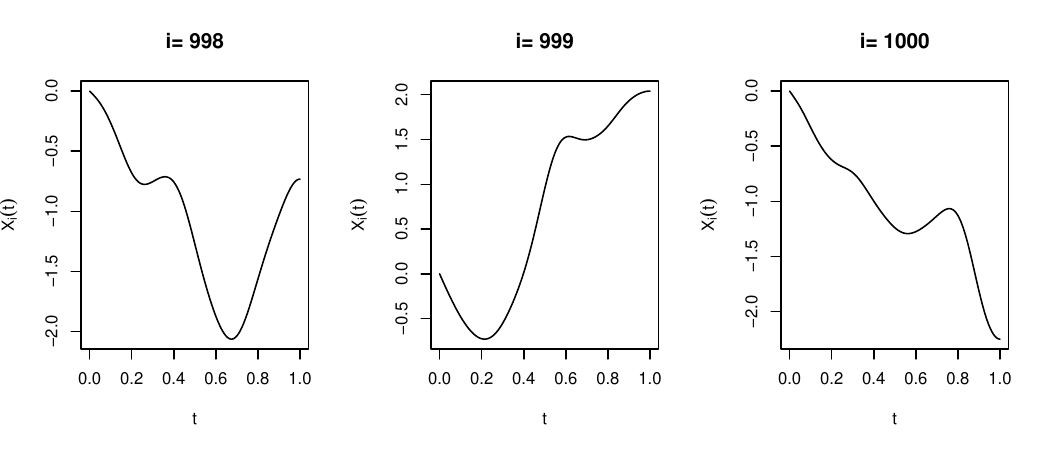}\\
\end{tabular}
\caption{\label{fig:process_trajectory} Example of trajectories of the process for different values of $(\lambda_j)_{j\geq 1}$ and $(\rho_j)_{j\geq 1}$ (here we chose $\rho_0=0.99$ so that $|\rho_j|<1$ for all $j\geq 1$. The random seed has been fixed so that the sequences $(\eta_1,\hdots,\eta_J)$ and $(\xi_j^{(\ell)})_{j=1,\hdots,J;  \ell=1,\hdots,n}$ are the same for all processes.}
\end{figure}

\subsection{Estimation method}

Numerically it is impossible to calculate explicitly the inverse of the operator $C_n+\alpha I$. To overcome this difficulty, a usual approach consists in choosing an orthonormal basis $(\varphi_\ell)_{\ell\geq 1}$ and $L$ a sufficiently large integer. Let $\mathbf C_n=\left(\langle C_n \varphi_\ell,\varphi_{\ell'}\rangle\right)_{\ell,\ell'=1,\hdots,L}$ (resp. $\mathbf D_n=\left(\langle D_n \varphi_\ell,\varphi_{\ell'}\rangle\right)_{\ell,\ell'=1,\hdots,L}=(d_{\ell,\ell'})_{\ell,\ell'=1,\hdots,L}$) the matrices of $C_n$ (resp. $D_n$) in the basis $(\varphi_\ell)_{\ell=1,\hdots,L}$. Remark that, for a function $f\in\mathbf H $, 
\begin{eqnarray*}
\widehat\rho f&=& D_n(C_n+\alpha I)^{-1}f =  D_n(C_n+\alpha I)^{-1}\left(\sum_{\ell\geq 1}f_\ell \varphi_\ell\right) = \sum_{\ell\geq 1}f_\ell D_n(C_n+\alpha I)^{-1}\varphi_\ell  \\
&=& \sum_{\ell,\ell'\geq 1}f_\ell \langle D_n(C_n+\alpha I)^{-1}\varphi_\ell,\varphi_{\ell'}\rangle \varphi_{\ell'} = \sum_{\ell,\ell',\ell''\geq 1}f_\ell \langle (C_n+\alpha I)^{-1}\varphi_\ell,\varphi_{\ell''}\rangle \langle D_n\varphi_{\ell''},\varphi_{\ell'}\rangle \varphi_{\ell'} \\
&\approx& \sum_{\ell,\ell',\ell''=1}^L f_\ell \left[\left(\mathbf C_n+\alpha I_L\right)^{-1}\right]_{\ell,\ell''} d_{\ell'',\ell'}  \varphi_{\ell'} = \sum_{\ell'=1}^L\left[\mathbf f\left(\mathbf C_n+\alpha I_L\right)^{-1}\mathbf D_n\right]_{\ell'}\varphi_{\ell'},
\end{eqnarray*}
where we denote by $f_\ell=\langle f,\varphi_\ell\rangle$, and for a matrix $A$, its coefficients by $[A]_{\ell,\ell'}$ and $I_L$ the identity matrix if size $L$.

We choose $\varphi_\ell(t) = \sqrt{p}\mathbf 1_{[t_{\ell-1},t_{\ell})}(t)$ with $t_\ell = \ell/p$, $\ell=1,\hdots,p$ with $p=1000$. A similar procedure could be conducted with other approximation basis such as Fourier, splines, wavelets,... However, this choice of basis minimizes the bias induced by the projection step.

In accordance with the theoretical results, a first attempt was made with $\alpha^*$ given by Equation~\eqref{m_opt}. However, this value of $\alpha^*$ was too large and induced a biased estimator as can be seen in Figure~\ref{fig:predict}. In accordance with the results of Figure~\ref{fig:predict}, we finally choose $\alpha^*/10$ which leads to better results. 
 \begin{figure}
\begin{tabular}{|c|c|c|}
\hline
 						& $\lambda_j=j^{-2}$ & $\lambda_j=e^{-j}$\\
						\hline
 $\rho_j=\rho_0j^{-3/2}$		&&\\
 & \includegraphics[width=0.4\textwidth]{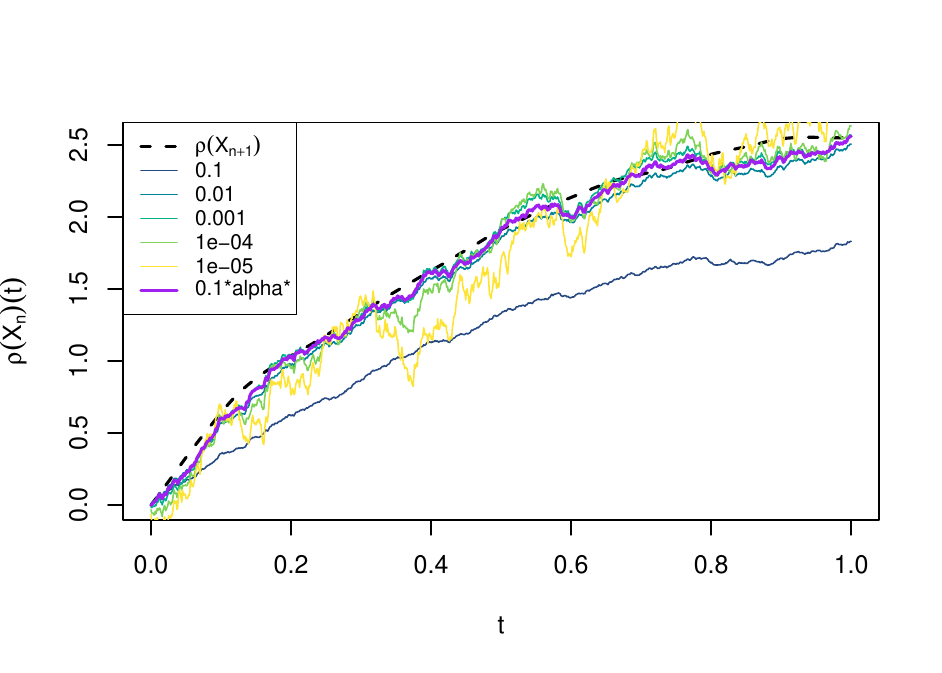}&\includegraphics[width=0.4\textwidth]{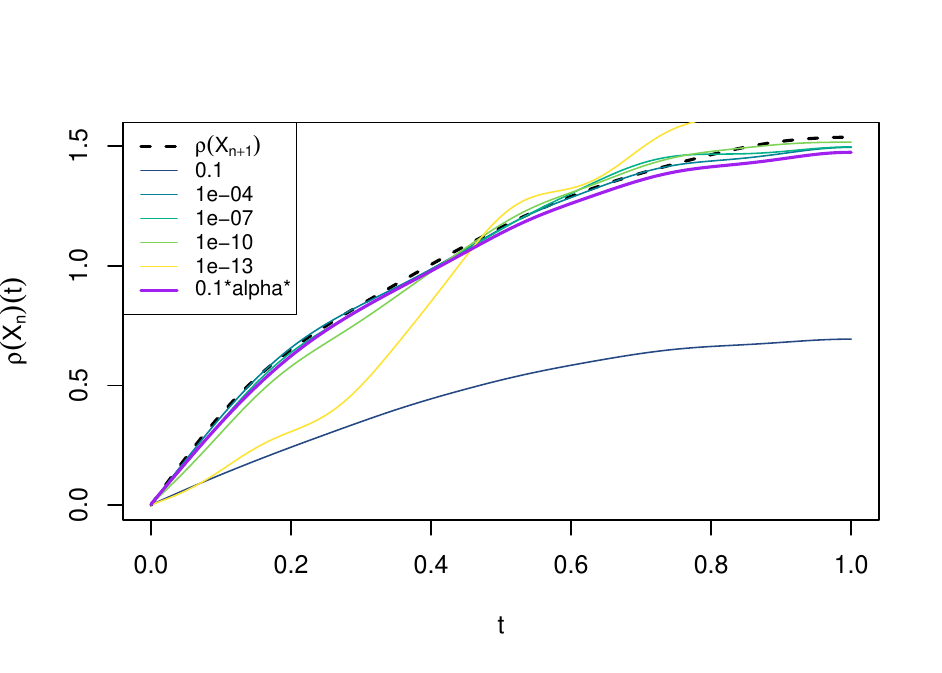}\\
 \hline
 $\rho_j=\rho_0e^{-(j-1)/2}$.       &&\\
 & \includegraphics[width=0.4\textwidth]{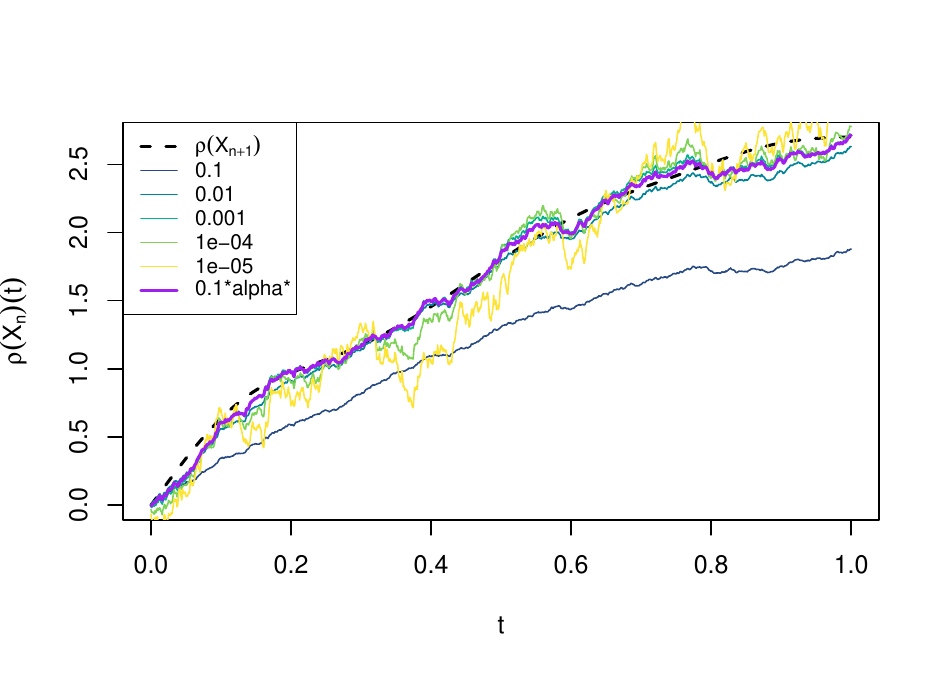}&\includegraphics[width=0.4\textwidth]{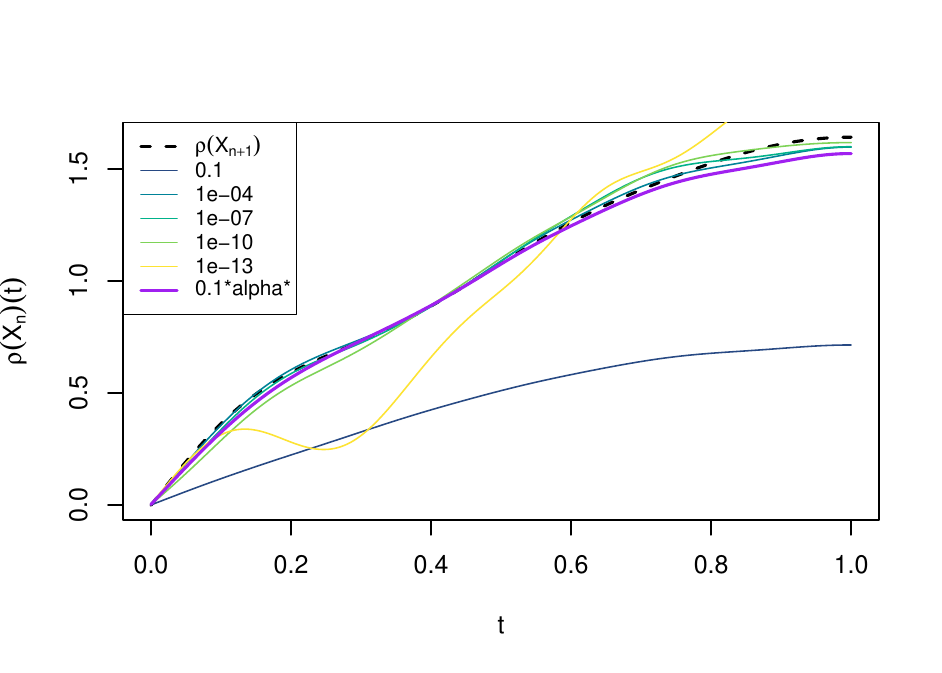}\\
 \hline
\end{tabular}
\caption{\label{fig:predict} Prediction of $\rho(X_{n})$ for different values of $\alpha$ and for $\alpha^*/10$.  }
\end{figure}

Now we estimate the risk $\mathbb{E}\left\Vert \left( \widehat{\rho}_{n}-\rho\right) X_{n+1} \right\Vert ^{2}$ with a Monte Carlo study. We generate independently $N=50$ samples $\{X_1^{(j)},\hdots,X_{n+1}^{(j)}\}$, $j=1,\hdots,N$ and compute, for different values of $n$, the quantity 
\[
\widehat {\mathcal R_n} = \frac1N\sum_{j=1}^N\left\|\left( \widehat{\rho}_{n}^{(j)}-\rho\right) X_{n}^{(j)} \right\|^2, 
\]
where $ \widehat{\rho}_{n}^{(j)}$ is the estimator calculated from $\{X_1^{(j)},\hdots,X_{n}^{(j)}\}$. The resulting risks are plotted in Figures~\ref{fig:risk} and \ref{fig:logrisk}. Remark that the rates obtained are coherent with the theoretical rates proven in Corollary~\ref{cor-examples}. 
\begin{figure}
\begin{tabular}{|c|c|c|}
\hline
 						& $\lambda_j=j^{-2}$ & $\lambda_j=e^{-j}$\\
						\hline
 $\rho_j=\rho_0j^{-3/2}$		&&\\
 & \includegraphics[width=0.4\textwidth]{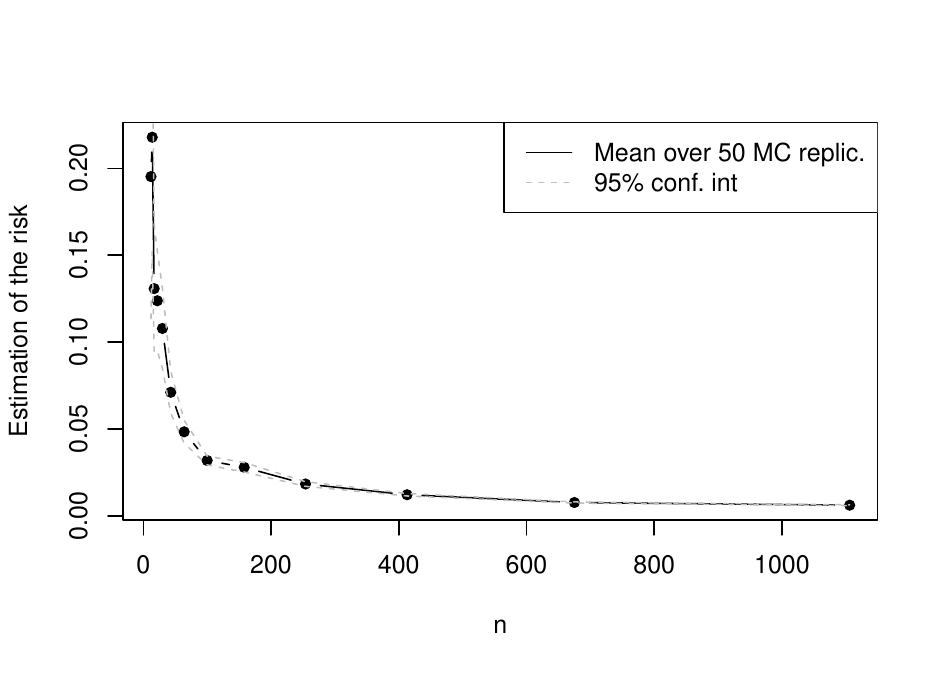}&\includegraphics[width=0.4\textwidth]{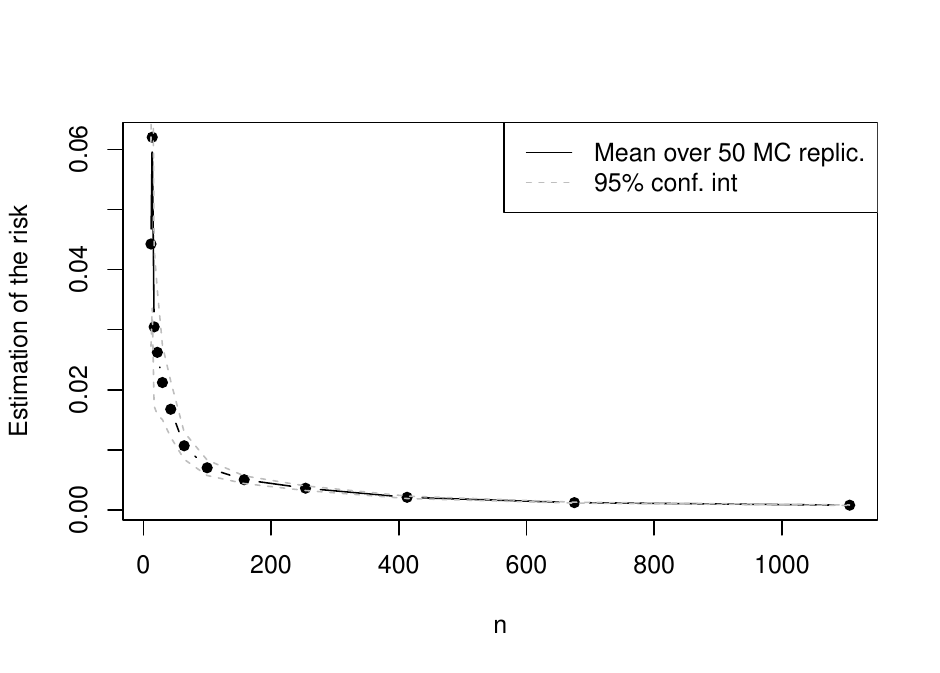}\\
 \hline
 $\rho_j=\rho_0e^{-(j-1)/2}$.       &&\\
 & \includegraphics[width=0.4\textwidth]{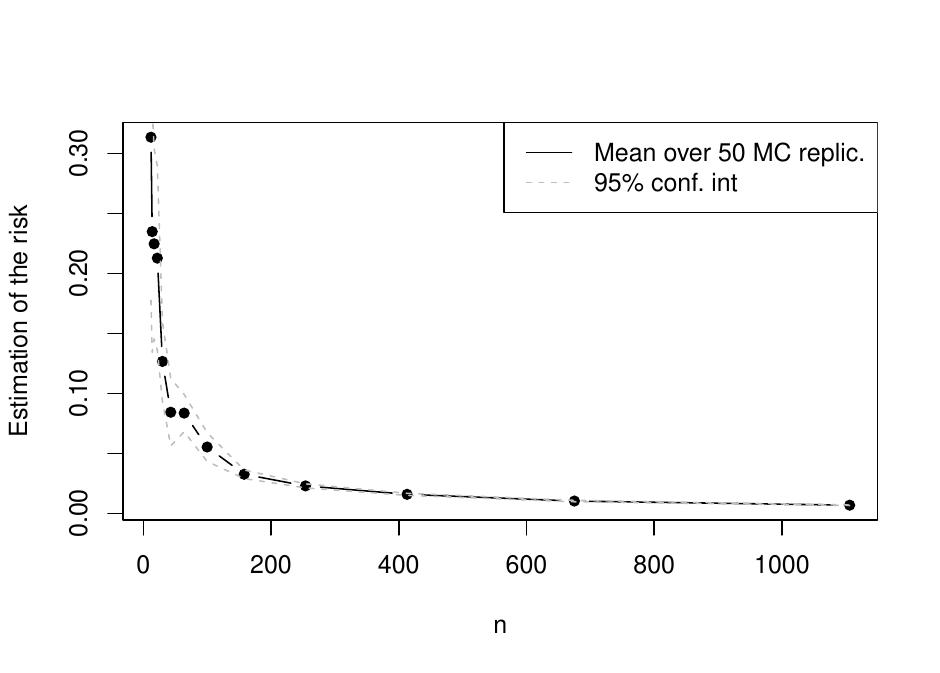}&\includegraphics[width=0.4\textwidth]{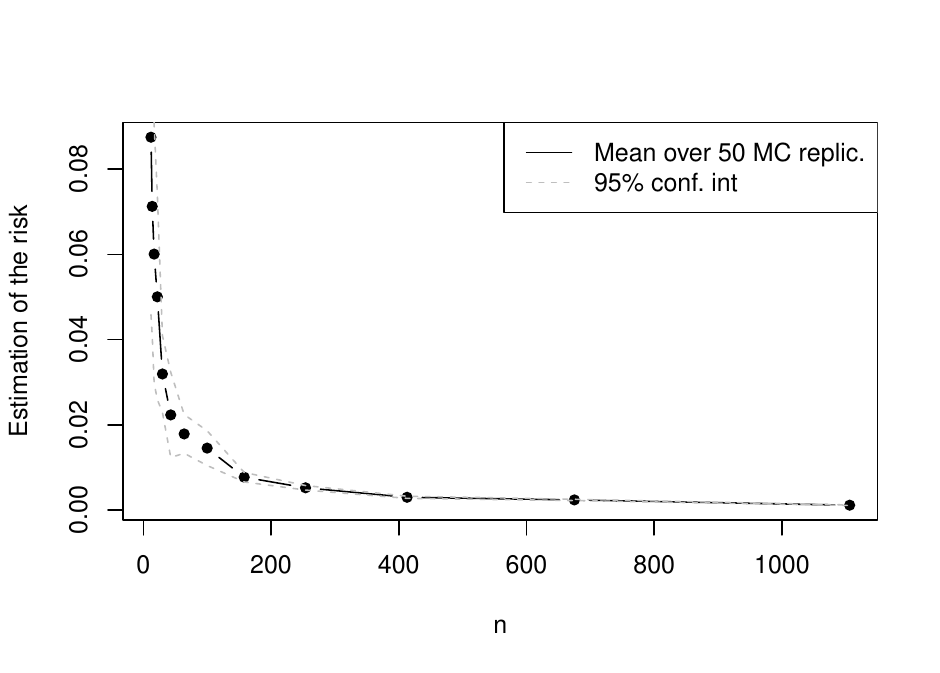}\\
 \hline
\end{tabular}
\caption{\label{fig:risk} Plot of Monte-Carlo estimation of the risk for different values of $n$.  }
\end{figure}

\begin{figure}
\begin{tabular}{|c|c|c|}
\hline
 						& $\lambda_j=j^{-2}$ & $\lambda_j=e^{-j}$\\
						\hline
 $\rho_j=\rho_0j^{-3/2}$		&&\\
 & \includegraphics[width=0.4\textwidth]{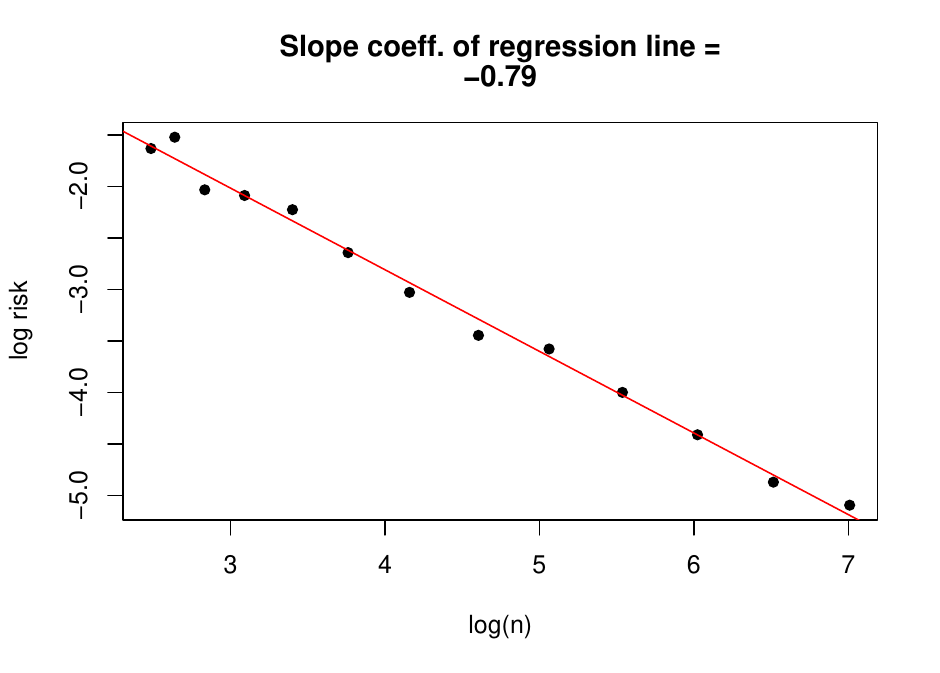}&\includegraphics[width=0.4\textwidth]{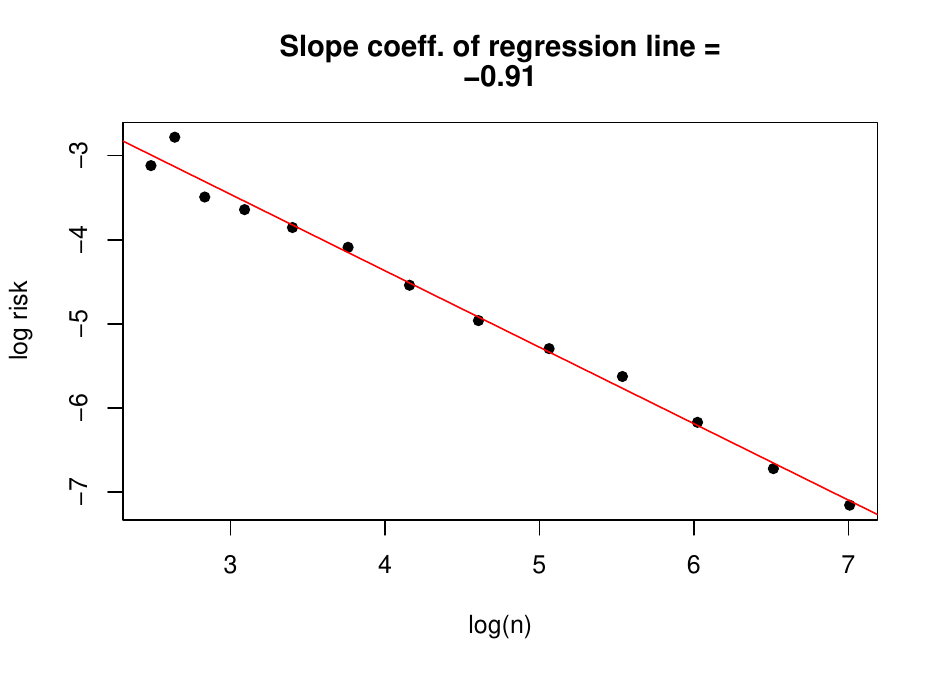}\\
 \hline
 $\rho_j=\rho_0e^{-(j-1)/2}$.       &&\\
 & \includegraphics[width=0.4\textwidth]{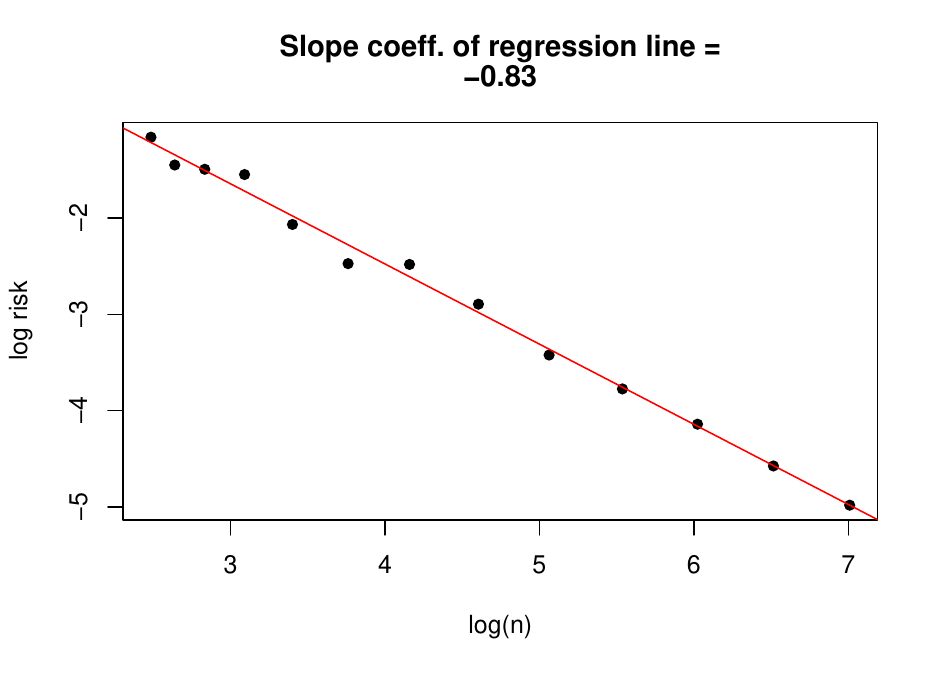}&\includegraphics[width=0.4\textwidth]{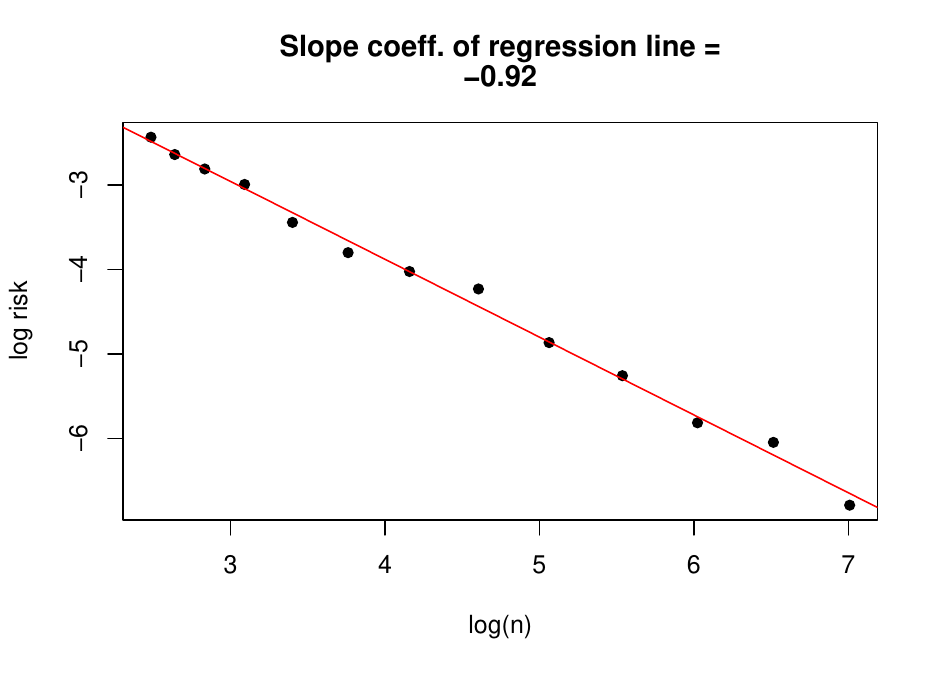}\\
 \hline
\end{tabular}
\caption{\label{fig:logrisk} Plot of Monte-Carlo estimation of the risk for different values of $n$ in $\log$-scale with estimation of the slope.  }
\end{figure}

\section{Conclusion and Perspectives}

In this paper, we have provided a comprehensive minimax analysis of the prediction error in the Functional Autoregressive model of order 1. While the existing literature has largely focused on the estimation of the autoregressive operator $\rho$ in terms of the operator norm or the Hilbert-Schmidt norm, our work addresses the fundamental problem of forecasting. By deriving minimax lower bounds and establishing matching upper bounds for the Tikhonov (Ridge) regularized estimator, we have characterized the optimal convergence rates for the predictive risk.

Our findings underscore the efficiency of Tikhonov regularization in the context of functional time series. We have demonstrated that by appropriately tuning the regularization parameter $\alpha$ in relation to the sample size $n$ and the spectral properties of the covariance operator $C$, the Ridge estimator achieves the minimax optimal rate. This rate is governed by a subtle balance between the regularity of the operator $\rho$ (expressed through harmonics decay) and the decay rate of the eigenvalues of the process. In particular, we have shown how the "effective dimension" $m$ of the functional data dictates the speed of convergence.

Several avenues for further investigation remain open. Adaptivity is an important matter, naturally prolonging this work. The optimal values of $\alpha$ depend on unknown quantities, such as the decreasing rate of the eigenvalues of the covariance operator $C$ as proven in Theorems \ref{LowBound} and \ref{risk-upper-bound} (see also Corollary \ref{cor-examples}). Conversely to linear regression, where the data are usually assumed to be independent, classical cross-validation cannot be applied directly in autoregressive models. Hence, developing a data-driven selection of the parameter leading to an adaptive estimator is an immediate perspective of this work. This necessitates defining an empirical version of the bias and estimating its variance appropriately.
    
Real-world functional data often exhibit more complex temporal dependencies. Extending this minimax framework to FAR($p$) models would require analyzing the resolved Yule-Walker equations under Tikhonov regularization and understanding how the prediction error scales with the lag order $p$. Extensions could also be investigated for more general linear processes such as:
\[
X_{n}=\varepsilon_{n}+\sum
_{k=1}^{+\infty}a_{k}\left( \varepsilon_{n-k}\right)
=\varepsilon_{n}+\sum_{k=1}^{n-1}a_{k}\left( \varepsilon_{n-k}\right) +X_{0}
\]
with $X_{0}=\sum_{k=n}^{+\infty}a_{k}\left( \varepsilon_{n-k}\right) $ and $\left( a_{k}\right)_{k\geq1}$ a sequence of linear operators satisfying the summability property $\sum_{k=1}^{+\infty}\left\Vert a_{k}\right\Vert _{\infty}^{2}<+\infty$.
    
It would be valuable to study the behavior of the Ridge-based predictor when the underlying process deviates from the FAR(1) structure. Quantifying the risk inflation due to model misspecification would provide practitioners with essential guarantees for the robustness of Tikhonov-based forecasting.

\section*{Acknowledgement} 
The authors acknowledge the support of the French Agence Nationale de la Recherche (ANR) under reference ANR-24-CE40-2439 (FUNMathStatproject).

\section*{AI Disclosure}
An AI was utilized on some parts of the manuscript exclusively for English language editing and proofreading to improve the clarity and readability . The AI was not used for data analysis, interpretation, or content generation, and the authors retain full responsibility for the scientific integrity and originality of the work.

\section{Proofs}

This section is split in three subsections and deserves a short introduction. The lower bound is entirely treated in the next subsection and the upper
bound in section \ref{section-upperbound} with two distinct parts dealing
respectively with a main term and a residual term. Within the proof of the
main terms upper bound an auxiliary result is required, namely Proposition 
\ref{variance}. The residual term in the upper bound depends on non-asymptotic concentration inequalities developped in the Appendix which also contains auxiliary results about entire series required throughout the paper. The last subsection of the proofs is short and devoted to the proof of Proposition \ref{prop.reg-var.bias}.

\subsection{Technical results about operators :\label{tech-res-op}}
Below we denote densely defined operators like bounded ones whenever they admit bounded extension. For instance $C^{1/2}C_{\varepsilon}^{-1/2}$ is densely defined with bounded extension because $C_{\varepsilon}^{-1/2}C^{1/2}$  is bounded. The four following facts are classical from functional analysis and may be found in \cite{Weidmann1980} at pages 58, 68, 71 and 72 respectively :
\begin{description}
\item[F1] If $T$ is self-adjoint $\left\Vert T\right\Vert _{\infty}=\sup\left\{ \left\vert
\left\langle Tx,x\right\rangle \right\vert ,\left\Vert x\right\Vert
=1\right\} ,$ and $\left\Vert T^{\ast}\right\Vert _{\infty}=\left\Vert
T\right\Vert _{\infty}$
\item[F2] If $T$ is a densely defined operator on $\mathbf{H}$. The operator 
$T$ is bounded if and only if $T^{\ast}$ is bounded. [ $C_{
\varepsilon}^{-1/2}C_{X}$ is bounded if and only if $\left(
C_{\varepsilon}^{-1/2}C_{X}\right) ^{\ast}$ is bounded]
\item[F3] Let $T$ and $S$ two densely defined operators on $H.$ If $ST$ is
densely defined then $\left( ST\right) ^{\ast}$ is an extension of $T^{\ast
}S^{\ast}$. If $S$ is bounded then $T^{\ast}S^{\ast}=\left( ST\right)
^{\ast} $.[ $C_{\varepsilon}^{-1/2}C_{X}^{1/2}$ is densely defined then $
\left( C_{\varepsilon}^{-1/2}C_{X}^{1/2}\right) ^{\ast}$ is an extension of $
\left( C_{X}^{1/2}\right) ^{\ast}\left( C_{\varepsilon}^{-1/2}\right)
^{\ast}=C_{X}^{1/2}\left( C_{\varepsilon}^{-1/2}\right) ^{\ast}$]
\item[F4] Let $T$ be densely defined and injective If $\mathrm{Im}T$ is
dense then $T^{\ast}$ is injective and $\left( T^{\ast}\right) ^{-1}=\left(
T^{-1}\right) ^{\ast}$ [$C_{\varepsilon}^{1/2}$ is densely defined and
injective. Since ran$C_{\varepsilon}^{1/2}$ is dense then $
\left( C_{\varepsilon}^{1/2}\right) ^{\ast}$ is injective and $\left[ \left(
C_{\varepsilon}^{1/2}\right) ^{\ast}\right] ^{-1}=\left(
C_{\varepsilon}^{1/2}\right) ^{-1}=C_{\varepsilon}^{-1/2}=\left(
C_{\varepsilon}^{-1/2}\right) ^{\ast}$].
\end{description}

\subsection{Lower bound}


The main tool for proving the lower bound is Proposition \ref{Kull-Leib-prop}, which gives us the Kullback-Leibler divergence between two distributions in the FAR model. The proof of Proposition \ref{Kull-Leib-prop} is based on the following variant of the Felman-Hajek theorem, proven in Section~\ref{subsec:ProofvarFeldHay}.
\begin{proposition}
\textbf{Variant of Feldman--H\'{a}jek's theorem} \label{varFeldHay}Let $\mu=
\mathcal{N}\left( m_{1},\Sigma_{1}\right) $ and $\nu=\mathcal{N}\left(
m_{2},\Sigma_{2}\right) $ be two Gaussian measures on $\mathbf{H}$ with $
\Sigma_{i}=R_{i}\Sigma_{0}R_{i}^{\ast}$ $i=1,2$,$\Sigma_{0}$ positive and trace class and $R_{i}$ bounded. Assume that $\ker\Sigma
_{1}=\ker\Sigma_{2}=\left\{ 0\right\} $. Denote $S_{i}=R_{i}\Sigma_{0}^{1/2}$. Then $\mu$ and $\nu$ are equivalent if and only if the 3 following
conditions hold~: (i) $\mathrm{Im}S_{1}=\mathrm{Im}S_{2}=H_{0}$, (ii) $
m_{1}-m_{2}\in H_{0}$, (iii) The operator $\left( S_{1}^{-1}S_{2}\right)
\left( S_{1}^{-1}S_{2}\right) ^{\ast}-I$ is a Hilbert-Schmidt operator on $
\overline{H}_{0}$. Besides if $R_{1}=I$ and $\left(
\Sigma_{0}^{-1/2}S_{2}\right) \left( \Sigma_{0}^{-1/2}S_{2}\right) ^{\ast}-I$
is trace-class we have :
\begin{equation}
\mathbf{KL}\left( \nu||\mu\right) =\mathrm{tr}\left[
\Sigma_{0}^{-1/2}S_{2}S_{2}^{\ast}\Sigma_{0}^{-1/2}-I\right] -\log\left(
\det\left( \Sigma_{0}^{-1/2}S_{2}S_{2}^{\ast}\Sigma_{0}^{-1/2}\right)
\right) .  \label{KL}
\end{equation}
\end{proposition}

\bigskip

\textbf{Proof of Proposition \ref{Kull-Leib-prop} :} We apply Proposition \ref{varFeldHay} to 
$ S_{1}=\mathbf{\Sigma_{0}^{1/2}}$ and $S_{2}=\mathbf{R}_{\rho}\mathbf{\Sigma_{0}^{1/2}}$. The reader will check that the operators $\mathbf{\Sigma_{0}^{-1/2}}\mathbf{R}_{\rho}\mathbf{\Sigma_{0}^{1/2}}$ and $\mathbf{\Sigma_{0}^{-1/2}}\mathbf{R}_{\rho}^{-1}\mathbf{\Sigma_{0}^{1/2}}$ are both bounded. This implies (i) in Proposition \ref{varFeldHay}. Remark that (ii) is trivially satisfied since the measures are centered. We then focus on point (iii) and we have to prove that the operator $\left( S_{1}^{-1}S_{2}\right) \left( S_{1}^{-1}S_{2}\right) ^{\ast }-I$ is an Hilbert-Schmidt operator when $\mathbf A_1$ holds. Notice that :
\begin{equation*}
\left( S_{1}^{-1}S_{2}\right) \left( S_{1}^{-1}S_{2}\right) ^{\ast }-I
\mathbf{=\Sigma_{0}^{-1/2}R}_{\rho}\mathbf{\Sigma_{0}R}_{\rho}^{\ast }
\mathbf{\Sigma_{0}^{-1/2}-I}
\end{equation*}
where $\mathbf{I}$ denotes here the identity operator in the space $\mathbf{H
}^{n+1}$. Simple computations lead to :
$$
\mathbf{\Sigma_{0}^{-1/2}R}_{\rho}\mathbf{\Sigma_{0}}^{1/2}=\mathbf{I}+\mathcal A, 
$$
where $\mathcal A$ can be written with the following block-matrix representation
$$
\mathcal A = \left[ 
\begin{array}{cccccc}
0 & A_{1} & A_{2} & \cdots & A_{n-1} & A_{nX} \\ 
0 & 0 & A_{1} & \ddots & \ddots & \vdots \\ 
\vdots & \ddots & \ddots & \ddots & \ddots & \vdots \\ 
\vdots & \cdots & \ddots & \ddots & A_{1} & A_{2X} \\ 
\vdots & \cdots & \cdots & \ddots & \ddots & A_{1X} \\ 
0 & \cdots & \cdots & \cdots & 0 & 0
\end{array}
\right], 
$$
with $A_{k}=C_{\varepsilon}^{-1/2}\rho^{k}C_{\varepsilon}^{1/2}$
and $A_{kX}=C_{\varepsilon}^{-1/2}\rho^{k}C^{1/2}$. Then 
$$
\left( S_{1}^{-1}S_{2}\right) \left( S_{1}^{-1}S_{2}\right) ^{\ast } = (\mathbf I+\mathcal A)(\mathbf I+\mathcal A^*) = \mathcal A\mathcal A^*+\mathcal A+\mathcal A^*+\mathbf I
$$
and 
$$
\left( S_{1}^{-1}S_{2}\right) \left( S_{1}^{-1}S_{2}\right) ^{\ast }-\mathbf I =  \mathcal A\mathcal A^*+\mathcal A+\mathcal A^*.
$$
Under $\mathbf A_1$, the operators $A_k$ and $A_{kX}$ are all Hilbert-Schmidt operators, which implies that $\mathcal A$, $\mathcal A^*$ and $\mathcal A\mathcal A^*$ are also Hilbert-Schmidt operators and that $(iii)$ is true. 
We now compute the Kullback-Leibler divergence.

Applying (\ref{KL}) of Theorem \ref{varFeldHay} we also get, since ${\rm tr}(\mathcal A)=0$,
\begin{align*}
\mathrm{tr}\left[ \mathbf{\Sigma_{0}^{-1/2}R}_{\rho}\mathbf{\Sigma_{0}R}
_{\rho}^{\ast}\mathbf{\Sigma_{0}^{-1/2}-I}\right] & =\mathrm{tr}\left[ 
\mathcal{A}\mathbf{+}\mathcal{A}^{\ast}+\mathcal{AA}^{\ast}\right] =\mathrm{
tr}\left[ \mathcal{AA}^{\ast}\right] =\left\Vert \mathcal{A}\right\Vert _{
\mathcal{S}_{2}\left( \mathbf{H}^{n}\right) }^{2} \\
& =\sum_{k=1}^{n-1}\left( n-k\right) \left\Vert A_{k}\right\Vert
_{2}^{2}+\sum_{k=1}^{n}\left\Vert A_{kX}\right\Vert _{2}^{2}.
\end{align*}
Moreover
\begin{equation*}
\det\left( \mathbf{\Sigma_{0}^{-1/2}R}_{\rho}\mathbf{\Sigma_{0}R}_{\rho
}^{\ast}\mathbf{\Sigma_{0}^{-1/2}}\right) =\det\left( \mathbf{I+}\mathcal{A}
\right) \left( \mathbf{I+}\mathcal{A}^{\ast}\right) =\det\left( \mathbf{I+}
\mathcal{A}\right) \det\left( \mathbf{I+}\mathcal{A}^{\ast }\right)=1, 
\end{equation*}
since all the eigenvalues of $\mathcal{A}$ are
null, hence $\log\left[ \det\left( \mathbf{
\Sigma_{0}^{-1/2}R}_{\rho}\mathbf{\Sigma_{0}R}_{\rho}^{\ast }\mathbf{
\Sigma_{0}^{-1/2}}\right) \right] =0$. This implies the expected result :
\begin{equation*}
\mathbf{KL}\left( \mathbf{P}_{\rho}||\mathbf{P}_{0}\right)
=\sum_{k=1}^{n-1}\left( n-k\right) \left\Vert
C_{\varepsilon}^{-1/2}\rho^{k}C_{\varepsilon}^{1/2}\right\Vert
_{2}^{2}+\sum_{k=1}^{n}\left\Vert
C_{\varepsilon}^{-1/2}\rho^{k}C^{1/2}\right\Vert _{2}^{2}.
\end{equation*}

Now in order to get the last bound we see first that : \begin{align*}
\mathbf{KL}\left(  \mathbf{P}_{\rho}||\mathbf{P}_{0}\right)  & \leq\left\Vert C_{\varepsilon}^{-1/2}C^{1/2}\right\Vert _{\infty}%
^{2}\left\{  \sum_{k=1}^{n}\left(  n-k\right)  \left\Vert C^{-1/2}%
C_{\varepsilon}^{1/2}\right\Vert _{\infty}^{2}\left\Vert C^{-1/2}\rho
^{k}C^{1/2}\right\Vert _{2}^{2}+\sum_{k=1}^{n}\left\Vert C^{-1/2}\rho
^{k}C^{1/2}\right\Vert _{2}^{2}\right\}  \\
& \leq a_{3}^{2}\sum_{k=1}^{n}\left(  n-k+1\right)  \left\Vert C^{-1/2}%
\rho^{k}C^{1/2}\right\Vert _{2}^{2}.
\end{align*}

Notice that $\left\Vert C^{-1/2}\rho^{k}C^{1/2}\right\Vert _{2}^{2}%
\leq\left\Vert C^{-1/2}\rho
C^{1/2}\right\Vert _{2}^{2}\left[  \left\Vert C^{-1/2}\rho C^{1/2}\right\Vert
_{\infty}^{2}\right]  ^{k-1}$which leads to :%
\begin{align*}
\mathbf{KL}\left(  \mathbf{P}_{\rho}||\mathbf{P}_{0}\right)    & \leq
na_{3}^{2}\left\Vert C^{-1/2}\rho C^{1/2}\right\Vert _{2}^{2}\sum_{k=1}%
^{n}\frac{n-k+1}{n}\left(  a_{1}^{\prime}\right)  ^{2\left(  k-1\right)  }\\
& \leq n\frac{a_{3}^{2}}{1-\left(  a_{1}^{\prime}\right)  ^{2}}\left\Vert
C^{-1/2}\rho C^{1/2}\right\Vert _{2}^{2}=n\left\Vert C^{-1/2}\rho
C^{1/2}\right\Vert _{2}^{2}b_{1}%
\end{align*}

with $b_{1}=\frac{a_{3}^{2}}{1-\left(  a_{1}^{\prime}\right)  ^{2}}$ which is the desired result.

\bigskip

\textbf{Proof of Theorem \ref{LowBound} :} The derivation of minimax lower bounds on a statistical model $\mathfrak{M}_{\rho }$ may
follow the route very nicely exposed for instance in 
\cite{Tsyb09} and conclude by one of the methods proposed ibidem (namely
Theorem 2.5 p.99). This route is summarized in three steps.

\textbf{Step 1 }: First the metric involved in the loss function must be
identified. Remind that $a_3 =\left\Vert C_{\varepsilon
}^{-1/2}C^{1/2}\right\Vert _{\infty }$. Here from :
\begin{equation*}
\mathbb{E}\left\Vert \left( \widehat{\rho }-\rho \right) X_{n+1}\right\Vert
^{2}\geq \mathbb{E}\left\Vert \left( \widehat{\rho }-\rho \right)
\varepsilon _{n+1}\right\Vert ^{2}=\left\Vert \left( \widehat{\rho }-\rho
\right) C^{1/2}C^{-1/2}C_{\varepsilon }^{1/2}\right\Vert _{2}^{2}\geq 
a_3^{-1}\left\Vert \left( \widehat{\rho }-\rho \right)
C^{1/2}\right\Vert _{2}^{2}
\end{equation*}
thanks to Proposition \ref{closed-op}. It is hence natural to take $d\left(
\rho _{1},\rho _{2}\right) =\left\Vert \left( \rho _{1}-\rho _{2}\right)
C^{1/2}\right\Vert _{2}^{2}$ as a metric on $\mathfrak{M}$. Note that this
metric involves both \textquotedblleft parts" of the model : the regularity
of $\rho $ and the regularity of $\varepsilon $ and $X$ through their covariance operator and $a_3$.

\textbf{Step 2 }: Another metric operating between two probability distributions $\mathbb{P}_{\rho}$ and $\mathbb{P}
_{\rho^{\prime}}$ $\rho,\rho^{\prime}\in\mathfrak{M}$ has to be selected.
Here $\mathbb{P}_{\rho}$ denotes the distribution of the vector of the $n$
output data which depends on $\rho$. This metric will measure how close are $
\mathbb{P}_{\rho}$ and $\mathbb{P}_{\rho^{\prime}}$ : Hellinger distance,
Kullback-Leibler divergence, Total-Variation distance are commonly used.
Below we take advantage of Theorem \ref{varFeldHay} and choose the Kullback-Leibler
divergence denoted $\mathbf{KL}\left( \mathbf{.},\mathbf{.}\right) $.

\textbf{Step 3 }: This tough step consists in selecting a possibly large set 
$\mathcal{I}$ of distinct parameters $\rho _{i}\in \mathfrak{M}$ and evaluate :
\begin{equation*}
\max_{_{\substack{ \rho _{i},\rho _{j}\in \mathfrak{M}  \\ \rho _{i}\neq
\rho _{j}}}}d\left( \rho _{i},\rho _{j}\right) \quad s.t.\quad \mathbf{KL}
\left( \mathbb{P}_{\rho _{i}},\mathbb{P}_{\rho _{j}}\right) \leq c\log
\left\vert \mathcal{I}\right\vert
\end{equation*}
for some constant $c$. With words, the testing parameters should be
far enough from each other for the metric $d\left( .,.\right) $ and
simultaneously should generate statistical models relatively close enough.
It should be noted that assuming that $\rho _{i}\in \mathfrak{M}$ comes down
practically to adding constraints to the program above. A lower bound for
the metric $d\left( .,.\right) $ is then $\max_{\rho _{i}\neq \rho
_{j}}d\left( \rho _{i},\rho _{j}\right) $.

Our only task is to complete \textbf{Step 3} above with $\mathfrak{M=M}_{G}$
. We construct a set of testing parameters. Let $\mathcal{VG}\left( m\right) 
$ be a Varshamov-Gilbert hypercube of size $m$. Then $\mathcal{VG}\left(
m\right) \subset \left\{ 0,1\right\} ^{m}$ and for all $\left( \omega
,\omega ^{\prime }\right) \in $ $\mathcal{VG}\left( m\right) $ it is
possible to take simultaneously :
\begin{equation*}
\mathrm{card}\mathcal{VG}\left( 2m\right) =M,\quad M>2^{m/8},\quad \mathrm{
ham}\left( \omega ,\omega ^{\prime }\right) \geq m/8
\end{equation*}
where $\mathrm{ham}\left( \omega ,\omega ^{\prime }\right) $ denotes the
Hamming distance between $\omega $ and $\omega ^{\prime }$ (see \cite{Tsyb09}
Lemma 2.9 p.104). Let $\theta =2\omega -1$, $\theta ^{\prime }=2\omega
^{\prime }-1$ and set :
\begin{equation*}
\rho _{\theta }\left( .\right) =\sqrt{\lambda _{0}}e_{1}\left( \sum_{p=1}^{m}
\frac{\theta _{p}}{\sqrt{\lambda _{p}}}\gamma _{p}\left\langle
e_{p},.\right\rangle +\sum_{p=m+1}^{2m}\frac{\gamma _{p}}{\sqrt{\lambda _{p}}
}\left\langle e_{p},.\right\rangle \right)
\end{equation*}
where $\lambda _{0}$ will be tuned below and $\left( \gamma _{p}\right)
_{1\leq p\leq 2m}$ is a sequence such that $\gamma _{p}=1/\sqrt{n}$ for all $
1\leq p\leq m$. Then :
\begin{equation}
\min_{\theta \neq \theta ^{\prime }}\left\Vert \left( \rho _{\theta }-\rho
_{\theta ^{\prime }}\right) C^{1/2}\right\Vert _{2}^{2}\geq 4\lambda
_{0}\min_{1\leq p\leq m}\gamma _{p}^{2}\cdot \mathrm{ham}\left( \omega
,\omega ^{\prime }\right) \geq \frac{\lambda _{0}}{2}\frac{m}{n}.
\end{equation}
Turn to the Kullback-Leibler divergence $\mathbf{KL}\left( \mathbf{P}_{\rho
_{\theta }}||\mathbf{P}_{0}\right) $ and take in addition $\gamma _{p}=\left[
\lambda _{2m}/\left( n\lambda _{m}\right) \right] ^{1/2}$ for $m+1\leq p\leq
2m$. The constant $b_1$ in (\ref{KL-bound}) depends on $\rho_{\theta}$ here, through its denumerator. It is easy to check that $a'_{1}=(\max_p \gamma _{p}) \sqrt{\lambda_0 / \lambda_1}$ and fixing $\lambda_0 \leq n \lambda_1 /2$ ensures that $b_{1} \leq 2 a_3^2$ where $a_3^2$ does not depend on $\rho$. Also notice that since $\lambda \left( \cdot \right) $ is decreasing we have that $\sup_{m}\lambda _{2m}/\lambda _{m} \leq 1$.
Then an application of (\ref{KL-bound}) gives : 


\begin{equation*}
\mathbf{KL}\left( \mathbf{P}_{\rho _{\theta }}||\mathbf{P}_{0}\right) \leq 
\frac{\lambda _{0}}{\lambda _{1}}b_{1}\left( m+m\sup_{m}\frac{
\lambda _{2m}}{\lambda _{m}}\right) \leq 4\frac{\lambda _{0}}{\lambda _{1}}
a_{3}^2 m\leq a_{3}^2 \frac{32}{\log 2}\frac{\lambda _{0}}{
\lambda _{1}}\log M<\frac{1}{9}\log M
\end{equation*}
when $\lambda _{0}=(\lambda _{1} \log{2})/(320 a_3^2)$ for instance (which is certainly strictly stronger than the first constraint $\lambda_0 \leq n \lambda_1 /2$). Now the integer $
m$ will be selected so that the constraints of model $\mathfrak{M}_{G}$
hold. The first restrictions are related to the definition and geometry of $
\mathfrak{M}_{G}$. We first require
simultaneously :
\begin{align*}
\left\Vert C_{\varepsilon }^{-1/2}\rho _{\theta }\right\Vert _{\infty }^{2}&
\leq a_3\left\Vert C^{-1/2}\rho _{\theta }\right\Vert _{2}^{2}=
\frac{\lambda _{0}a_3}{n\lambda _{1}}\left( \sum_{p=1}^{m}\frac{1}{
\lambda _{p}}+\sup_{m}\frac{\lambda _{2m}}{\lambda _{m}}\sum_{p=m+1}^{2m}
\frac{1}{\lambda _{p}}\right) \leq a_{2}^{2}, \\
\left\Vert \rho _{\theta }\right\Vert _{\infty }^{2}& \leq \frac{\lambda _{0}
}{n}\left( \sum_{p=1}^{m}\frac{1}{\lambda _{p}}+\sup_{m}\frac{\lambda _{2m}}{
\lambda _{m}}\sum_{p=m+1}^{2m}\frac{1}{\lambda _{p}}\right) \leq a_{1}^{2}.
\end{align*}
Clearly both equations above are quite similar with different
constants. The last constraint, expressing the right-regularity of $\rho
_{\theta }$, turns out to be stricter so that we ignore the two equations
just above and turn to :
\begin{equation*}
\max \left\{ \sum_{p=1}^{m}\frac{\left\Vert \rho _{\theta }\left(
e_{p}\right) \right\Vert ^{2}}{\lambda _{p}},\frac{1}{\lambda _{m}}
\sum_{p=m+1}^{2m}\left\Vert \rho _{\theta }\left( e_{p}\right) \right\Vert
^{2}\right\} \leq \tau \left( m\right)
\end{equation*}
with :
\begin{eqnarray*}
\mathbf{b}_{1}\left( n,m\right) &=&\sum_{1\leq p\leq m}\frac{\left\Vert \rho
_{\theta }\left( e_{p}\right) \right\Vert ^{2}}{\lambda _{p}}=\frac{\lambda
_{1}}{n}\cdot \sum_{1\leq p\leq m}\frac{1}{\lambda _{p}^{2}}\leq \frac{m}{n}
\frac{\lambda _{1}}{\lambda _{m}^{2}} \\
\mathbf{b}_{2}\left( n,m\right) &=&\frac{1}{\lambda _{m}}\sum_{p=m+1}^{2m}
\left\Vert \rho _{\theta }\left( e_{p}\right) \right\Vert ^{2}=\frac{1}{n}
\frac{\lambda _{1}}{\lambda _{m}}\frac{\lambda _{2m}}{\lambda _{m}}
\sum_{p=m+1}^{2m}\frac{1}{\lambda _{p}}\leq \frac{m}{n}\frac{\lambda _{1}}{
\lambda _{m}^{2}}.
\end{eqnarray*}
We see above that the sequence $\gamma _{p}$ for $m+1\leq p\leq 2m$ was
chosen so that $\mathbf{b}_{1}\left( n,m\right) $ and $\mathbf{b}_{2}\left(
n,m\right) $ are bounded by the same quantity. This last constraint leads to 
$m/\left( n\lambda _{m}^{2}\right) \leq \tau \left( m\right) $ Since the
minimax rate takes the form (up to a constant) : $m/n$ we are led to
choosing the largest $m$ so that ;
\begin{equation*}
m^{\ast }=\max_{m}\left\{ \frac{1}{n}\leq \frac{\tau \left( m\right) \lambda
_{m}^{2}}{m}\right\} .
\end{equation*}
given at (\ref{alpha.star}).

We constructed above a set of testing parameters $\left\{ \rho _{\theta
},\theta =2\omega -1,\omega \in \mathcal{VG}\left( m^{\ast }\right) \right\}
\subset \mathfrak{M}_{\rho }$. We are now in position to apply (\cite{Tsyb09}
, Theorem 2.7 p.101). which yields for some universal constant $
\mathfrak{c}$ :
\begin{equation*}
\inf_{\widehat{\rho }}\sup_{\rho _{\theta }\in \mathfrak{M}_{G}}\mathbb{E}
\left( \left\Vert \left( \rho _{\theta }-\widehat{\rho }\right)
C_{\varepsilon }^{1/2}\right\Vert _{2}^{2}\right) \geq \mathfrak{c}\frac{
\lambda _{0}}{4}\frac{m^{\ast }}{n}.
\end{equation*}

\subsection{General and Minimax upper bound\label{section-upperbound}}
All along this section denote 
$$
C_n^\dag = (C_n+\alpha\mathbf I)^{-1},\ C^\dag = (C+\alpha\mathbf I)^{-1}
$$
and
\begin{equation*}
U_{n}=\dfrac{1}{n}\sum_{k=1}^{n-1}\varepsilon_{k+1}\otimes X_{k},\quad
R_{n,\alpha}=C^{\dagger1/2}\left( C-C_{n}\right) C^{\dagger1/2},\quad
T_{n,\alpha}=U_{n}C^{\dag}-\alpha\rho C^{\dag}.
\end{equation*}

Notice that $nU_{n}$ is a $\mathcal{S}_{2}$-valued martingale.

\textbf{Proof of Theorem \ref{risk-upper-bound} :} Let $0<c_{n}<1$ and define the set $\mathcal{E}_{n}$ this way :
\begin{equation}
\mathcal{E}_{n}=\mathcal{E}_{n}\left( c_{n}\right) =\left\{ \left\Vert
R_{n,\alpha}\right\Vert _{\infty}<c_{n}\right\} \cap\left\{ \left\Vert
R_{n,\alpha}C^{\dagger1/2}X_{n+1}\right\Vert <1\right\}. \label{E_n}
\end{equation}

The prediction risk is split in two terms : 
\begin{equation}
\mathbb{E}\left\Vert \left( \widehat{\rho}_{n}-\rho\right)
X_{n+1}\right\Vert ^{2}=\underset{\mathcal{M}_{n}}{\underbrace{\mathbb{E}
\left\{ \left\Vert \left( \widehat{\rho}_{n}-\rho\right) X_{n+1}\right\Vert
^{2}1_{\mathcal{E}_{n}}\right\} }}+\underset{\mathcal{R}_{n}}{\underbrace{
\mathbb{E}\left\{ \left\Vert \left( \widehat{\rho}_{n}-\rho\right)
X_{n+1}\right\Vert ^{2}1_{\overline{\mathcal{E}}_{n}}\right\} }}.
\label{risk-prelim-dec}
\end{equation}
The expectation in $\mathcal{M}_{n}$ is the main part and $\mathcal{R}_{n}$
is a residual term. Roughly speaking when $\mathcal{E}
_{n} $ holds, $C_{n}$ is close enough to $C$ so that replacing the random
and regularized operator $C_{n}^{\dagger}$ by a non random -though still
regularized- version $C^{\dagger}$ in $\mathcal{M}_{n}$ is possible. The price to pay is the control of $\mathbb{P}\left( \overline{\mathcal{E}}
_{n}\right) $ in $\mathcal{R}_{n}$ which comes down to a concentration
inequality on $C_{n}^{\dagger}$.

The proof of the Theorem is a consequence of Propositions \ref{main-term} and \ref{prop-up-bpund-res}. The tuning parameter $c_n$ is just left in $(0,1)$ in Proposition \ref{main-term} and fixed to $(mn)^{-1/6}$ in Lemma \ref{P-barre} and through subsection \ref{residual-meansq-risk-section}. The condition $m/(n\lambda_m)<1$ coupled with $m \lambda_m \leq 1$ leads to $m^3\leq mn$ and finally $(mn)^{-1/6} \leq m^{-1/2}$. Assuming $m \geq 4$ Proposition \ref{main-term} gives $\Lambda_{\mathbf{B}} \leq 12$ and $\Lambda_{\mathbf{V}} \leq 8 +4\left(
	a_{3}^{\prime}\right) ^{2}/\sigma_{\varepsilon}^{2}$. This is the upper bound given in Theorem \ref{risk-upper-bound}. All the details are deferred to the two next sections \ref{bias-variance-section} and \ref{residual-meansq-risk-section}. But a sketch of the proof for Theorem \ref{risk-upper-bound} is given now in order to help the reader.	 The three main steps of the proof are :

\textbf{First step :} The main term $\mathcal{M}_{n}$ is bounded with $\Lambda_{\mathbf{V}},\mathbf{V}_{n},\Lambda_{\mathbf{B}}$ and $\mathbf{B}_{n}$ plus an additional term that may be neglected.

\textbf{Second step :} The calculation of variance $\mathbf{V}_{n}$ is
given in subsection \ref{bias-variance-section}.

\textbf{Third step :} The block $\mathcal{R}_{n}$ is bounded by a non
asymptotic concentration inequality in section \ref
{residual-meansq-risk-section}. This inequality, whose proof is deferred to the Appendix, is adapted from existing
Hilbert-valued martingale difference general theorems.

\subsubsection{Main term\label{bias-variance-section}}

\begin{proposition}
\label{main-term}For all $n$ :
\begin{equation*}
\mathcal{M}_{n}=\mathbb{E}\left[ \left\Vert \left( \widehat{\rho}
_{n}-\rho\right) X_{n+1}\right\Vert ^{2}1_{\mathcal{E}_{n}}\right] \leq
\Lambda_{\mathbf{V}}\cdot\mathbf{V}_{n}+\Lambda_{\mathbf{B}}\cdot \mathbf{B}
_{n}+4\frac{\mathcal{R}'_{n}}{n}
\end{equation*}
where : 
\[
\Lambda_{\mathbf{V}} =\frac{2}{\left( 1-c_n\right) ^{2}}+4(a'_{3})^{2},\quad\Lambda _{\mathbf{B
}}=\frac{2}{\left( 1-c_n\right) ^{2}}+4
\]

and $\mathcal{R}'_{n}$ is a bounded sequence.
\end{proposition}

\textbf{Proof of Proposition \ref{main-term} :} By definition $\widehat{\rho}_{n}=D_{n}C_{n}^{\dagger}$ and $D_n = \rho C_n+U_n$ then
\begin{align*}
\widehat{\rho}_{n} & =D_{n}C_{n}^{\dagger}\left( C-C_{n}\right)
C^{\dagger}+D_{n}C^{\dagger} \\
& =\widehat{\rho}_{n}\left( C-C_{n}\right) C^{\dagger}+\rho\left(
C_{n}-C\right) C^{\dag}+\rho CC^{\dag}+U_{n}C^{\dag}.
\end{align*}
Hence :
\begin{equation}
\widehat{\rho}_{n}-\rho=\left( \widehat{\rho}_{n}-\rho\right) \left(
C+\alpha \mathbf I\right) ^{1/2}R_{n,\alpha}C^{\dagger1/2}+T_{n,\alpha}
\label{main-temp}
\end{equation}
with $T_{n,\alpha}$ and $R_{n,\alpha}$ defined just above ($
T_{n,\alpha}=U_{n}C^{\dag}-\alpha\rho C^{\dag}$). We state a first
intermediate equation :
\begin{align*}
\left( \widehat{\rho}_{n}-\rho\right) \left( C+\alpha\mathbf I\right) ^{1/2} &
=\left( \widehat{\rho}_{n}-\rho\right) \left( C+\alpha\mathbf I\right)
^{1/2}R_{n,\alpha}+T_{n,\alpha}\left( C+\alpha\mathbf I\right) ^{1/2} \\
\left( \widehat{\rho}_{n}-\rho\right) \left( C+\alpha\mathbf I\right) ^{1/2}\left(
I-R_{n,\alpha}\right) & =T_{n,\alpha}\left( C+\alpha\mathbf I\right) ^{1/2}.
\end{align*}
When $\left\Vert R_{n,\alpha}\right\Vert _{\infty}<c_{n}<1$ :
\begin{equation*}
\left( \widehat{\rho}_{n}-\rho\right) \left( C+\alpha\mathbf I\right)
^{1/2}=T_{n,\alpha}\left( C+\alpha\mathbf I\right) ^{1/2}\left( I-R_{n,\alpha
}\right) ^{-1}.
\end{equation*}
Then plugging the previous bound in (\ref{main-temp}) :
\begin{align*}
\left( \widehat{\rho}_{n}-\rho\right) X_{n+1}1_{\mathcal{E}_{n}} & =\left( \widehat{\rho}_{n}-\rho\right) \left(
C+\alpha \mathbf I\right) ^{1/2}R_{n,\alpha}C^{\dagger1/2}X_{n+1}1_{\mathcal{E}_{n}}+T_{n,\alpha}X_{n+1}1_{\mathcal{E}_{n}}\\
&=T_{n,\alpha}\left( C+\alpha\mathbf I\right) ^{1/2}\left( I-R_{n,\alpha}\right)
^{-1}R_{n,\alpha}C^{\dagger1/2}X_{n+1}1_{\mathcal{E}_{n}}+T_{n,
\alpha}X_{n+1}1_{\mathcal{E}_{n}} \\
\left\Vert \left( \widehat{\rho}_{n}-\rho\right) X_{n+1}\right\Vert 1_{
\mathcal{E}_{n}} & \leq\left\Vert T_{n,\alpha}\left( C+\alpha\mathbf I\right)
^{1/2}\right\Vert _{2}\left\Vert \left( I-R_{n,\alpha}\right)
^{-1}R_{n,\alpha}C^{\dagger1/2}X_{n+1}\right\Vert 1_{\mathcal{E}
_{n}}+\left\Vert T_{n,\alpha}X_{n+1}\right\Vert \\
& \leq\left( 1-c_{n}\right) ^{-1}\left\Vert T_{n,\alpha}\left( C+\alpha
I\right) ^{1/2}\right\Vert _{2}\left\Vert
R_{n,\alpha}C^{\dagger1/2}X_{n+1}\right\Vert +\left\Vert
T_{n,\alpha}X_{n+1}\right\Vert .
\end{align*}
We used the bound $\left\Vert (I-T)^{-1} \right\Vert_{\infty} \leq (1-\left\Vert T \right\Vert_{\infty} )^{-1}$ when $\left\Vert T \right\Vert_{\infty}<1$. Now when $\left\Vert R_{n,\alpha}C^{\dagger1/2}X_{n+1}\right\Vert <1$ : 
\begin{equation*}
\left\Vert \left( \widehat{\rho}_{n}-\rho\right) X_{n+1}\right\Vert
\leq\left( 1-c_{n}\right) ^{-1}\left\Vert T_{n,\alpha}\left( C+\alpha
I\right) ^{1/2}\right\Vert _{2}+\left\Vert T_{n,\alpha}X_{n+1}\right\Vert.
\end{equation*}

The behaviour of $c_{n}$ is not important for the moment and will be
detailed later when elaborating on the concentration inequalities.

Summarizing the previous results with $\mathcal{E}_{n}=\left\{ \left\Vert
R_{n,\alpha}\right\Vert _{\infty}<c_{n}\right\} \cap\left\{ \left\Vert
R_{n,\alpha}C^{\dagger1/2}X_{n+1}\right\Vert <1\right\} $ we get :
\begin{equation*}
\left\Vert \left( \widehat{\rho}_{n}-\rho\right) X_{n+1}\right\Vert 1_{
\mathcal{E}_{n}}\leq\frac{1}{1-c_{n}}\left\Vert T_{n,\alpha}\left( C+\alpha
I\right) ^{1/2}\right\Vert _{2}+\left\Vert T_{n,\alpha}X_{n+1}\right\Vert
\end{equation*}
\begin{equation}
\mathbb{E}\left\{ \left\Vert \left( \widehat{\rho}_{n}-\rho\right)
X_{n+1}\right\Vert ^{2}1_{\mathcal{E}_{n}}\right\} \leq2\left( \frac {1}{
1-c_{n}}\right) ^{2}\mathbb{E}\left\Vert T_{n,\alpha}\left( C+\alpha
I\right) ^{1/2}\right\Vert _{2}^{2}+2\mathbb{E}\left\Vert
T_{n,\alpha}X_{n+1}\right\Vert ^{2}.  \label{rain}
\end{equation}

It is simpler to go on with $T_{n,\alpha}\left( C+\alpha\mathbf I\right)
^{1/2}=U_{n}C^{\dag/2}-\alpha\rho C^{\dag/2}$ and :
\begin{equation*}
\mathbb{E}\left\Vert T_{n,\alpha}\left( C+\alpha\mathbf I\right) ^{1/2}\right\Vert
_{2}^{2}=\mathbb{E}\left\Vert U_{n}C^{\dag/2}\right\Vert _{2}^{2}+\alpha
\left\Vert \rho C^{\dag/2}\right\Vert _{2}^{2},
\end{equation*}
which is split in a bias term :
\begin{equation*}
\mathbf{B}_{n}=\left\Vert \alpha\rho C^{\dag/2}\right\Vert _{2}^{2}=\alpha
^{2}\sum_{p=1}^{+\infty}\frac{\left\Vert \rho e_{p}\right\Vert ^{2}}{
\lambda_{p}+\alpha},
\end{equation*}
and a variance part related to $\mathbb{E}\left\Vert U_{n}C^{\dag
/2}\right\Vert _{2}^{2}$, term that may developped :
\begin{align*}
\mathbb{E}\left\Vert U_{n}C^{\dag/2}\right\Vert _{2}^{2} & =\frac{1}{n^{2}}
\sum_{p=1}^{+\infty}\mathbb{E}\left\Vert \sum_{k=1}^{n-1}\varepsilon
_{k+1}\left\langle C^{\dag/2}X_{k},e_{p}\right\rangle \right\Vert ^{2}=\frac{
1}{n^{2}}\sum_{p=1}^{+\infty}\sum_{k=1}^{n-1}\mathbb{E}\left\Vert
\varepsilon_{k+1}\right\Vert ^{2}\mathbb{E}\left\langle
C^{\dag/2}X_{k},e_{p}\right\rangle ^{2} \\
& \leq\frac{\mathbb{E}\left\Vert \varepsilon_{1}\right\Vert ^{2}}{n}
\sum_{p=1}^{+\infty}\mathbb{E}\left\langle
C^{\dag/2}X_{0},e_{p}\right\rangle ^{2}=\frac{\sigma_{\varepsilon}^{2}}{n}
\sum_{p=1}^{+\infty}\frac {\mathbb{E}\left\langle X_{0},e_{p}\right\rangle
^{2}}{\lambda_{p}+\alpha }=\frac{\sigma_{\varepsilon}^{2}}{n}
\sum_{p=1}^{+\infty}\frac{\lambda_{p}}{\lambda_{p}+\alpha}=\mathbf{V}_{n}
\end{align*}
so that 
\begin{equation*}
\mathbb{E}\left\{ \left\Vert \left( \widehat{\rho}_{n}-\rho\right)
X_{n+1}\right\Vert ^{2}1_{\mathcal{E}_{n}}\right\} \leq2\left( \frac {1}{
1-c_{n}}\right) ^{2}\left\{ \mathbf{B}_{n}+\mathbf{V}_{n}\right\} +2\mathbb{E
}\left\Vert T_{n,\alpha}X_{n+1}\right\Vert ^{2}.
\end{equation*}
We are going to prove that $\mathbb{E}\left\Vert
T_{n,\alpha}X_{n+1}\right\Vert ^{2}$ is comparable to $\mathbb{E}\left\Vert
T_{n,\alpha}\left( C+\alpha\mathbf I\right) ^{1/2}\right\Vert _{2}^{2}$. If $
X_{n+1} $ was independent from $T_{n,\alpha}$ this would be straigthforward.
Here $X_{n+1}$ depends on all the terms inside $U_{n}$ which yields
technical difficulties. We have first :
\begin{equation*}
\mathbb{E}\left\Vert T_{n,\alpha}X_{n+1}\right\Vert ^{2}=\mathbb{E}
\left\Vert \left( U_{n}C^{\dag}-\alpha\rho C^{\dag}\right)
X_{n+1}\right\Vert ^{2}\leq2\mathbb{E}\left\Vert
U_{n}C^{\dag}X_{n+1}\right\Vert ^{2}+2\alpha ^{2}\mathbb{E}\left\Vert \rho
C^{\dag}X_{n+1}\right\Vert ^{2}
\end{equation*}
Clearly $\mathbb{E}\left\Vert \rho C^{\dag}X_{n+1}\right\Vert ^{2}=\sum
_{p=1}^{+\infty}\frac{\lambda_{p}\left\Vert \rho e_{p}\right\Vert ^{2}}{
\left( \lambda_{p}+\alpha\right) ^{2}}\leq\sum_{p=1}^{+\infty}\frac{
\left\Vert \rho e_{p}\right\Vert ^{2}}{\lambda_{p}+\alpha}$. It is proved in
Corollary \ref{cor-variance} below that :
\begin{equation*}
\mathbb{E}\left\Vert U_{n}C^{\dag}\left( X_{n+1}\right) \right\Vert ^{2}\leq
(a'_{3})^{2} \,\mathbf{V}
_{n}+\frac{\mathcal{R}'_{n}}{n} 
\end{equation*}
so that 
\begin{equation*}
\mathbb{E}\left\Vert T_{n,\alpha}X_{n+1}\right\Vert ^{2}\leq 2(a'_{3})^{2} \,\mathbf{V}
_{n}+2\mathbf{
B}_{n}+2\frac{\mathcal{R}'_{n}}{n}.
\end{equation*}
Summing up the few lines above with equation (\ref{rain}) gives :
\begin{align*}
\mathbb{E}\left\{ \left\Vert \left( \widehat{\rho}_{n}-\rho\right)
X_{n+1}\right\Vert ^{2}1_{\mathcal{E}_{n}}\right\} & \leq2\left( \frac {1}{
1-c_{n}}\right) ^{2}\left[ \mathbf{B}_{n}+\mathbf{V}_{n}\right] +4\left[ 
\mathbf{B}_{n}+(a'_{3})^2 \mathbf{V}_{n}+\frac{1}{n}\mathcal{R}_{n}^{\prime }\left(
\alpha\right) \right] \\
& \leq\Lambda_{\mathbf{V}}\cdot\mathbf{V}_{n}+\Lambda_{\mathbf{B}}\cdot
\mathbf{B}_{n}+4\frac{\mathcal{R}'_{n}}{n}
\end{align*}
with :
\begin{equation*}
\Lambda_{\mathbf{V}}=\frac{2}{\left( 1-c_{n}\right) ^{2}}+4 (a'_{3})^{2} \quad\Lambda _{\mathbf{B
}}=\frac{2}{\left( 1-c_{n}\right) ^{2}}+4
\end{equation*}
which is the desired result.

The next crucial Proposition is proved in the Appendix.

\begin{proposition}
\label{variance}We have for all $n$ :
\begin{align*}
\mathbb{E}\left\Vert U_{n}C^{\dag}\left( X_{n+1}\right) \right\Vert ^{2} & =
\frac{n-1}{n^2}  \sigma_{\varepsilon}^{2} \left[ \sum_{k=1}^{+\infty}\frac{\lambda_{k}}{\left( \lambda
_{k}+\alpha\right) ^{2}}\left\langle C_{\varepsilon}e_{k},e_{k}\right\rangle 
\right] +\frac{1}{n}\mathcal{R}_{n}^{\prime}, \\
0 & \leq\mathcal{R}_{n}^{\prime}\leq r_{v}+\frac{1}{n}\left(r'_{v}+
\Theta_{1}+m \Theta_{2}\right)
\end{align*}
with :
\[
0 \leq r_{v}\leq\frac{\sigma_{\varepsilon}^{2}}{1-a_{1}^{2}}\left(
a_{2}^{\prime}\right) ^{2}, \quad  0 \leq r'_{v} \leq  2\left\Vert C_{\varepsilon}^{2}\right\Vert _{1} (a'_2)^2 \frac{a_{1}^{2}}{\left(
1-a_{1}^{2}\right) ^{2}}
\]
and $\Theta_{1}, \Theta_{2}$ are two constants given in Lemma \ref{moinsmoins}. 
\end{proposition}

\begin{remark}
The statement of the previous Proposition gives an exact approximation of the dominant variance term plus a residual part (namely $\mathcal{R}_{n}^{\prime} /n $) 
\end{remark}

The next Corollary gives a rough but more compact version of Proposition \ref{variance}. It is not proved.

\begin{corollary}
\label{cor-variance}The following bound holds under the assumptions of Proposition \ref{variance} :
\begin{align*}
\mathbb{E}\left\Vert U_{n}C^{\dag}\left( X_{n+1}\right) \right\Vert ^{2} &
\leq (a'_{3})^{2} \,\mathbf{V}_{n}+\frac{1}{n}  \mathcal{R}'_{n}.
\end{align*}
\end{corollary}

\noindent \textbf{Proof of Proposition \ref{variance}:} The technique of proof is inspired by \cite{Mas2007}. But
the non asymptotic framework and the search of sharp bounds make
computations quite longer :
\begin{equation*}
U_{n}C^{\dag}\left( X_{n+1}\right) =\frac{1}{n}\sum_{k=1}^{n-1}
\varepsilon_{k+1}\left\langle X_{k},C^{\dag}\left( X_{n+1}\right)
\right\rangle =\frac{1}{n}\sum_{k=1}^{n-1}\left(
Z_{k,n}^{+}+Z_{k,n}^{-}\right)
\end{equation*}

with :
\begin{equation*}
Z_{k,n}^{+}=\varepsilon_{k+1}\left\langle X_{k},C^{\dag}\left( \varepsilon
_{n+1}+...+\rho^{n-k-1}\varepsilon_{k+2}\right) \right\rangle \quad
Z_{k,n}^{-}=\varepsilon_{k+1}\left\langle
X_{k},C^{\dag}\rho^{n-k}X_{k+1}\right\rangle .
\end{equation*}
A primary decomposition is :
\begin{equation*}
\left\Vert U_{n}C^{\dag}\left( X_{n+1}\right) \right\Vert ^{2}=\frac {1}{
n^{2}}\left\Vert \sum_{k=1}^{n-1}\left( Z_{k,n}^{+}+Z_{k,n}^{-}\right)
\right\Vert ^{2}=\frac{1}{n^{2}}\left\{ \left\Vert
\sum_{k=1}^{n-1}Z_{k,n}^{+}\right\Vert ^{2}+2\sum_{1\leq j,k\leq
n-1}\left\langle Z_{j,n}^{+},Z_{k,n}^{-}\right\rangle +\left\Vert
\sum_{k=1}^{n-1}Z_{k,n}^{-}\right\Vert ^{2}\right\}.
\end{equation*}

The Proposition follows from three forthcoming Lemmas adressing each of the
three terms in the equation above.

\begin{lemma}\label{plusplus}
We have :
\begin{align*}
\frac{1}{n^{2}}\mathbb{E}\left\Vert \sum_{k=1}^{n-1}Z_{k,n}^{+}\right\Vert
^{2} & = \sigma_{\varepsilon}^{2} \frac{n-1}{n^2}\left[ \sum_{p=1}^{+\infty}\frac{\lambda_{p}}{\left(
\lambda_{p}+\alpha\right) ^{2}}\left\langle
C_{\varepsilon}e_{p},e_{p}\right\rangle +r_{v}\right], \\
0 & \leq r_{v}\leq \sigma_{\varepsilon}^{2} \frac{(a'_2)^2}{1-a_{1}^{2}}.
\end{align*}
\end{lemma}

\textbf{Proof of Lemma \ref{plusplus} :} Take $1\leq i<k\leq n-1$ then :
\begin{align*}
\mathbb{E}\left\langle Z_{k,n}^{+},Z_{i,n}^{+}\right\rangle & =\mathbb{E}
\left\langle \varepsilon_{k+1},\varepsilon_{i+1}\right\rangle \left\langle
X_{k},C^{\dag}\left( \varepsilon_{n+1}+...+\rho^{n-k-1}\varepsilon
_{k+2}\right) \right\rangle \left\langle X_{i},C^{\dag}\left( \varepsilon
_{n+1}+...+\rho^{n-i-1}\varepsilon_{i+2}\right) \right\rangle \\
& =\mathbb{E}\left\langle \varepsilon_{k+1},\varepsilon_{i+1}\right\rangle
\left\langle X_{k},C^{\dag}\left(
\varepsilon_{n+1}+...+\rho^{n-k-1}\varepsilon_{k+2}\right) \right\rangle
\left\langle X_{i},C^{\dag}\left(
\varepsilon_{n+1}+...+\rho^{n-k-1}\varepsilon_{k+2}\right) \right\rangle \\
& +\mathbb{E}\left\langle \varepsilon_{k+1},\varepsilon_{i+1}\right\rangle
\left\langle X_{k},C^{\dag}\left(
\varepsilon_{n+1}+...+\rho^{n-k-1}\varepsilon_{k+2}\right) \right\rangle
\left\langle X_{i},C^{\dag}\left(
\rho^{n-k}\varepsilon_{k+1}+...+\rho^{n-i-1}\varepsilon_{i+2}\right)
\right\rangle \\
& =0
\end{align*}
because $\varepsilon_{k+1}$ is independent of all the rest in the first term
and $\left( \varepsilon_{n+1}+...+\rho^{n-k-1}\varepsilon_{k+2}\right) $ is
independent of all the rest in the second. Hence :
\begin{equation*}
\mathbb{E}\left\Vert \sum_{k=1}^{n-1}Z_{k,n}^{+}\right\Vert ^{2}=\sum
_{k=1}^{n-1}\mathbb{E}\left\Vert Z_{k,n}^{+}\right\Vert ^{2}
\end{equation*}
with :
\begin{align*}
\mathbb{E}\left\Vert Z_{k,n}^{+}\right\Vert ^{2} & =\mathbb{E}\left\Vert
\varepsilon_{k+1}\right\Vert ^{2}\mathbb{E}\left\langle X_{k},C^{\dag}\left(
\varepsilon_{n+1}+...+\rho^{n-k-1}\varepsilon_{k+2}\right) \right\rangle ^{2}
\\
& =\mathbb{E}\left\Vert \varepsilon_{1}\right\Vert ^{2}\mathbb{E}\left\Vert
C^{1/2}C^{\dag}\left( \varepsilon_{n+1}+...+\rho^{n-k-1}\varepsilon
_{k+2}\right) \right\Vert ^{2}
\end{align*}
after conditioning with respect to $\varepsilon_{n+1}+...+\rho^{n-k-1}
\varepsilon_{k+2}$. By independence :
\begin{align*}
\mathbb{E}\left\Vert C^{1/2}C^{\dag}\left( \varepsilon_{n+1}+...+\rho
^{n-k-1}\varepsilon_{k+2}\right) \right\Vert ^{2} & =\mathbb{E}\left\Vert
C^{1/2}C^{\dag}\varepsilon_{n+1}\right\Vert ^{2}+...+\mathbb{E}\left\Vert
C^{1/2}C^{\dag}\rho^{n-k-1}\varepsilon_{k+2}\right\Vert ^{2} \\
& =\left\Vert C_{\varepsilon}^{1/2}C^{1/2}C^{\dag}\right\Vert
_{2}^{2}+...+\left\Vert
C_{\varepsilon}^{1/2}C^{1/2}C^{\dag}\rho^{n-k-1}\right\Vert _{2}^{2} \\
& =\left\Vert C_{\varepsilon}^{1/2}C^{1/2}C^{\dag}\right\Vert _{2}^{2}+r_{v}
\end{align*}
with $r_{v}=\sum_{k=1}^{n-k-1}\left\Vert C_{\varepsilon}^{1/2}C^{1/2}C^{\dag
}\rho^{n-k-1}\right\Vert _{2}^{2}$ and possibly $r_{v}=0$ when $n-k-1=0$.
Observe also that for all $j\geq1$ and $\left\Vert CC^{\dag}\right\Vert
_{\infty}\leq1$ :
\begin{align*}
\left\Vert C_{\varepsilon}^{1/2}C^{1/2}C^{\dag}\rho^{j}\right\Vert _{2}^{2}
& \leq\left\Vert C_{\varepsilon}^{1/2}\right\Vert _{2}^{2}\left\Vert
C^{1/2}C^{\dag}\rho^{j}\right\Vert _{\infty}^{2} \\
& \leq\mathrm{tr}C_{\varepsilon}\left\Vert C^{-1/2}\rho\right\Vert _{\infty
}^{2}\left\Vert \rho\right\Vert
_{\infty}^{2j-2}=\sigma_{\varepsilon}^{2}\left\Vert C^{-1/2}\rho\right\Vert
_{\infty}^{2}a_{1}^{2j-2}
\end{align*}
hence $r_{v}\leq\frac{\sigma_{\varepsilon}^{2}}{1-a_{1}^{2}} (a'_2)^{2}$. At last :
\begin{equation*}
\left\Vert C_{\varepsilon}^{1/2}C^{1/2}C^{\dag}\right\Vert _{2}^{2}=\mathrm{
tr}\left( C^{\dag}C^{1/2}C_{\varepsilon}C^{1/2}C^{\dag}\right)
=\sum_{p=1}^{+\infty}\frac{\lambda_{p}}{\left( \lambda_{p}+\alpha\right) ^{2}
}\left\langle C_{\varepsilon}e_{p},e_{p}\right\rangle.
\end{equation*}
This finishes the proof of the first Lemma.

\begin{lemma}
\label{plusmoins}We have :
\begin{equation*}
\left\vert \sum_{1\leq j,k\leq n-1}\mathbb{E}\left\langle
Z_{j,n}^{+},Z_{k,n}^{-}\right\rangle \right\vert \leq 2\left\Vert C_{\varepsilon}^{2}\right\Vert _{1} (a'_2)^2 \frac{a_{1}^{2}}{\left(
1-a_{1}^{2}\right) ^{2}}.
\end{equation*}
\end{lemma}

\textbf{Proof of Lemma \ref{plusmoins} :} Start with :
\begin{equation*}
\sum_{1\leq j,k\leq n-1}\left\langle Z_{j,n}^{+},Z_{k,n}^{-}\right\rangle
=\sum_{1\leq k\leq n-1}\left\langle Z_{k,n}^{+},Z_{k,n}^{-}\right\rangle
+\sum_{1\leq k<j\leq n-1}\left\langle Z_{j,n}^{+},Z_{k,n}^{-}\right\rangle
+\sum_{1\leq j<k\leq n-1}\left\langle Z_{j,n}^{+},Z_{k,n}^{-}\right\rangle .
\end{equation*}

Go on with :
\begin{equation*}
\mathbb{E}\left\langle Z_{k,n}^{+},Z_{k,n}^{-}\right\rangle =\mathbb{E}\left[
\left\Vert \varepsilon_{k+1}\right\Vert ^{2}\left\langle
X_{k},C^{\dag}\left(
\varepsilon_{n+1}+...+\rho^{n-k-1}\varepsilon_{k+2}\right) \right\rangle
\left\langle X_{k},C^{\dag}\rho^{n-k}X_{k+1}\right\rangle \right] =0
\end{equation*}
due to $\varepsilon_{n+1}+...+\rho^{n-k-1}\varepsilon_{k+2}$. Take $k<j$
then $\mathbb{E}\left\langle Z_{j,n}^{+},Z_{k,n}^{-}\right\rangle =0$ is
clear by independence of $\varepsilon_{j+1}$ from the rest. Keep in mind
that from now on until the end of the proof : $k>j$. The remaining term is :
\begin{equation*}
\sum_{1\leq j<k\leq n-1}\left\langle Z_{j,n}^{+},Z_{k,n}^{-}\right\rangle
=\sum_{1\leq j<k\leq n-1}\left\langle \varepsilon_{k+1},\varepsilon
_{j+1}\right\rangle \left\langle
X_{k},C^{\dag}\rho^{n-k}X_{k+1}\right\rangle \left\langle
X_{j},C^{\dag}\left(
\varepsilon_{n+1}+...+\rho^{n-j-1}\varepsilon_{j+2}\right) \right\rangle .
\end{equation*}

We have to split the previous term in different pieces in order to compute
it more accurately. Denote :
\begin{align*}
\mathbf{B}_{1} & =\left\langle
\varepsilon_{k+1},\varepsilon_{j+1}\right\rangle \left\langle
X_{k},C^{\dag}\rho^{n-k+1}X_{k}\right\rangle \quad\mathbf{B}
_{2}=\left\langle \varepsilon_{k+1},\varepsilon_{j+1}\right\rangle
\left\langle X_{k},C^{\dag}\rho^{n-k}\varepsilon_{k+1}\right\rangle \\
\mathbf{B}_{3} & =\left\langle X_{j},C^{\dag}\rho^{n-k}\varepsilon
_{k+1}\right\rangle \quad\mathbf{B}_{4}=\left\langle X_{j},C^{\dag}\left(
\rho^{n-k+1}\varepsilon_{k}+...+\rho^{n-j-1}\varepsilon_{j+2}\right)
\right\rangle
\end{align*}
where the sequence $C^{\dag}\left(
\varepsilon_{n+1}+...+\rho^{n-k-1}\varepsilon_{k+2}\right) $ was dropped for
obvious reasons in the second line above. The four $\mathbf{B}_{p}$'s all depend on $k$ and $j$
although these indices were dropped for clarity in the notations. We have :
\begin{equation*}
\sum_{1\leq j<k\leq n-1}\mathbb{E}\left\langle
Z_{j,n}^{+},Z_{k,n}^{-}\right\rangle =\sum_{1\leq j<k\leq n-1}\mathbb{E}
\left[ \left( \mathbf{B}_{1}+\mathbf{B}_{2}\right) \left( \mathbf{B}_{3}+
\mathbf{B}_{4}\right) \right] .
\end{equation*}
Notice that when $k=j+1$ then $\mathbf{B}_{4}$ does not exist. Otherwise $
\mathbb{E}\left( \mathbf{B}_{1}\mathbf{B}_{4}\right) =0$ due to the single $
\varepsilon_{k+1}$. In order to alleviate the notations set $
u_{k,j}=\varepsilon_{k}+...+\rho^{k-j-1}\varepsilon_{j+1}$. Remind that $
X_{k}=u_{k,j}+\rho^{k-j}X_{j}$. We go on with :
\begin{align*}
\mathbb{E}\left( \mathbf{B}_{2}\mathbf{B}_{3}\right) & =\mathbb{E}
\left\langle \varepsilon_{k+1},\varepsilon_{j+1}\right\rangle \left\langle
u_{k,j}+\rho^{k-j}X_{j},C^{\dag}\rho^{n-k}\varepsilon_{k+1}\right\rangle
\left\langle X_{j},C^{\dag}\rho^{n-k}\varepsilon_{k+1}\right\rangle \\
& =\mathbb{E}\left\langle \varepsilon_{k+1},\varepsilon_{j+1}\right\rangle
\left\langle \rho^{k-j-1}\varepsilon_{j+1}+\rho^{k-j}X_{j},C^{\dag}\rho
^{n-k}\varepsilon_{k+1}\right\rangle \left\langle X_{j},C^{\dag}\rho
^{n-k}\varepsilon_{k+1}\right\rangle \\
& =\mathbb{E}\left\langle \varepsilon_{k+1},\varepsilon_{j+1}\right\rangle
\left\langle
\rho^{k-j}X_{j},C^{\dag}\rho^{n-k}\varepsilon_{k+1}\right\rangle
\left\langle X_{j},C^{\dag}\rho^{n-k}\varepsilon_{k+1}\right\rangle \\
& +\mathbb{E}\left\langle \varepsilon_{k+1},\varepsilon_{j+1}\right\rangle
\left\langle \rho^{k-j-1}\varepsilon_{j+1},C^{\dag}\rho^{n-k}\varepsilon
_{k+1}\right\rangle \left\langle X_{j},C^{\dag}\rho^{n-k}\varepsilon
_{k+1}\right\rangle
\end{align*}
and $\mathbb{E}\left( \mathbf{B}_{2}\mathbf{B}_{3}\right) =0$ again because $
\varepsilon_{j+1}$ and $X_{j}$ are centered. A development similar ot the
one above shows that $\mathbb{E}\left( \mathbf{B}_{2}\mathbf{B}_{4}\right)
=0 $ for the same reasons. Our last concern is :
\begin{align*}
\mathbb{E}\left( \mathbf{B}_{1}\mathbf{B}_{3}\right) & =\mathbb{E}
\left\langle \varepsilon_{k+1},\varepsilon_{j+1}\right\rangle \left\langle
X_{j},C^{\dag}\rho^{n-k}\varepsilon_{k+1}\right\rangle \left\langle
X_{k},C^{\dag}\rho^{n-k+1}X_{k}\right\rangle \\
& =\mathbb{E}\left\langle \varepsilon_{k+1},\varepsilon_{j+1}\right\rangle
\left\langle X_{j},C^{\dag}\rho^{n-k}\varepsilon_{k+1}\right\rangle
\left\langle u_{k,j}+\rho^{k-j}X_{j},C^{\dag}\rho^{n-k+1}\left(
u_{k,j}+\rho^{k-j}X_{j}\right) \right\rangle .
\end{align*}

After some tedious calculations the latter reduces to :
\begin{equation*}
\mathbb{E}\left( \mathbf{B}_{1}\mathbf{B}_{3}\right) =\mathbb{E}\left\langle
\varepsilon_{k+1},\varepsilon_{j+1}\right\rangle \left\langle X_{j},C^{\dag
}\rho^{n-k}\varepsilon_{k+1}\right\rangle \left( \left\langle \rho
^{k-j-1}\varepsilon_{j+1},C^{\dag}\rho^{n-j+1}X_{j}\right\rangle
+\left\langle
\rho^{k-j}X_{j},C^{\dag}\rho^{n-j}\varepsilon_{j+1}\right\rangle \right) .
\end{equation*}
A conditioning in cascade in the previous expectation then leads to :
\begin{align*}
\mathbb{E}\left( \mathbf{B}_{1}\mathbf{B}_{3}\right) & =\mathbb{E}
\left\langle C_{\varepsilon}\left( \rho^{\ast}\right)
^{n-k}C^{\dag}X_{j},\varepsilon_{j+1}\right\rangle \left( \left\langle
\varepsilon _{j+1},\left( \rho^{\ast}\right)
^{k-j-1}C^{\dag}\rho^{n-j+1}X_{j}\right\rangle +\left\langle \left(
\rho^{\ast}\right)
^{n-j}C^{\dag}\rho^{k-j}X_{j},\varepsilon_{j+1}\right\rangle \right) \\
& =\mathbb{E}\left\langle C_{\varepsilon}^{2}\left( \rho^{\ast}\right)
^{n-k}C^{\dag}X_{j},\left( \rho^{\ast}\right) ^{k-j-1}C^{\dag}\rho
^{n-j+1}X_{j}\right\rangle +\mathbb{E}\left\langle C_{\varepsilon}^{2}\left(
\rho^{\ast}\right) ^{n-k}C^{\dag}X_{j},\left( \rho^{\ast}\right)
^{n-j}C^{\dag}\rho^{k-j}X_{j}\right\rangle \\
& =\mathbb{E}\left\langle X_{j},C^{\dag}\rho^{n-k}C_{\varepsilon}^{2}\left(
\rho^{\ast}\right) ^{k-j-1}C^{\dag}\rho^{n-j+1}X_{j}\right\rangle +\mathbb{E}
\left\langle \left( \rho^{\ast}\right) ^{k-j}C^{\dag}\rho
^{n-j}C_{\varepsilon}^{2}\left( \rho^{\ast}\right)
^{n-k}C^{\dag}X_{j},X_{j}\right\rangle \\
& =\mathrm{tr}\left( CC^{\dag}\rho^{n-k}C_{\varepsilon}^{2}\left(
\rho^{\ast}\right) ^{k-j-1}C^{\dag}\rho^{n-j+1}\right) +\mathrm{tr}\left(
CC^{\dag}\rho^{n-k}C_{\varepsilon}^{2}\left( \rho^{\ast}\right)
^{n-j}C^{\dag}\rho^{k-j}\right), \\
\left\vert \mathbb{E}\left( \mathbf{B}_{1}\mathbf{B}_{3}\right) \right\vert
& \leq2\left\Vert \rho\right\Vert _{\infty}^{2n-2j-2}\left\Vert
C_{\varepsilon}^{2}\right\Vert _{1}\left\Vert C^{-1/2}\rho\right\Vert
_{\infty}^{2}.
\end{align*}
The inequality $\mathrm{tr}\left( AB\right) \leq\left\Vert A\right\Vert
_{\infty}\left\Vert B\right\Vert _{1}$ was used for the last line above. Several other properties are implicitely used without proof : $\sup_{\alpha}\left\Vert CC^{\dag}\right\Vert _{\infty}\leq1$ or $
\sup_{\alpha}\left\Vert C^{\dag1/2}\rho\right\Vert _{\infty}=\left\Vert
C^{-1/2}\rho\right\Vert _{\infty}<+\infty$ by assumption $\mathbf{A}_{1}$, etc.

Collecting the computations above yields the desired result :
\begin{align*}
\left\vert \sum_{1\leq j<k\leq n-1}\mathbb{E}\left\langle
Z_{j,n}^{+},Z_{k,n}^{-}\right\rangle \right\vert & \leq2\left\Vert
C_{\varepsilon}^{2}\right\Vert _{1}\left\Vert C^{-1/2}\rho\right\Vert
_{\infty}^{2}\sum_{1\leq j<k\leq n-1}\left\Vert \rho\right\Vert
_{\infty}^{2n-2j-2} \\
& \leq2\left\Vert C_{\varepsilon}^{2}\right\Vert _{1} (a'_2)^{2}\frac{a_{1}^{2}}{\left(
1-a_{1}^{2}\right) ^{2}}.
\end{align*}

\begin{lemma}
\label{moinsmoins}Under assumptions $\mathbf{A}_{1-3}$ :
$\mathbb{E}\left\Vert \sum_{k=1}^{n-1}Z_{k,n}^{-}\right\Vert ^{2}\leq \Theta
_{1}+m\Theta _{2}$ with :
\begin{eqnarray*}
\Theta _{1} &=&\frac{\sigma _{\varepsilon }^{2}\left\Vert \rho
^{2}\right\Vert _{1}+\left( a_{2}^{\prime }\right) ^{2}\mathbb{E}\left\Vert
\varepsilon _{1}\right\Vert ^{4}}{1-a_{1}^{2}}+\frac{2\left( a_{2}^{\prime
}\right) ^{2}\left\Vert C_{\varepsilon }^{2}\right\Vert _{1}}{\left(
1-a_{1}^{2}\right) ^{2}}\left( a_{1}^{2}+\left\Vert \rho ^{2}\right\Vert
_{1}\right),  \\
\Theta _{2} &=&2\mathbb{E}\left\Vert \varepsilon _{1}\right\Vert ^{3}\frac{%
1+\eta }{\eta }\left( \xi _{\mathbf{1}}\xi _{\mathbf{2}}\mathrm{tr}C\right)
^{1/2}\frac{a_{1}}{1-a_{1}^{2}}\left( a_{2}^{\prime }\right) ^{2}.
\end{eqnarray*}
\end{lemma}

\textbf{Proof of Lemma \ref{moinsmoins} :} Starting from :$Z_{k,n}^{-}=\varepsilon
_{k+1}\left\langle X_{k},C^{\dag }\rho ^{n-k}X_{k+1}\right\rangle $, 
\begin{equation*}
Z_{k,n}^{-}=\varepsilon _{k+1}\left\langle X_{k},C^{\dag }\rho
^{n-k+1}X_{k}\right\rangle +\varepsilon _{k+1}\left\langle X_{k},C^{\dag
}\rho ^{n-k}\varepsilon _{k+1}\right\rangle 
\end{equation*}%
computations similar to the one carried out in the previous lemmas show that
: 
\begin{align}
\mathbb{E}\left\Vert \sum_{k=1}^{n-1}Z_{k,n}^{-}\right\Vert ^{2}& =\mathbb{E}%
\left\Vert \sum_{k=1}^{n-1}\varepsilon _{k+1}\left\langle X_{k},C^{\dag
}\rho ^{n-k+1}X_{k}\right\rangle \right\Vert ^{2}+\mathbb{E}\left\Vert
\sum_{k=1}^{n-1}\varepsilon _{k+1}\left\langle X_{k},C^{\dag }\rho
^{n-k}\varepsilon _{k+1}\right\rangle \right\Vert ^{2}  \label{pok} \\
& +2\sum_{1\leq j,k\leq n-1}\mathbb{E}\left\langle \varepsilon
_{k+1},\varepsilon _{j+1}\right\rangle \left\langle X_{k},C^{\dag }\rho
^{n-k}\varepsilon _{k+1}\right\rangle \left\langle X_{j},C^{\dag }\rho
^{n-j+1}X_{j}\right\rangle   \notag
\end{align}%
A noticeable fact will be used below : we know that $\rho $ is necessarily
Hilbert-Schmidt under the assumption $\left\Vert C^{-1/2}\rho \right\Vert
_{\infty }<+\infty $. This entails that $\rho ^{2}$ is trace class hence
that $\left\Vert \rho ^{2}\right\Vert _{1}<+\infty $. We can start with :%
\begin{align*}
\mathbb{E}\left\Vert \sum_{k=1}^{n-1}\varepsilon _{k+1}\left\langle
X_{k},C^{\dag }\rho ^{n-k+1}X_{k}\right\rangle \right\Vert ^{2}&
=\sum_{k=1}^{n-1}\mathbb{E}\left\Vert \varepsilon _{k+1}\right\Vert ^{2}%
\mathbb{E}\left\langle X_{k},C^{\dag }\rho ^{n-k+1}X_{k}\right\rangle ^{2} \\
& =\mathbb{E}\left\Vert \varepsilon _{1}\right\Vert ^{2}\sum_{k=1}^{n-1}%
\left[ \mathrm{tr}\left( CC^{\dag }\rho ^{n-k+1}\right) \right] ^{2}
\end{align*}

with $\left\vert \mathrm{tr}\left( CC^{\dag }\rho ^{n-k+1}\right)
\right\vert \leq \left\Vert CC^{\dag }\right\Vert _{\infty }\left\Vert \rho
^{2}\right\Vert _{1}\left\Vert \rho \right\Vert _{\infty }^{2\left(
n-k-1\right) }$ and :%
\begin{equation*}
\mathbb{E}\left\Vert \sum_{k=1}^{n-1}\varepsilon _{k+1}\left\langle
X_{k},C^{\dag }\rho ^{n-k+1}X_{k}\right\rangle \right\Vert ^{2}\leq \frac{%
\sigma _{\varepsilon }^{2}\left\Vert \rho ^{2}\right\Vert _{1}}{1-a_{1}^{2}}
\end{equation*}

Now we turn to : 
\begin{align}
\mathbb{E}\left\Vert \sum_{k=1}^{n-1}\varepsilon_{k+1}\left\langle
X_{k},C^{\dag}\rho^{n-k}\varepsilon_{k+1}\right\rangle \right\Vert ^{2} & = 
\mathbb{E}\sum_{k=1}^{n-1}\left\Vert \varepsilon_{k+1}\left\langle
X_{k},C^{\dag}\rho^{n-k}\varepsilon_{k+1}\right\rangle \right\Vert ^{2} 
\notag \\
& +2\sum_{1\leq j<k\leq n-1}\mathbb{E}\left\langle
\varepsilon_{k+1},\varepsilon_{j+1}\right\rangle \left\langle
X_{k},C^{\dag}\rho^{n-k}\varepsilon_{k+1}\right\rangle \left\langle
X_{j},C^{\dag}\rho^{n-j}\varepsilon_{j+1}\right\rangle  \label{pok2}
\end{align}

where 
\begin{align*}
\mathbb{E}\left\Vert \varepsilon _{k+1}\right\Vert ^{2}\left\langle
X_{k},C^{\dag }\rho ^{n-k}\varepsilon _{k+1}\right\rangle ^{2}& =\mathbb{E}%
\left\Vert \varepsilon _{1}\right\Vert ^{2}\left\langle \varepsilon
_{1},\left( \rho ^{\ast }\right) ^{n-k}C^{\dag }CC^{\dag }\rho
^{n-k}\varepsilon _{1}\right\rangle  \\
& \leq \mathbb{E}\left\Vert \varepsilon _{1}\right\Vert ^{4}\left\Vert
C^{\dag /2}CC^{\dag /2}\right\Vert _{\infty }\left\Vert C^{\dag /2}\rho
\right\Vert _{\infty }^{2}\left\Vert \rho \right\Vert _{\infty }^{2n-2k-2}.
\end{align*}%
Then : 
\begin{align*}
\sum_{k=1}^{n-1}\mathbb{E}\left\Vert \varepsilon _{k+1}\right\Vert
^{2}\left\langle X_{k},C^{\dag }\rho ^{n-k}\varepsilon _{k+1}\right\rangle
^{2}& \leq \mathbb{E}\left\Vert \varepsilon _{1}\right\Vert ^{4}\left\Vert
C^{\dag /2}\rho \right\Vert _{\infty }^{2}\sum_{k=1}^{n-1}\left\Vert \rho
\right\Vert _{\infty }^{2n-2k-2} \\
& \leq \mathbb{E}\left\Vert \varepsilon _{1}\right\Vert ^{4}\frac{\left(
a_{2}^{\prime }\right) ^{2}}{1-a_{1}^{2}}.
\end{align*}

Notice that the operator $C^{\dag/2}CC^{\dag/2}$ is bounded in norm for all
values of $\alpha$ : $\sup_{\alpha}\left\Vert C^{\dag1/2}CC^{\dag
1/2}\right\Vert _{\infty}\leq1$.

The cross terms in (\ref{pok2}) are : 
\begin{align*}
& \sum_{1\leq j<k\leq n-1}\mathbb{E}\left\langle \varepsilon
_{k+1},\varepsilon _{j+1}\right\rangle \left\langle X_{k},C^{\dag }\rho
^{n-k}\varepsilon _{k+1}\right\rangle \left\langle X_{j},C^{\dag }\rho
^{n-j}\varepsilon _{j+1}\right\rangle  \\
& =\sum_{1\leq j<k\leq n-1}\mathbb{E}\left\langle C_{\varepsilon }\left(
\rho ^{\ast }\right) ^{n-k}C^{\dag }X_{k},\varepsilon _{j+1}\right\rangle
\left\langle X_{j},C^{\dag }\rho ^{n-j}\varepsilon _{j+1}\right\rangle  \\
& =\sum_{1\leq j<k\leq n-1}\mathbb{E}\left\langle C_{\varepsilon }\left(
\rho ^{\ast }\right) ^{n-k}C^{\dag }\rho ^{k-j}X_{j},\varepsilon
_{j+1}\right\rangle \left\langle \left( \rho ^{\ast }\right) ^{n-j}C^{\dag
}X_{j},\varepsilon _{j+1}\right\rangle  \\
& =\sum_{1\leq j<k\leq n-1}\mathbb{E}\left\langle C^{\dag }\rho
^{n-j}C_{\varepsilon }^{2}\left( \rho ^{\ast }\right) ^{n-k}C^{\dag }\rho
^{k-j}X_{j},X_{j}\right\rangle =\sum_{1\leq j<k\leq n-1}\mathrm{tr}\left(
CC^{\dag }\rho ^{n-j}C_{\varepsilon }^{2}\left( \rho ^{\ast }\right)
^{n-k}C^{\dag }\rho ^{k-j}\right)  \\
& \leq \left\Vert CC^{\dag }\right\Vert _{\infty }\left\Vert C^{\dag /2}\rho
\right\Vert _{\infty }^{2}\left\Vert C_{\varepsilon }^{2}\right\Vert
_{1}\sum_{1\leq j<k\leq n-1}\left\Vert \left( \rho ^{\ast }\right)
^{n-k-1}\right\Vert _{\infty }\left\Vert \rho ^{k-j-1}\right\Vert _{\infty
}\left\Vert \rho \right\Vert _{\infty }^{n-j} \\
& \leq \left\Vert C_{\varepsilon }^{2}\right\Vert _{1}\left\Vert C^{\dag
/2}\rho \right\Vert _{\infty }^{2}\sum_{1\leq j<k\leq
n-1}a_{1}^{2n-2j-2}\leq \left\Vert C_{\varepsilon }^{2}\right\Vert
_{1}\left\Vert C^{\dag /2}\rho \right\Vert _{\infty }^{2}\frac{a_{1}^{2}}{%
\left( 1-a_{1}^{2}\right) ^{2}}.
\end{align*}%
Collecting the last two bounds we get : 
\begin{equation*}
\mathbb{E}\left\Vert \sum_{k=1}^{n-1}\varepsilon _{k+1}\left\langle
X_{k},C^{\dag }\rho ^{n-k}\varepsilon _{k+1}\right\rangle \right\Vert
^{2}\leq \frac{\left( a_{2}^{\prime }\right) ^{2}}{1-a_{1}^{2}}\left( \frac{%
2a_{1}^{2}}{1-a_{1}^{2}}\left\Vert C_{\varepsilon }^{2}\right\Vert _{1}+%
\mathbb{E}\left\Vert \varepsilon _{1}\right\Vert ^{4}\right) .
\end{equation*}

We are ready to attack the last part of this Lemma with the double sum in (%
\ref{pok}) which is dominant. This double sum is split in three parts.

Clearly $\sum_{1\leq k<j\leq n-1}\mathbb{E}\left\langle \varepsilon
_{k+1},\varepsilon _{j+1}\right\rangle \left\langle X_{k},C^{\dag }\rho
^{n-k}\varepsilon _{k+1}\right\rangle \left\langle X_{j},C^{\dag }\rho
^{n-j+1}X_{j}\right\rangle =0$. Then : 
\begin{align*}
& \sum_{1\leq j<k\leq n-1}\mathbb{E}\left\langle \varepsilon
_{k+1},\varepsilon _{j+1}\right\rangle \left\langle X_{k},C^{\dag }\rho
^{n-k}\varepsilon _{k+1}\right\rangle \left\langle X_{j},C^{\dag }\rho
^{n-j+1}X_{j}\right\rangle  \\
& =\sum_{1\leq j<k\leq n-1}\mathbb{E}\left\langle \varepsilon
_{k+1},\varepsilon _{j+1}\right\rangle \left\langle \rho ^{k-j}\left(
X_{j}\right) ,C^{\dag }\rho ^{n-k}\varepsilon _{k+1}\right\rangle
\left\langle X_{j},C^{\dag }\rho ^{n-j+1}X_{j}\right\rangle  \\
& +\sum_{1\leq j<k\leq n-1}\mathbb{E}\left\langle \varepsilon
_{k+1},\varepsilon _{j+1}\right\rangle \left\langle u_{k,j},C^{\dag }\rho
^{n-k}\varepsilon _{k+1}\right\rangle \left\langle X_{j},C^{\dag }\rho
^{n-j+1}X_{j}\right\rangle 
\end{align*}%
where it is recalled that $u_{k,j}=\varepsilon _{k}+...+\rho
^{k-j-1}\varepsilon _{j+1}$.The first term just above is null because of $%
\varepsilon _{j+1}$, the second rewrites : 
\begin{align*}
& \sum_{1\leq j<k\leq n-1}\mathbb{E}\left[ \left\langle \varepsilon
_{k+1},\varepsilon _{j+1}\right\rangle \left\langle \rho ^{k-j-1}\varepsilon
_{j+1},C^{\dag }\rho ^{n-k}\varepsilon _{k+1}\right\rangle \right] \mathbb{E}%
\left\langle X_{j},C^{\dag }\rho ^{n-j+1}X_{j}\right\rangle  \\
& =\sum_{1\leq j<k\leq n-1}\mathbb{E}\left\langle C_{\varepsilon
}\varepsilon _{j+1},\left( \rho ^{n-k}\right) ^{\ast }C^{\dag }\rho
^{k-j-1}\varepsilon _{j+1}\right\rangle \mathrm{tr}\left( CC^{\dag }\rho
^{n-j+1}\right)  \\
& =\sum_{1\leq j<k\leq n-1}\mathrm{tr}\left( C_{\varepsilon }^{2}\left( \rho
^{n-k}\right) ^{\ast }C^{\dag }\rho ^{k-j-1}\right) \mathrm{tr}\left(
CC^{\dag }\rho ^{n-j+1}\right) 
\end{align*}

With $\left\vert \mathrm{tr}\left( CC^{\dag }\rho ^{n-j+1}\right)
\right\vert \leq \left\Vert CC^{\dag }\right\Vert _{\infty }\left\Vert \rho
^{2}\right\Vert _{1}\left\Vert \rho ^{n-j-1}\right\Vert _{\infty }$ and :%
\begin{align*}
& \sum_{1\leq j<k\leq n-1}\mathbb{E}\left[ \left\langle \varepsilon
_{k+1},\varepsilon _{j+1}\right\rangle \left\langle \rho ^{k-j-1}\varepsilon
_{j+1},C^{\dag }\rho ^{n-k}\varepsilon _{k+1}\right\rangle \right] \mathbb{E}%
\left\langle X_{j},C^{\dag }\rho ^{n-j+1}X_{j}\right\rangle  \\
& \leq \left\Vert C^{\dag /2}\rho \right\Vert _{\infty }^{2}\left\Vert
C_{\varepsilon }^{2}\right\Vert _{1}\left\Vert \rho ^{2}\right\Vert
_{1}\sum_{1\leq j<k\leq n-1}\left\Vert \rho ^{n-j-1}\right\Vert _{\infty
}\left\Vert \rho \right\Vert _{\infty }^{k-j-2}\left\Vert \rho ^{\ast
}\right\Vert _{\infty }^{n-k-1} \\
& \leq \left\Vert C^{\dag /2}\rho \right\Vert _{\infty }^{2}\left\Vert
C_{\varepsilon }^{2}\right\Vert _{1}\left\Vert \rho ^{2}\right\Vert
_{1}\sum_{1\leq j<k\leq n-1}\left( a_{1}^{2}\right) ^{n-j-2}\leq \frac{%
\left\Vert C^{\dag /2}\rho \right\Vert _{\infty }^{2}\left\Vert
C_{\varepsilon }^{2}\right\Vert _{1}\left\Vert \rho ^{2}\right\Vert _{1}}{%
\left( 1-a_{1}^{2}\right) ^{2}}
\end{align*}

At last the diagonal term in the double sum (\ref{pok}) is the largest term
containing the (inner) product of three similar random elements. It writes : 
\begin{align*}
& \sum_{1\leq k\leq n-1}\mathbb{E}\left[ \left\Vert \varepsilon
_{k+1}\right\Vert ^{2}\left\langle X_{k},C^{\dag }\rho ^{n-k}\varepsilon
_{k+1}\right\rangle \left\langle X_{k},C^{\dag }\rho
^{n-k+1}X_{k}\right\rangle \right]  \\
& \leq \left\Vert C^{\dag 1/2}\rho \right\Vert _{\infty }^{2}\mathbb{E}%
\left\Vert \varepsilon _{1}\right\Vert ^{3}\sum_{1\leq k\leq n-1}\left\Vert
\rho \right\Vert _{\infty }^{n-k-1}\left\Vert \rho \right\Vert _{\infty
}^{n-k}\mathbb{E}\left[ \left\Vert C^{\dag 1/2}X_{k}\right\Vert
^{2}\left\Vert X_{k}\right\Vert \right]  \\
& =\left\Vert C^{\dag 1/2}\rho \right\Vert _{\infty }^{2}\mathbb{E}%
\left\Vert \varepsilon _{1}\right\Vert ^{3}\mathbb{E}\left[ \left\Vert
C^{\dag 1/2}X_{1}\right\Vert ^{2}\left\Vert X_{1}\right\Vert \right]
\sum_{1\leq k\leq n-1}\left\Vert \rho \right\Vert _{\infty }^{2n-2k-1} \\
& \leq \mathbb{E}\left\Vert \varepsilon _{1}\right\Vert ^{3}\mathbb{E}\left[
\left\Vert C^{\dag 1/2}X_{1}\right\Vert ^{2}\left\Vert X_{1}\right\Vert %
\right] \left( a_{2}^{\prime }\right) ^{2}\frac{a_{1}}{1-a_{1}^{2}}.
\end{align*}

By Cauchy-Schwartz inequality and Lemma \ref{lem.conc.mom} :%
\begin{equation*}
\mathbb{E}\left[ \left\Vert C^{\dag 1/2}X_{1}\right\Vert ^{2}\left\Vert
X_{1}\right\Vert \right] \leq \left[ \mathbb{E}\left\Vert C^{\dag
1/2}X_{1}\right\Vert ^{4}\mathbb{E}\left\Vert X_{1}\right\Vert ^{2}\right]
^{1/2}\leq m\frac{1+\eta }{\eta }\left( \xi _{\mathbf{1}}\xi _{\mathbf{2}}%
\mathrm{tr}C\right) ^{1/2}
\end{equation*}%
hence%
\begin{equation*}
\sum_{1\leq k\leq n-1}\mathbb{E}\left[ \left\Vert \varepsilon
_{k+1}\right\Vert ^{2}\left\langle X_{k},C^{\dag }\rho ^{n-k}\varepsilon
_{k+1}\right\rangle \left\langle X_{k},C^{\dag }\rho
^{n-k+1}X_{k}\right\rangle \right] \leq \mathbb{E}\left\Vert \varepsilon
_{1}\right\Vert ^{3}m\frac{1+\eta }{\eta }\left( \xi _{\mathbf{1}}\xi _{%
\mathbf{2}}\mathrm{tr}C\right) ^{1/2}\left( a_{2}^{\prime }\right) ^{2}\frac{%
a_{1}}{1-a_{1}^{2}}.
\end{equation*}

Collecting the results of the computations above we get $
\mathbb{E}\left\Vert \sum_{k=1}^{n-1}Z_{k,n}^{-}\right\Vert ^{2}\leq \Theta
_{1}+m\Theta _{2}$ with : \\ $\Theta _{1} =\frac{\sigma _{\varepsilon }^{2}\left\Vert \rho
^{2}\right\Vert _{1}+\left( a_{2}^{\prime }\right) ^{2}\mathbb{E}\left\Vert
\varepsilon _{1}\right\Vert ^{4}}{1-a_{1}^{2}}+\frac{2\left( a_{2}^{\prime
}\right) ^{2}\left\Vert C_{\varepsilon }^{2}\right\Vert _{1}}{\left(
1-a_{1}^{2}\right) ^{2}}\left( a_{1}^{2}+\left\Vert \rho ^{2}\right\Vert
_{1}\right)$ and $\Theta _{2} =2\mathbb{E}\left\Vert \varepsilon _{1}\right\Vert ^{3}\frac{%
1+\eta }{\eta }\left( \xi _{\mathbf{1}}\xi _{\mathbf{2}}\mathrm{tr}C\right)
^{1/2}\frac{a_{1}}{1-a_{1}^{2}}\left( a_{2}^{\prime }\right) ^{2}$.

All this finishes the proof of the Lemma.

\subsubsection{Residual term\label{residual-meansq-risk-section}}

The goal of this subsection is to prove Proposition \ref{prop-up-bpund-res}.
The next Lemma gives a Markov type concentration inequality.

\begin{lemma}
\label{P-barre}When $c_{n}=\left( nm\right) ^{-1/6}$ the following bound is
granted for all $n$ :
\begin{equation*}
\mathbb{P}\left( \overline{\mathcal{E}}_{n}\right) \leq\left[ \mathbf{c}
_{s}^{\ast}+\xi_{2\mathbf{s}}\left( \frac{1+\eta}{\eta}\right) ^{s}\right]
\left( \frac{m^{2}}{n}\right) ^{s/3}
\end{equation*}
where $\mathbf{c}_{s}^{\ast}$ is defined in Proposition \ref{conc.mom}.
\end{lemma}

\textbf{Proof of the Lemma :} From

\[\left\{ \left\Vert R_{n,\alpha}\right\Vert
_{\infty}<c_{n}\right\} \cap\left\{ \left\Vert
C^{\dagger1/2}X_{n+1}\right\Vert <1/c_{n}\right\} \subset\left\{ \left\Vert
R_{n,\alpha}\right\Vert _{\infty }<c_{n}\right\} \cap\left\{ \left\Vert
R_{n,\alpha}C^{\dagger1/2}X_{n+1}\right\Vert <1\right\} \]

 we obtain :
\begin{equation*}
\mathbb{P}\left( \overline{\mathcal{E}}_{n}\right) \leq\mathbb{P}\left(
\left\Vert R_{n,\alpha}\right\Vert _{\infty}>c_{n}\right) +\mathbb{P}\left(
\left\Vert C^{\dagger1/2}X_{n+1}\right\Vert >1/c_{n}\right) ,
\end{equation*}
then :
\begin{equation*}
\mathbb{P}\left( \left\Vert R_{n,\alpha}\right\Vert _{\infty}>c_{n}\right)
\leq\frac{\mathbf{c}_{s}^{\ast}}{c_{n}^{\mathbf{s}}}\left( \frac{m}{n}
\right) ^{\mathbf{s}/2},\quad\mathbb{P}\left( \left\Vert C^{\dagger
1/2}X_{n+1}\right\Vert >\frac{1}{c_{n}}\right) \leq\xi_{\mathbf{s}
}c_{n}^{2s}\left( \frac{1+\eta}{\eta}m\right) ^{s}
\end{equation*}
after applying Markov's inequality and Proposition \ref{conc.mom} or Lemma 
\ref{lem.conc.mom} respectively for the two bounds above. Tuning $c_{n}$ in
order to balance both probabilities above then a good choice is defined by :
\begin{equation*}
\frac{\mathbf{c}_{s}^{\ast}}{c_{n}^{\mathbf{s}}}\left( \frac{m}{n}\right) ^{
\mathbf{s}/2}\approx\xi_{\mathbf{s}}c_{n}^{2\mathbf{s}}\left( \frac{1+\eta }{
\eta}m\right) ^{s}.
\end{equation*}
Letting $c_n=\left( nm\right) ^{-1/6}$ gives : 
\begin{equation*}
\mathbb{P}\left( \left\Vert R_{n,\alpha}\right\Vert _{\infty}>c\right) \leq
\mathbf{c}_{s}^{\ast}\left( \frac{m^{2}}{n}\right) ^{s/3}\quad \mathbb{P}
\left( \left\Vert C^{\dagger1/2}X_{n+1}\right\Vert >\frac{1}{c}\right)
\leq\xi_{\mathbf{s}}\left( \frac{m^{2}}{n}\right) ^{s/3}\left( \frac{1+\eta}{
\eta}\right) ^{s}
\end{equation*}
hence the result.

We can turn to the Proposition.

\begin{proposition}
\label{prop-up-bpund-res}Remind formula (\ref{biais-var}) introducing $
\mathbf{V}_{n}$ and $\mathbf{B}_{n}$. Assume that $\alpha\leq\lambda_{1}$ and $m/(n \lambda_m) \leq 1$
then for some constant $C_*$ explicitely given in the proof below :
\begin{equation}
\mathcal{R}_{n}=\mathbb{E}\left[ \left\Vert \left( \widehat{\rho}
_{n}-\rho\right) X_{n+1}\right\Vert ^{2}1_{\overline{\mathcal{E}}_{n}}\right]
\leq C_* \left( \mathbf{V}_{n}+\mathbf{B}_{n}\right) m \left( \frac{m^{2}}{n}\right)
^{\left( s-2\right) /3}.  \label{up-bound-residual}
\end{equation}
\end{proposition}

\textbf{Proof :} It takes three short steps to derive the Proposition. The
expectation in $\mathcal{R}_{n}$ is evaluated by a control of $\left\Vert
\left( \widehat{\rho }_{n}-\rho \right) X_{n+1}\right\Vert ^{2}$ by Lemmas 
\ref{lem.conc.mom} and \ref{P-barre} and the concentration inequality given
in Proposition \ref{conc.mom} which is crucial to evaluate $\mathbb{P}\left( 
\overline{\mathcal{E}}_{n}\right) $ under the moment assumptions considered
on $X$.

\textbf{First step :} From $\widehat{\rho }_{n}-\rho =\rho
C_{n}C_{n}^{\dagger }-\rho +U_{n}C_{n}^{\dag }$ we always have : 
\begin{equation*}
\left( \widehat{\rho }_{n}-\rho \right) X_{n+1}=-\alpha \rho C_{n}^{\dagger
}X_{n+1}+U_{n}C_{n}^{\dagger }X_{n+1}=-\alpha \rho C^{\dagger
/2}G_{n}C^{\dagger /2}X_{n+1}+U_{n}C^{\dagger /2}G_{n}C^{\dagger /2}X_{n+1}
\end{equation*}
with $G_{n}=\left( C+\alpha\mathbf I\right) ^{1/2}C_{n}^{\dagger }\left( C+\alpha\mathbf I\right) ^{1/2}$ then :
\begin{equation}
\left\Vert \left( \widehat{\rho }_{n}-\rho \right) X_{n+1}\right\Vert
^{2}\leq \left\Vert C^{\dagger /2}X_{n+1}\right\Vert ^{2}\left\Vert
G_{n}\right\Vert ^{2}\left\Vert U_{n}C^{\dagger /2}-\alpha \rho C^{\dagger
/2}\right\Vert _{\infty }^{2}.  \label{res.bound.1}
\end{equation}

\textbf{Second step :} An intermediate result about $G_{n}$ is needed. First, 
from the definition of $G_n$ and the fact that $\|C_n^\dag\|_\infty\leq \alpha^{-1}$we have, $\left\Vert
G_{n}\right\Vert _{\infty }\leq \alpha ^{-1}\left\Vert C+\alpha\mathbf I\right\Vert
_{\infty }\leq \left( \lambda _{1}+\alpha \right) /\alpha $ almost surely so
that $\mathbb{E}\left\Vert G_{n}\right\Vert_\infty ^{M}<+\infty $ for all $M>0.$
The problem is that the previous almost sure upper bound on $\left\Vert
G_{n}\right\Vert _{\infty }$ is not sharp. On the other hand simple
computations show that $G_{n}=I+G_{n}R_{n,\alpha }$ hence for all $
\varepsilon \in \left( 0,1\right) $ :
\begin{equation}\label{eq:maxGn}
\left\Vert G_{n}\right\Vert _{\infty }\leq 1+\left\Vert G_{n}\right\Vert
_{\infty }\left\Vert R_{n,\alpha }\right\Vert _{\infty }\leq 1+\left\Vert
G_{n}\right\Vert _{\infty }^{1-\varepsilon }\left( \frac{\lambda _{1}+\alpha 
}{\alpha }\right) ^{\varepsilon }\left\Vert R_{n,\alpha }\right\Vert
_{\infty }.
\end{equation}
For simplicity set temporarily $x=\left\Vert G_{n}\right\Vert _{\infty }$, $
r_{n}=\left\Vert R_{n,\alpha }\right\Vert _{\infty }/\alpha ^{\varepsilon }$
and $c=\left( \lambda _{1}+\alpha \right) ^{\varepsilon }$ the formula above
rewrites~: $x\leq 1+cr_{n}x^{1-\varepsilon }$. A simple study of the
function $\chi \left( x\right) =1-x+cr_{n}x^{1-\varepsilon }$ for $x>0$
shows that $\chi $ is concave, increases, reaches a maximum at $x^{\ast }=
\left(cr_{n}\left( 1-\varepsilon \right) \right) ^{1/\varepsilon }$ then
decreases. Consider $\widetilde{x}=\left[ cr_{n}\right] ^{1/\varepsilon
}>x^{\ast }$. Now by concavity $\chi $ is below its tangent at $\widetilde{x}
$ and this tangent has a negative slope of $-\varepsilon $ meaning that 
$$
\chi(x)\leq \chi(\tilde x)+\chi'(\tilde x)(x-\tilde x) = \chi(\tilde x)-\varepsilon(x-\tilde x). 
$$
Combining with the condition $\chi(x) \geq 0$ deduced from~\eqref{eq:maxGn} we have
$x\leq (cr_{n})^{1/\varepsilon }+\frac{1}{\varepsilon }$. So that
we get as a consequence for all $\varepsilon \in \left( 0,1\right) $ :
\begin{equation*}
\left\Vert G_{n}\right\Vert _{\infty }\leq \frac{\lambda _{1}+\alpha }{
\alpha }\left\Vert R_{n,\alpha }\right\Vert _{\infty }^{1/\varepsilon }+
\frac{1}{\varepsilon }.
\end{equation*}

\textbf{Third step :} Plugging this bound in (\ref{res.bound.1}) gives :
\begin{eqnarray}
\left\Vert \left( \widehat{\rho }_{n}-\rho \right) X_{n+1}\right\Vert ^{2}
&\leq &\left\Vert C^{\dagger /2}X_{n+1}\right\Vert ^{2}\left\Vert
U_{n}C^{\dagger /2}-\alpha \rho C^{\dagger /2}\right\Vert _{\infty }^{2}
\left[ \frac{\lambda _{1}+\alpha }{\alpha }\left\Vert R_{n,\alpha
}\right\Vert _{\infty }^{1/\varepsilon }+\frac{1}{\varepsilon }\right] ^{2} 
\notag \\
&\leq &2\left\Vert C^{\dagger /2}X_{n+1}\right\Vert ^{2}\left\Vert
U_{n}C^{\dagger /2}-\alpha \rho C^{\dagger /2}\right\Vert _{\infty }^{2}
\left[ \frac{\left( \lambda _{1}+\alpha \right) ^{2}}{\alpha ^{2}}\left\Vert
R_{n,\alpha }\right\Vert _{\infty }^{2/\varepsilon }+\frac{1}{\varepsilon
^{2}}\right]   \label{res.bound.2}
\end{eqnarray}
which is split in two. Below we start with the most technical part of the
bound above including $\left\Vert R_{n,\alpha }\right\Vert $. A generalized H\"{o}lder inequality gives :
\begin{eqnarray*}
&&\frac{\left( \lambda _{1}+\alpha \right) ^{2}}{\alpha ^{2}}\mathbb{E}\left[
\left\Vert C^{\dagger /2}X_{n+1}\right\Vert ^{2}\left\Vert U_{n}C^{\dagger
/2}-\alpha \rho C^{\dagger /2}\right\Vert _{\infty }^{2}\left\Vert
R_{n,\alpha }\right\Vert _{\infty }^{2/\varepsilon }1_{\overline{\mathcal{E}}
_{n}}\right]  \\
&\leq &\frac{\left( \lambda _{1}+\alpha \right) ^{2}}{\alpha ^{2}}\left[ 
\mathbb{E}\left\Vert C^{\dagger /2}X_{n+1}\right\Vert ^{2s}\right] ^{1/s}
\left[ \mathbb{E}\left\Vert U_{n}C^{\dagger /2}-\alpha \rho C^{\dagger
/2}\right\Vert _{\infty }^{2s}\right] ^{1/s}\left[ \mathbb{E}\left\Vert
R_{n,\alpha }\right\Vert _{\infty }^{s}\right] ^{2/(s\varepsilon )}\mathbb{P}
^{\left( s-2-2/\varepsilon \right) /s}\left( \overline{\mathcal{E}}
_{n}\right).
\end{eqnarray*}

Consider Lemma \ref{lem.conc.mom} below, then $\left[ \mathbb{E}\left\Vert
C^{\dagger /2}X_{n+1}\right\Vert ^{2s}\right] ^{1/s}\leq \xi _{s}\frac{
1+\eta }{\eta }m$.

The next term includes $\mathbb{E}\left\Vert
U_{n}C^{\dagger /2}-\alpha \rho C^{\dagger /2}\right\Vert _{\infty }^{2s}$.
Remind that $m$ and $\alpha $ are still connected through (\ref{alpha.star})
with :
\begin{equation*}
\mathbb{E}\left\Vert U_{n}C^{\dagger /2}-\alpha \rho C^{\dagger
/2}\right\Vert _{2}^{2s}\leq 2^{s}\mathbb{E}\left( \left\Vert
U_{n}C^{\dagger /2}\right\Vert _{2}^{2}+\left\Vert \alpha \rho C^{\dagger
/2}\right\Vert _{2}^{2}\right) ^{s}=2^{s}\sum_{k=0}^{s}{\binom{s}{k}}
\left\Vert \alpha \rho C^{\dagger /2}\right\Vert _{2}^{2\left( s-k\right) }
\mathbb{E}\left\Vert U_{n}C^{\dagger /2}\right\Vert _{2}^{2k}.
\end{equation*}
By Corollary \ref{corr-conc-U_n} below for all $k\in \left\{ 0,...,s\right\} $ : 
$\mathbb{E}\left\Vert U_{n}C^{\dagger /2}\right\Vert _{2}^{2k}\leq \mathcal{C}
_{s}\mathbf{V}_{n}^{k}$ hence : 
\begin{equation*}
\mathbb{E}\left\Vert U_{n}C^{\dagger /2}-\alpha \rho C^{\dagger
/2}\right\Vert _{2}^{2s}\leq 2^{s}\mathcal{C}_{s}\sum_{k=0}^{s}{\binom{s}{k}}
\mathbf{B}_{n}^{2\left( s-k\right) }\mathbf{V}_{n}^{k}=2^{s}\mathcal{C}
_{s}\left( \mathbf{V}_{n}+\mathbf{B}_{n}\right) ^{s}.
\end{equation*}

Then take advantage of the bound given in Proposition \ref{conc.mom} : $
\mathbb{E}\left\Vert R_{n,\alpha }\right\Vert _{\infty }^{\mathbf{s}}\leq 
\mathbf{c}_{s}^{\ast }\left( m/n\right) ^{\mathbf{s}/2}$ and finally of the
previous Lemma \ref{P-barre} showing that $\mathbb{P}\left( \overline{
\mathcal{E}}_{n}\right) \leq \overline{K}_{s,\eta }\left( \frac{m^{2}}{n}
\right) ^{s/3}$. All this gives :
\begin{eqnarray*}
&&\frac{\left( \lambda _{1}+\alpha \right) ^{2}}{\alpha ^{2}}\mathbb{E}
\left\Vert C^{\dagger /2}X_{n+1}\right\Vert ^{2}\left\Vert U_{n}C^{\dagger
/2}-\alpha \rho C^{\dagger /2}\right\Vert _{\infty }^{2}\left\Vert
R_{n,\alpha }\right\Vert _{\infty }^{2/\varepsilon }1_{\overline{\mathcal{E}}
_{n}} \\
&\leq &2\left( \overline{K}_{s,\eta }\right) ^{\left( s-2-2/\varepsilon
\right) /s}\mathcal{C}_{s}^{1/s}\left( \lambda _{1}+\alpha \right) ^{2}\xi
_{s}\frac{1+\eta }{\eta }\left( \mathbf{c}_{s}^{\ast }\right)
^{2/(s\varepsilon )}\left( \mathbf{V}_{n}+\mathbf{B}_{n}\right) \frac{m}{
\alpha ^{2}}\left( \frac{m}{n}\right) ^{1/\varepsilon }\left( \frac{m^{2}}{n}
\right) ^{\left( s-2-2/\varepsilon \right) /3} \\
&=&K^{\ast }\left( \mathbf{V}_{n}+\mathbf{B}_{n}\right) \frac{m}{\alpha ^{2}}
\left( \frac{m}{n}\right) ^{1/\varepsilon }\left( \frac{m^{2}}{n}\right)
^{\left( s-2-2/\varepsilon \right) /3}.
\end{eqnarray*}

Notice that for $\varepsilon =1/6$ the bound above turns out to be : $
K^{\ast }\left( \mathbf{V}_{n}+\mathbf{B}_{n}\right) m\left( \frac{m^{2}}{n}
\right) ^{\left( s-2\right) /3}/\left( nm\alpha \right) ^{2}\leq K^{\ast
}\left( \mathbf{V}_{n}+\mathbf{B}_{n}\right) m\left( \frac{m^{2}}{n}\right)
^{\left( s-2\right) /3}/m^{4}$ under the assumption $m/\left( n\lambda
_{m}\right) \leq 1$. On the other hand with the same technique as above: 
\begin{equation*}
\mathbb{E}\left[ \left\Vert C^{\dagger /2}X_{n+1}\right\Vert ^{2}\left\Vert
U_{n}C^{\dagger /2}-\alpha \rho C^{\dagger /2}\right\Vert _{\infty }^{2}1_{
\overline{\mathcal{E}}_{n}}\right] \leq 2\left( \overline{K}_{s,\eta
}\right) ^{1-2/s}\xi _{s}\mathcal{C}_{s}^{1/s}\frac{1+\eta }{\eta }\left( 
\mathbf{V}_{n}+\mathbf{B}_{n}\right) m\left( \frac{m^{2}}{n}\right) ^{\left(
s-2\right) /3}.
\end{equation*}

From all that was done previously we conclude from (\ref{res.bound.2}) that :
\begin{equation*}
\mathbb{E}\left[ \left\Vert \left( \widehat{\rho }_{n}-\rho \right)
X_{n+1}\right\Vert ^{2}1_{\overline{\mathcal{E}}_{n}}\right] \leq 2\left(
\left( \overline{K}_{s,\eta }\right) ^{1-2/s}\xi _{s}\mathcal{C}_{s}^{1/s}
\frac{1+\eta }{\eta }+\frac{K^{\ast }}{m^{4}}\right) m\left( \frac{m^{2}}{n}
\right) ^{\left( s-2\right) /3}.
\end{equation*}
\bigskip 

This is the announced result when constant $C_*$ was set for instance to $2\left(
\left( \overline{K}_{s,\eta }\right) ^{1-2/s}\xi _{s}\mathcal{C}_{s}^{1/s}
\frac{1+\eta }{\eta }+K^{\ast }\right)$.
\bigskip

\textbf{Proof of Theorem \ref{minimaxUpperBound} :}

By successive applications of Lemma \ref{series} we have first $\mathbf{V}
_{n}\leq \sigma _{\varepsilon }^{2}\frac{m}{n}\left( 1+1/\eta \right) $ then
:
\begin{equation*}
\mathbf{B}_{n}\leq \alpha ^{2}\left( \sum_{p=1}^{m}\frac{\left\Vert \rho
e_{p}\right\Vert ^{2}}{\lambda _{p}}+\frac{1}{\alpha }\sum_{p\geq
m+1}\left\Vert \rho e_{p}\right\Vert ^{2}\right) \leq 2\alpha ^{2}\tau
\left( m\right) .
\end{equation*}
Now we have to check that the residual term involving $\mathfrak{R}_{n}$ in
Theorem \ref{risk-upper-bound} has no impact on the risk. Clearly the only
concern is about the second bound with $\mathcal{R}''_{n}$. We have to make sure that $m\left(
m^{2}/n\right) ^{\left( \mathbf{s}-2\right) /3}$ is bounded and this will be enough.  From the two assumptions : 
$m/\left( n\lambda _{m}\right) <1$ (required to get a lower bound) and $
\lambda _{m}\leq m^{-1-\eta }$ for some $\eta >0$,  we get $m\left( m^{2}/n\right) ^{\left( 
\mathbf{s}-2\right) /3}<m^{1-\eta \left( \mathbf{s}-2\right) /3}$, which
should be bounded. This comes down to $\eta \geq 3/\left( s-2\right) 
$. Then~:
\begin{equation*}
\mathfrak{R}_{n}\leq \frac{m}{n}\left( \mathcal{R}_{n}^{\prime} +\mathcal{R}_{s}^{\prime
\prime }\sigma _{\varepsilon }^{2}\left( 1+1/\eta \right) \right) +2\alpha
^{2}\tau \left( m\right) \mathcal{R}_{s}^{\prime \prime }.
\end{equation*}

Plugging this bound in (\ref{risk-main}) leads to the upper bound given in
the Theorem.

\subsection{Proof of Proposition \ref{prop.reg-var.bias} :}

From Definition (\ref{thau}) the bound $\mathcal{B}_{1}$ (resp $\mathcal{B}
_{2}$) rewrite :
\begin{equation}
\frac{\int_{x}^{+\infty }\varrho \left( s\right) ds}{\varrho \left( x\right) 
}\leq M\frac{g\left( x\right) }{g^{\prime }\left( x\right) },\quad \mathrm{
resp.}\frac{g\left( x\right) }{g^{\prime }\left( x\right) }\leq M^{\prime }
\frac{\int_{x}^{+\infty }\varrho \left( s\right) ds}{\varrho \left( x\right) 
}  \label{grana-bounds}
\end{equation}
The derivation follows from extensive reference to \cite{Granata2016a}, denotes [GI] for short and to  \cite{Granata2016b} denoted 
[GII]. The exact references are systematically provided with
equation numbers in [GI-II]. For cases $a.1$, $a.2$ and $a.3$ the left hand
side of $\int_{x}^{+\infty }\varrho \left( s\right) ds/\varrho \left(
x\right) $ tends to a constant [GII (8.27)]. The function $g$ is then $\mathbf{EV}_{c}
$ (cases $a.2$ and $a.3$) or possibly $\mathbf{EV}_{0}$ (case $a.1$ if $\lambda $ and 
$\rho $ have the same exponential index) : see [GII (9.1)] then $g\left(
x\right) /g^{\prime }\left( x\right) $ necessarily tends to a constant when $
g\in \mathbf{EV}_{c}$ [GII - Prop 8.4] or to infinity otherwise [GII (8.12)].

Now we examine cases $c.1-3$. When $\rho \in \mathbf{RV}_{\alpha }$, $
\int_{x}^{+\infty }\varrho \left( s\right) ds/\varrho \left( x\right) $
tends to infinity [GI-(2.79)] whereas $g(x)/g^{\prime }(x)$ tends to a
constant when $\lambda \in \mathbf{EV}_{c}$ (case $c.1$). When $\lambda \in \mathbf{EV}_{0}$
note that $\int_{x}^{+\infty }\varrho \left( s\right) ds/(x\varrho (x))$
tends to a constant [GI-(2.79)] whereas $g\left( x\right) /(xg^{\prime
}\left( x\right) )$ tends to $0$ [GI-(2.82)]. So when $c.1$ or $c.2$ hold, $
\mathcal{B}_{2}$ holds as well. For $c.3$ both sides tend to constants when
divided by $x$ hence the conclusion.

For the last line $b.1$ is easy : $g(x)/g^{\prime }(x)$ tends to a constant
(see above) whereas $\int_{x}^{+\infty }\varrho \left( s\right) ds/\varrho
\left( x\right) $ tends to infinity and $\mathcal{B}_{1}$ holds strictly.
Case $b.3$ is granted by [GII-(10.4)] (Take $f=\rho /\lambda $ and $
g=\lambda $ in equation (10.4) by Granata). The last problem remaining is $
b.2$. It turns out that here nothing consistent can be concluded as shown on
the two following examples. First take $\varrho \left( x\right) =\lambda
\left( x\right) =2^{-1}x^{-1/2}\exp \left( x^{-1/2}\right) $ then $\mathcal{B
}_{1}$ holds for an accurate choice of $M$. Now take $\varrho \left( x\right)
=2x^{-1}\log x\exp \left( -\log ^{2}x\right) $ and $\lambda \left( x\right)
=2x^{1/2}\exp \left( -x^{-1/2}\right) \varrho \left( x\right) $ then $
g(x)/g^{\prime }(x)=2x^{1/2}$ and $\int_{x}^{+\infty }\varrho \left(
s\right) ds/\varrho \left( x\right) =x/\left( 2\log x\right) $ and $\mathcal{
B}_{2}$ holds.

\section{Appendix}

\subsection{Proof of the variant of Feldman-Hayek Theorem (Proposition~\ref{varFeldHay})}
\label{subsec:ProofvarFeldHay}
Let $\nu$ and $\mu$ be two measures on a Hilbert space $\mathbf{H}$ and
assume that $\mu$ is absolutely continuous with respect to $\nu$. The
Kullback-Leibler divergence from $\mu$ to $\nu$ is very classically defined
as $\mathbf{KL}\left( \nu||\mu\right) =\int\log\left( \frac{d\mu}{d\nu }
\right) d\nu$.


The proof follows the lines of Theorem 2.23 in \cite{DPZ} with a slight
adaptation. First the issue of $m_{1}$ and $m_{2}$ may be tackled similarly
to \cite{DPZ} since the main difference in this variant relies on the
covariance operators only. Here too the shift $m_{1}-m_{2}$ is admissible if
and only if it belongs to $H_{0}$. We will suppose that $m_{1}=m_{2}=0$.

Let $\mathbf{h}$ denote the Hellinger affinity between two probability
measures, $\mathbf{h}\left( \mu,\nu\right) =\int\left( \frac{d\mu}{d\lambda}
\frac{d\nu}{d\lambda}\right) ^{1/2}d\lambda$ where $\lambda$ is a reference
probability measure and both $\mu$ and $\nu$ are absolutely continuous w.r.t. $\lambda$. Then the previous integral does not depend on $\lambda$. The
point is : $\mathbf{h}\left( \mu,\nu\right) =0$ if and only if $\mu$ and $
\nu $ are singular (see Proposition 2.19 p.59 in \cite{DPZ}).

\textbf{Step 1 : Sufficiency.} The goal is to compute $\mathbf{h}\left(
\mu,\nu\right) $ and to prove that $\mathbf{h}\left( \mu,\nu\right) >0$
under condition (i) and (iii). This will provide sufficiency.

Notice that $S_{1}$ and $S_{2}$ are Hilbert-Schmidt because $
\Sigma_{0}^{1/2} $ is. Let $R=\left( S_{1}^{-1}S_{2}\right) \left(
S_{1}^{-1}S_{2}\right) ^{\ast}$, when (i) holds the closed graph theorem
ensures that $S_{1}^{-1}S_{2}$ is bounded and invertible from and onto $\overline{H_0}$, and $R$ too. Since $S_{1}$ and $S_{2}$ are compact the
eigenvalues of $R$ are denumerable, since $R$ is symmetric and positive they
are nonnegative. Denote $\left( \tau_{j}\right) _{j\in\mathbb{N}}$ these
eigenvalues and $\left( \phi_{j}\right) _{j\in\mathbb{N}}$ their associated
eigenvectors. Clearly by invertibility of $R$ : for all $j$ $
0<\tau_{\ast}\leq\tau_{j}\leq\tau^{\ast}.$

Let $\xi_{j}$ be a sequence of real i.i.d $\mathcal{N}\left( 0,1\right) $
random variables and denote now :
\begin{equation*}
U=\sum_{j=1}^{+\infty}\xi_{j}S_{1}\phi_{j}\quad V=\sum_{j=1}^{+\infty}\sqrt{
\tau_{j}}\xi_{j}S_{1}\phi_{j}.
\end{equation*}
First $\sum_{j=1}^{+\infty}\left\Vert S_{1}\phi_{j}\right\Vert ^{2}<+\infty$
because $S_{1}$ is Hilbert Schmidt. The 0-1 law implies that $U$ is a
well-defined square integrable random element in $\mathbf{H}$. Finite
dimensional distributions taken from $U$ are all centered Gaussian, hence $U$
is Gaussian and its covariance operator is fully determined by quadratic
forms :
\begin{equation*}
\mathbb{E}\left\langle U,x\right\rangle \left\langle U,y\right\rangle
=\sum_{j=1}^{+\infty}\left\langle \phi_{j},S_{1}^{\ast}x\right\rangle
\left\langle \phi_{j},S_{1}^{\ast}y\right\rangle =\sum_{j=1}^{+\infty
}\left\langle S_{1}S_{1}^{\ast}x,y\right\rangle =\left\langle \Sigma
_{1}x,y\right\rangle .
\end{equation*}
Hence $U\sim\mu$. Similarly :
\begin{align*}
\mathbb{E}\left\langle V,x\right\rangle \left\langle V,y\right\rangle &
=\sum_{j=1}^{+\infty}\tau_{j}\left\langle S_{1}\phi_{j},x\right\rangle
\left\langle \phi_{j},S_{1}^{\ast}y\right\rangle =\sum_{j=1}^{+\infty
}\left\langle S_{1}R\phi_{j},x\right\rangle \left\langle
\phi_{j},S_{1}^{\ast }y\right\rangle =\sum_{j=1}^{+\infty}\left\langle
S_{2}\left( S_{1}^{-1}S_{2}\right) ^{\ast}\phi_{j},x\right\rangle
\left\langle \phi_{j},S_{1}^{\ast}y\right\rangle \\
& =\sum_{j=1}^{+\infty}\left\langle \phi_{j},S_{1}^{-1}S_{2}S_{2}^{\ast
}x\right\rangle \left\langle \phi_{j},S_{1}^{\ast}y\right\rangle
=\left\langle S_{1}^{-1}S_{2}S_{2}^{\ast}x,S_{1}^{\ast}y\right\rangle
=\left\langle \Sigma_{2}x,y\right\rangle
\end{align*}
and $V\sim\nu$. Considering the series representations of $U$ and $V$ we are
likely to apply Proposition 2.20 p.60 in \cite{DPZ} about product measures
on $\mathbb{R}^{\infty}$ and their Hellinger integrals. Denote here $\mu_{j}$
(resp $\nu_{j}$) the distribution of $\xi_{j}$ (resp. $\sqrt{\tau_{j}}
\xi_{j} $). Classical calculations give $\mathbf{h}\left(
\mu_{j},\nu_{j}\right) =\sqrt{2\sqrt{\tau_{j}}/\left( 1+\tau_{j}\right) }.$
Turning to $\log\left[ \Pi_{j\in\mathbb{N}}\mathbf{h}\left(
\mu_{j},\nu_{j}\right) \right] =\left( 1/2\right) \sum_{j\in\mathbb{N}}\log
\frac{2\sqrt{\tau_{j}}}{1+\tau_{j}}$ set $\delta_{j}=1-\tau_{j}$ then :
\begin{equation*}
\log\frac{\left( 1+\tau_{j}\right) ^{2}}{4\tau_{j}}=\log\frac{\left(
2-\delta_{j}\right) ^{2}}{4-4\delta_{j}}=\log\left( 1+\frac{\delta_{j}^{2}}{
4-4\delta_{j}}\right) \leq\frac{\delta_{j}^{2}}{4-4\delta_{j}}\leq \frac{
\delta_{j}^{2}}{4\tau_{\ast}}.
\end{equation*}
Finally when $\sum_{j}\left( 1-\tau_{j}\right) ^{2}<+\infty$ $\Pi _{j\in
\mathbb{N}}\mathbf{h}\left( \mu_{j},\nu_{j}\right) >0$ and $\mu\sim \nu$.
But the condition $\sum_{j}\left( 1-\tau_{j}\right) ^{2}$ means that the
eigenvalues of $R-I$ are square summable hence that $R-I$ is
Hilbert-Schmidt. This proves the first part of the Theorem.

\textbf{Step 2 : Necessity. }Here again we follow the proof of \cite{DPZ}.
Suppose that $\mu$ and $\nu$ are non-singular. In order to prove that (i)
holds everything works by replacing $Q_{1}^{1/2}$ (resp. $Q_{2}^{1/2}$) by $
S_{1}$ (resp. $S_{2}$) in the proof of Theorem 2.23 of \cite{DPZ} p.64. In
order to prove that (iii) holds an adaptation is required in the proof. We
give some details in order to sweep a possible typo appearing in \cite{DPZ}
p.66. The goal here is to prove that $\sum_{j}\left( 1-\tau_{j}\right)
^{2}<+\infty$ where again $\left( \tau_{j}\right) _{j\in\mathbb{N}}$ is the
eigenvalue sequence of $R=$ $\left( S_{1}^{-1}S_{2}\right) \left(
S_{1}^{-1}S_{2}\right) ^{\ast}$. Consider the singular value decomposition
of $S_{1}=\sum_{j}s_{1j}\phi_{j}\left\langle \varphi_{j},.\right\rangle $
then $S_{1}^{-1}=\sum_{j}\frac{1}{s_{1j}}\varphi_{j}\left\langle \phi
_{j},.\right\rangle $ and $\left( S_{1}^{-1}\right) ^{\ast}=\sum_{j}\frac {1
}{s_{1j}}\phi_{j}\left\langle \varphi_{j},.\right\rangle $. Consequently :
\begin{equation*}
\left\langle R\varphi_{i},\varphi_{j}\right\rangle =\frac{\left\langle
\Sigma_{2}\phi_{i},\phi_{j}\right\rangle }{s_{1j}s_{1i}}=r_{i,j}=r_{j,i}.
\end{equation*}
Denote $R_{n}$ the positive symmetric square matrix of size $n$ with $\left[
R_{n}\right] _{i,j}=r_{i,j}$. The matrix $R_{n}$ is positive definite
whenever $\ker\Sigma_{2}=\left\{ 0\right\} $.

Let $G_{1}\sim\mathcal{N}\left( 0,\Sigma_{1}\right) $ and $G_{1}\sim\mathcal{
N}\left( 0,\Sigma_{2}\right) $ and consider the $n$-dimensional centered
Gaussian vectors : $X_{n}=\left( x_{1},...,x_{n}\right) $ and $Y_{n}=\left(
y_{1},...,y_{n}\right) $ with :
\begin{equation*}
x_{j}=\frac{\left\langle G_{1},\phi_{j}\right\rangle }{s_{1j}},\quad y_{j}=
\frac{\left\langle G_{2},\phi_{j}\right\rangle }{s_{1j}}.
\end{equation*}
We have $X_{n}\sim\mathcal{N}\left( 0,I_{n}\right) $ and $Y_{n}\sim\mathcal{N
}\left( 0,R_{n}\right) $ because $\mathbb{E}x_{i}x_{j}=0$ if $i\neq j$, $
\mathbb{E}x_{i}^{2}=1$ and $\mathbb{E}y_{i}y_{j}=r_{i,j}$. A standard result
for multivariate Gaussian distributions gives $\mathbf{h}\left(
X_{n},Y_{n}\right) =\det^{1/4}R_{n}/\det^{1/2}\left[ \left(
R_{n}+I_{n}\right) /2\right] $. Everything is ready for the final steps.
Since $\mu$ and $\nu$ are equivalent $\mathbf{h}\left( \mu,\nu\right) >0$
besides Proposition 2.19 (iii) in \cite{DPZ} (or Theorem 1 in \cite
{vanErven2014}) gives for all $n$ :
\begin{equation*}
\mathbf{h}\left( X_{n},Y_{n}\right) \geq\mathbf{h}\left( \mu,\nu\right) >0
\end{equation*}
but denoting $\left( \tau_{jn}\right) _{1\leq j\leq n}$ the eigenvalues of $
R_{n}$, $\log\mathbf{h}\left( X_{n},Y_{n}\right) =\sum_{j=1}^{n}\log \frac{
\sqrt{2}\tau_{jn}^{1/4}}{\left( 1+\tau_{jn}\right) ^{1/2}}$ . It was seen
above that for all $n$ and $j$ $\tau_{jn}>0$ hence :
\begin{equation*}
-\log\mathbf{h}\left( X_{n},Y_{n}\right) =\left( 1/4\right)
\sum_{j=1}^{n}\log\frac{\left( 1+\tau_{jn}\right) ^{2}}{4\tau_{jn}}\leq-\log 
\mathbf{h}\left( \mu,\nu\right) .
\end{equation*}

From $\left( 1+\tau_{jn}\right) ^{2}\geq4\tau_{jn}$ we get that the array in
the equation above converges when $n$ tends to infinity so that $\left(
1+\tau_{jn}\right) ^{2}/4\tau_{jn}$ tends to $1$ and we may pick a constant $
c_{\tau}>1$ such that $\tau_{jn}\in\left[ 1/c_{\tau},c_{\tau}\right] $ for
all $j$ and $n$. It is elementary to see that $\left( 1-\tau_{jn}\right)
^{2}/\log\frac{\left( 1+\tau_{jn}\right) ^{2}}{4\tau_{jn}}\leq c_{\tau
}^{\prime}$ for $\tau_{jn}\in\left[ 1/c_{\tau},c_{\tau}\right] $ hence the
bound
\begin{equation*}
\sup_{n}\sum_{j=1}^{n}\left( 1-\tau_{jn}\right) ^{2}\leq-c_{\tau}^{\prime
}\log\mathbf{h}\left( \mu,\nu\right) .
\end{equation*}
The conclusion is the same as \cite{DPZ} : $\sum_{j=1}^{n}\left( 1-\tau
_{jn}\right) ^{2}$ is the Hilbert-Schmidt norm of $I_{n}-R_{n}=\Pi_{n}\left(
I-R\right) \Pi_{n}$ with $\Pi_{n}$ the projection on $\left( \varphi
_{1},...,\varphi_{n}\right) $. The bound above implies that $\left\Vert
I-R\right\Vert _{2}$ is finite. This finishes the necessity part of the
Theorem. We turn to equation (\ref{KL}).

\textbf{Step 3 :} Assume now that $R_{1}=I$ and that $R=
\Sigma_{0}^{-1/2}S_{2}S_{2}^{\ast}\Sigma_{0}^{-1/2}-I$ is trace-class. Then
we can apply Theorem 2 in \cite{Minh2020} so that formula (\ref{KL}) holds.
This concludes the proof of Theorem \ref{varFeldHay}.

\subsection{Concentration inequalities\label{section-conc-ineq}}

\begin{lemma}
\label{lem.conc.mom}With $m=m_{\alpha}=\min\left\{
p:\lambda_{p}\leq\alpha\right\} $ the following bounds hold for $p\leq
\mathbf{s}$ :
\begin{align*}
\mathbb{E}\left\Vert C^{\dagger1/2}X_{1}\right\Vert ^{2p} & \leq\xi
_{p}\left( \frac{1+\eta}{\eta}m\right) ^{p},\quad\mathbb{E}\left\Vert
C^{\dagger1/2}\varepsilon_{1}\right\Vert ^{2p}\leq\xi_{p}^{\prime}\left( 
\frac{1+\eta}{\eta}m\right) ^{p} \\
\mathbb{E}\left[ \left\Vert C^{\dag1/2}X\right\Vert ^{2}\left\Vert
X\right\Vert ^{2}\right] & \leq\xi_{\mathbf{2}}\left( \mathrm{tr}C\right)
\sum_{k=1}^{+\infty}\frac{\lambda_{k}}{\lambda_{k}+\alpha}\quad\mathbb{E}
\left\Vert X_{1}\right\Vert ^{2p}\leq\xi_{p}\left( \mathrm{tr}C\right) ^{p}.
\end{align*}
Here $\xi_{s}^{\prime}$ is a constant which is explicit and given in the
proof.
\end{lemma}

\textbf{Proof of the Lemma :} The bound for $\varepsilon$ on the first line is proved since its
derivation is similar but slightly more general than the one involving $X$. Developing $\varepsilon_{1}$ in the eigenbasis of $C$ we get :
\begin{equation*}
\mathbb{E}\left\Vert C^{\dagger1/2}\varepsilon_{1}\right\Vert ^{2p}=\mathbb{E
}\left[ \left( \sum_{k}\frac{\left\langle \varepsilon_{1},e_{k}\right\rangle ^{2}}{
\lambda_{k}+\alpha}\right) ^{p}\right] =\mathbb{E}\left[ \left( \sum_{k}\frac{\lambda_{k}}{
\lambda_{k}+\alpha}\frac{\left\langle \varepsilon_{1},e_{k}\right\rangle ^{2}
}{\lambda_{k}}\right) ^{p}\right] .
\end{equation*}
H\"{o}lder's (or Jensen's) inequality give :
\begin{equation*}
\left( \sum_{p}\frac{\lambda_{p}}{\lambda_{p}+\alpha}\frac{\left\langle
\varepsilon_{1},e_{p}\right\rangle ^{2}}{\lambda_{p}}\right) ^{p}\leq\left(
\sum_{l}\frac{\lambda_{l}}{\lambda_{l}+\alpha}\right) ^{p-1}\sum_{k}\frac{
\lambda_{k}}{\lambda_{k}+\alpha}\frac{\left\langle
\varepsilon_{1},e_{k}\right\rangle ^{2p}}{\lambda_{k}^{p}}.
\end{equation*}
Assume that $\sup_{k\in\mathbb{N}}\left\{ \mathbb{E}\left\langle
\varepsilon_{1},e_{k}\right\rangle ^{2p}\lambda_{k}^{-p}\right\} =\xi
_{p}^{\prime}<+\infty$. Taking expectation provides :
\begin{equation*}
\mathbb{E}\left\Vert C^{\dagger1/2}\varepsilon_{1}\right\Vert ^{2p}\leq\xi
_{p}^{\prime}\left( \sum_{l}\frac{\lambda_{l}}{\lambda_{l}+\alpha}\right)
^{p}\leq\xi_{p}^{\prime}\left( 1+\frac{1}{\eta}\right) ^{p}m^{p}
\end{equation*}
thanks to the first line of Lemma \ref{series}. The remaining task is to
prove that $\xi_{p}^{\prime}$ is finite as mentioned just above. From $
\varepsilon _{1}=X_{1}-\rho\left( X_{0}\right) $ we obtain :
\begin{align*}
\mathbb{E}\left\langle \varepsilon_{1},e_{k}\right\rangle ^{2p} & \leq 2^{2p}
\left[ \mathbb{E}\left\langle X_{1},e_{k}\right\rangle ^{2p}+\mathbb{E}
\left\langle X_{0},\rho^{\ast}e_{k}\right\rangle ^{2p}\right] \\
& \leq 2^{2p}\lambda_{k}^{p}\left[ \frac{\mathbb{E}\left\langle
X_{1},e_{k}\right\rangle ^{2p}}{\lambda_{k}^{p}}+\mathbb{E}\left\langle
X_{1},\rho^{\ast}C^{-1/2}e_{k}\right\rangle ^{2p}\right]
\leq2^{2p}\lambda_{k}^{p}\left[ \xi_{p}+\left\Vert
C^{-1/2}\rho\right\Vert _{\infty}^{2p}\mathbb{E}\left\Vert
X_{1}\right\Vert ^{2p}\right]
\end{align*}
and we can set $\xi_{p}^{\prime}=2^{2p}\left[ \xi_{p}+\left\Vert C^{-1/2}\rho\right\Vert _{\infty}^{2p}\mathbb{E}\left\Vert
X_{1}\right\Vert ^{2p}\right] $which is the desired result.

For the second line of the Lemma, developing $X$ in its Fourier basis :
\begin{equation*}
\mathbb{E}\left[ \left\Vert C^{\dag1/2}X\right\Vert ^{2}\left\Vert
X\right\Vert ^{2}\right] =\mathbb{E}\left[ \sum_{k,p}\frac{\left\langle
X,e_{k}\right\rangle ^{2}\left\langle X,e_{p}\right\rangle ^{2}}{\lambda
_{k}+\alpha}\right] \leq\left[ \sum_{k,p}\mathbb{E}\left( \varsigma
_{k}\varsigma_{p}\right) \frac{\lambda_{k}\lambda_{p}}{\lambda_{k}+\alpha }
\right]
\end{equation*}
with $\varsigma_{k}=\frac{\left\langle X_{1},e_{k}\right\rangle ^{2}}{
\lambda_{k}}$. Applying Cauchy-Schwartz inequality then assumption $\mathbf{A}_{2}$ with $s=2$ gives :
\begin{equation*}
\mathbb{E}\left[ \left\Vert C^{\dag1/2}X\right\Vert ^{2}\left\Vert
X\right\Vert ^{2}\right] \leq\xi_{\mathbf{2}}\sum_{k,p}\frac{\lambda
_{k}\lambda_{p}}{\lambda_{k}+\alpha}=\xi_{\mathbf{2}}\left( \mathrm{tr}
C\right) \sum_{k}\frac{\lambda_{k}}{\lambda_{k}+\alpha}.
\end{equation*}

The last bound $\mathbb{E}\left\Vert X_{1}\right\Vert ^{2p}\leq\xi_{p}\left( 
\mathrm{tr}C\right) ^{p}$ is obtained with the same technique.

\begin{proposition}
\label{conc.mom}For all $p \in \left[ 2+2/(1+\eta),\mathbf{s}\right] $ and for $\alpha<\alpha_{0}$ where $\alpha_{0}$ is defined in
Proposition \ref{closed-op} there exists a constant $\mathbf{c}_{p}^{\ast}$ given within the proof
such that for all $n$ (with $m/n\leq1$) :
\begin{equation*}
\mathbb{E}\left\Vert C^{\dagger1/2}\left( C-C_{n}\right) C^{\dagger
1/2}\right\Vert _{2}^{p}\leq\mathbf{c}_{p}^{\ast}\left( \frac{m}{n}\right)
^{p/2}.
\end{equation*}
\end{proposition}

\textbf{Proof of the Proposition :} We start with Bosq's trick (see \cite{Bosq2000} Lemma 4.1 p.96). Set 
$Z_{i}=X_{i}\otimes X_{i}-C$ :
\begin{align*}
C_{n}-C & =\rho\left( C_{n}-C\right) \rho^{\ast}-\frac{Z_{n}}{n}+\frac {1}{n}
\sum_{i=1}^{n-1}\mathbf{E}_{i}, \\
\mathbf{E}_{i} & =\rho
X_{i-1}\otimes\varepsilon_{i}+\varepsilon_{i}\otimes\rho
X_{i-1}+\varepsilon_{i}\otimes\varepsilon_{i}-C_{\varepsilon}.
\end{align*}
From $R_{n,\alpha}=C^{\dagger1/2}\left( C_{n}-C\right)
C^{\dagger1/2}$ :
\begin{align*}
R_{n,\alpha} & =C^{\dagger1/2}\rho\left( C_{n}-C\right) \rho^{\ast
}C^{\dagger1/2}-\frac{C^{\dagger1/2}Z_{n}C^{\dagger1/2}}{n}+\frac{1}{n}
\sum_{i=1}^{n-1}C^{\dagger1/2}\mathbf{E}_{i}C^{\dagger1/2} \\
& =\mathcal{F}_{\alpha}\left( R_{n,\alpha}\right) -\frac{C^{\dagger
1/2}Z_{n}C^{\dagger1/2}}{n}+\frac{1}{n}\sum_{i=1}^{n-1}C^{\dagger 1/2}
\mathbf{E}_{i}C^{\dagger1/2}
\end{align*}
where $\mathcal{F}_{\alpha}\left( T\right) =\widetilde{C}_{\alpha }T
\widetilde{C}_{\alpha}^{\ast}$ and $\widetilde{C}_{\alpha}=C^{\dagger
1/2}\rho\left( C+\alpha\mathbf I\right) ^{1/2}$. Notice that $\mathcal{F}_{\alpha} $
is considered here as a linear operator acting on and onto the Hilbert space
of Hilbert-Schmidt operators. We need to bound its sup-norm $\left\Vert 
\mathcal{F}_{\alpha}\right\Vert _{2,\infty}=\sup_{\left\Vert T\right\Vert
_{2}<1}\left\Vert \mathcal{F}_{\alpha}\left( T\right) \right\Vert
_{2}\leq\left\Vert \widetilde{C}_{\alpha}\right\Vert _{\infty}^{2}$ (since $
\left\Vert \widetilde{C}_{\alpha}\right\Vert _{\infty}=\left\Vert \widetilde{
C}_{\alpha}^{\ast}\right\Vert _{\infty}$).

As a consequence from Proposition \ref{closed-op} $\sup_{\alpha<\alpha_{0}}
\left\Vert \widetilde{C}_{\alpha}\right\Vert _{\infty}^{2}<1-t_{0}$ for a $
t_{0}$ in $\left( 0,1\right) $ for all $\alpha<\alpha_{0}$ and $I-\mathcal{F}
_{\alpha}$ is invertible with bounded inverse and $\left\Vert \left( I-
\mathcal{F}_{\alpha}\right) ^{-1}\right\Vert _{2,\infty}\leq\left(
1-\sup_{\alpha<\alpha_{0}}\left\Vert \widetilde{C}_{\alpha}\right\Vert
_{\infty}^{2}\right) ^{-1}=1/t_{0}$ :
\begin{align*}
R_{n,\alpha} & =\left( I-\mathcal{F}_{\alpha}\right) ^{-1}\left[ \frac {1}{n}
\sum_{i=1}^{n-1}C^{\dagger1/2}\mathbf{E}_{i}C^{\dagger1/2}-\frac{
C^{\dagger1/2}Z_{n}C^{\dagger1/2}}{n}\right] , \\
\left\Vert R_{n,\alpha}\right\Vert _{2} & \leq\frac{1}{n}\left\Vert \left( I-
\mathcal{F}_{\alpha}\right) ^{-1}\right\Vert _{2,\infty}\left[ \left\Vert
\sum_{i=1}^{n-1}C^{\dagger1/2}\mathbf{E}_{i}C^{\dagger1/2}\right\Vert
_{2}+\left\Vert C^{\dagger1/2}Z_{n}C^{\dagger1/2}\right\Vert _{2}\right] \\
& \leq\frac{1}{nt_{0}}\left[ 2\left\Vert \sum_{i=1}^{n-1}\Delta
_{i}\right\Vert _{2}+\left\Vert C^{\dagger1/2}Z_{n}C^{\dagger1/2}\right\Vert
_{2}\right]
\end{align*}
where we split $\mathbf{E}_{i}=\Delta_{i}+\Delta_{i}^{\ast}$ :
\begin{equation}
\Delta_{i}=C^{\dagger1/2}\varepsilon_{i}\otimes C^{\dagger1/2}\rho
X_{i-1}+\left( C^{\dagger1/2}\varepsilon_{i}\otimes
C^{\dagger1/2}\varepsilon_{i}-C^{\dagger1/2}C_{\varepsilon}C^{\dagger1/2}
\right) /2.  \label{delta-i}
\end{equation}

As a by product of the computations above we obtain :
\begin{equation*}
\left\Vert R_{n,\alpha}\right\Vert _{2}^{p}\leq\frac{2^{p-1}}{n^{p}t_{0}^{p}}
\left\{ 2^{p}\left\Vert \sum_{i=1}^{n-1}\Delta_{i}\right\Vert
_{2}^{p}+\left\Vert C^{\dagger1/2}Z_{n}C^{\dagger1/2}\right\Vert
_{2}^{p}\right\} .
\end{equation*}
The first step consists in estimating $\mathbb{E}\left\Vert C^{\dagger
1/2}Z_{n}C^{\dagger1/2}\right\Vert _{2}^{p}$ :
\[
\mathbb{E}\left\Vert C^{\dagger1/2}Z_{n}C^{\dagger1/2}\right\Vert _{2}^{p} \leq2^{p-1}\left[ \mathbb{E}\left\Vert C^{\dagger1/2}X_{n}\right\Vert
^{2p}+\left\Vert C^{\dagger}C\right\Vert _{2}^{p}\right] \leq 2^{p-1}\left[ \xi_{p}\left( \frac{1+\eta}{\eta}m\right) ^{p}+\left(
\sum_{k=1}^{+\infty}\frac{\lambda_{k}^{2}}{\left( \lambda_{k}+\alpha\right)
^{2}}\right) ^{p/2}\right]
\]

Leading to
\begin{equation}\label{CZC}
\mathbb{E}\left\Vert C^{\dagger1/2}Z_{n}C^{\dagger1/2}\right\Vert _{2}^{p} \leq 2^{p-1}\left( \frac{1+\eta}{\eta}m\right) ^{p}\left[ \xi_{p}+\left[ 
\frac{\eta}{\left( 1+\eta\right) m}\right] ^{p/2}\right] \leq m^{p}\left( 2
\frac{1+\eta}{\eta}\right) ^{p}\left( \frac{\xi_{p}+\min^{p/2}\left(
1,\eta\right) }{2}\right)
\end{equation}

by respectively Lemma \ref{lem.conc.mom} and direct calculations (see also
Lemma \ref{series}).

The second step handles $\mathbb{E}\left\Vert \sum_{i=1}^{n-1}\Delta
_{i}\right\Vert _{2}^{p}$. We are going to apply Burkholder-Rosenthal
inequality for Hilbert-valued martingales (see \cite{pinelis1994} ). This
inequality writes here :
\begin{equation}
\mathbb{E}\left\Vert \sum_{i=1}^{n-1}\Delta_{i}\right\Vert _{2}^{p}\leq
\mathbf{c}_{p}\left\{ \mathbb{E}\left[ \sum_{i=1}^{n-1}\mathbb{E}\left(
\left\Vert \Delta_{i}\right\Vert _{2}^{2}|X_{i-1}\right) \right]
^{p/2}+\sum_{i=1}^{n-1}\mathbb{E}\left\Vert \Delta_{i}\right\Vert
_{2}^{p}\right\} .  \label{burk-rosen}
\end{equation}

where $\mathbf{c}_{p}$ is deduced from Theorem 4.1 p1688 \cite{pinelis1994}
(a possible choice for $\mathbf{c}_{p}$ is $p/\log p$, see formula (4.3))
and $D=1$ in the notations of Pinelis' article). From (\ref{delta-i}) :
\begin{align*}
\left\Vert \Delta_{i}\right\Vert _{2}^{2} & \leq2\left\Vert C^{\dagger
1/2}\varepsilon_{i}\right\Vert ^{2}\left\Vert C^{\dagger1/2}\rho
X_{i-1}\right\Vert ^{2}+\frac{1}{2}\left\Vert C^{\dagger1/2}\varepsilon
_{i}\otimes C^{\dagger1/2}\varepsilon_{i}-C^{\dagger1/2}C_{\varepsilon
}C^{\dagger1/2}\right\Vert _{2}^{2}, \\
\mathbb{E}\left( \left\Vert \Delta_{i}\right\Vert _{2}^{2}|X_{i-1}\right) &
\leq\tau\left\Vert X_{i-1}\right\Vert ^{2}+\beta, \\
\tau & =2\left\Vert C^{-1/2}\rho\right\Vert ^{2}\mathbb{E}\left\Vert
C^{\dagger1/2}\varepsilon_{1}\right\Vert ^{2},\quad\beta=\frac{\mathbb{E}
\left\Vert C^{\dagger1/2}\varepsilon_{1}\right\Vert ^{4}}{2}.
\end{align*}

Now turning to the first term in the right hand side of (\ref{burk-rosen}) :
\begin{align*}
\mathbb{E}\left[ \sum_{i=1}^{n-1}\mathbb{E}\left( \left\Vert \Delta
_{i}\right\Vert _{2}^{2}|X_{i-1}\right) \right] ^{p/2} & \leq \mathbb{E}
\left[ \tau\sum_{i=1}^{n-1}\left\Vert X_{i-1}\right\Vert ^{2}+\beta\right]
^{p/2} \\
& \leq2^{p/2}\left[ \left( n-1\right) ^{p/2}\tau^{p/2}\mathbb{E}\left( \frac{
1}{n-1}\sum_{i=1}^{n-1}\left\Vert X_{i-1}\right\Vert ^{2}\right)
^{p/2}+\beta^{p/2}\right] \\
& \leq2^{p/2}\left[ \left( n-1\right) ^{p/2}\tau^{p/2}\left( \frac {1}{n-1}
\sum_{i=1}^{n-1}\mathbb{E}\left\Vert X_{i-1}\right\Vert ^{p}\right)
+\beta^{p/2}\right]
\end{align*}
where the discrete Jensen's inequality gives for $p\geq2$ $\left( \frac {1}{
n-1}\sum_{i=1}^{n-1}\left\Vert X_{i-1}\right\Vert ^{2}\right) ^{p/2}\leq
\frac{1}{n-1}\sum_{i=1}^{n-1}\left\Vert X_{i-1}\right\Vert ^{p}$. The last
bound becomes :
\begin{equation*}
\mathbb{E}\left[ \sum_{i=1}^{n-1}\mathbb{E}\left( \left\Vert \Delta
_{i}\right\Vert _{2}^{2}|X_{i-1}\right) \right] ^{p/2}\leq2^{p/2}\left[
\left( n-1\right) ^{p/2}\tau^{p/2}\mathbb{E}\left\Vert X_{1}\right\Vert
^{p}+\beta^{p/2}\right] .
\end{equation*}
By Lemma \ref{lem.conc.mom} :
\begin{align*}
\mathbb{E}\left[ \sum_{i=1}^{n-1}\mathbb{E}\left( \left\Vert \Delta
_{i}\right\Vert _{2}^{2}|X_{i-1}\right) \right] ^{p/2} & \leq 2^{p/2}\left[
\left( 2\xi_{1}^{\prime}\left( a_{2}^{\prime}\right) ^{2}\left( n-1\right)
\left( \mathrm{tr}C\right) \frac{1+\eta}{\eta }m\right)
^{p/2}\xi_{p/2}+\left( \frac{\xi_{2}^{\prime}}{2}\right) ^{p/2}\left( \frac{
1+\eta}{\eta}m\right) ^{p}\right] \\
& \leq \left( nm\right) ^{p/2} \left( \frac{1+\eta}{\eta}\right) ^{p/2}\left[
\left(4 \xi_{1}^{\prime}\left( a_{2}^{\prime}\right) ^{2}\mathrm{tr}C\right) ^{p/2}\xi_{p/2}+\left( \xi_{2}^{\prime}\frac{1+\eta}{\eta}
\right) ^{p/2}\left( \frac{m}{n}\right) ^{p/2}\right] \\
& \leq\mathbf{c}_{p}^{\prime}\left( nm\right) ^{p/2}
\end{align*}
with (for $m/n\leq1$) :
\begin{equation*}
\mathbf{c}_{p}^{\prime}=\left( \frac{1+\eta}{\eta}\right) ^{p/2}\left[
\left(4 \xi_{1}^{\prime}\left( a_{2}^{\prime}\right) ^{2}\mathrm{tr}C\right) ^{p/2}\xi_{p/2}+\left( \xi_{2}^{\prime}\frac{1+\eta}{\eta}
\right) ^{p/2}\right]
\end{equation*}

Let us focus on the second term in the right hand side of (\ref{burk-rosen}):
\[\Delta_{i}=C^{\dagger1/2}\varepsilon_{i}\otimes C^{\dagger1/2}\rho
X_{i-1}+\left( C^{\dagger1/2}\varepsilon_{i}\otimes
C^{\dagger1/2}\varepsilon_{i}-C^{\dagger1/2}C_{\varepsilon}C^{\dagger1/2}
\right) /2.\]
\begin{align*}
\mathbb{E}\left\Vert \Delta_{i}\right\Vert _{2}^{p} & =\mathbb{E}\left\Vert
C^{\dagger1/2}\varepsilon_{i}\otimes C^{\dagger1/2}\rho X_{i-1}+\frac {
C^{\dagger1/2}\varepsilon_{i}\otimes
C^{\dagger1/2}\varepsilon_{i}-C^{\dagger1/2}C_{\varepsilon}C^{\dagger1/2}}{2}
\right\Vert _{2}^{p} \\
& \leq2^{p}\mathbb{E}\left\Vert C^{\dagger1/2}\varepsilon_{i}\otimes
C^{\dagger1/2}\rho X_{i-1}\right\Vert _{2}^{p}+\mathbb{E}\left\Vert
C^{\dagger1/2}\varepsilon_{i}\otimes
C^{\dagger1/2}\varepsilon_{i}-C^{\dagger1/2}C_{\varepsilon}C^{\dagger1/2}
\right\Vert _{2}^{p} \\
& \leq2^{p}\left\{ \mathbb{E}\left\Vert C^{\dagger1/2}\varepsilon
_{i}\right\Vert ^{p}\mathbb{E}\left\Vert C^{\dagger1/2}\rho
X_{i-1}\right\Vert ^{p}+\mathbb{E}\left\Vert
C^{\dagger1/2}\varepsilon_{i}\otimes C^{\dagger
1/2}\varepsilon_{i}\right\Vert _{2}^{p}+\left\Vert
C^{\dagger1/2}C_{\varepsilon}C^{\dagger1/2}\right\Vert _{2}^{p}\right\} \\
& =2^{p}\left[ \mathbb{E}\left\Vert C^{\dagger1/2}\varepsilon_{i}\right\Vert
^{p}\mathbb{E}\left\Vert C^{\dagger1/2}\rho X_{i-1}\right\Vert ^{p}+\mathbb{E
}\left\Vert C^{\dagger1/2}\varepsilon_{i}\right\Vert ^{2p}+\left\Vert
C^{\dagger1/2}C_{\varepsilon}C^{\dagger1/2}\right\Vert _{2}^{p}\right] .
\end{align*}

We start with the last term above :
\begin{align*}
\left\Vert C^{\dagger1/2}C_{\varepsilon}C^{\dagger1/2}\right\Vert _{2}^{p} &
=\left( \sum_{k,l}\frac{\left\langle C_{\varepsilon
}e_{k},e_{l}\right\rangle ^{2}}{\left( \lambda_{k}+\alpha\right) \left(
\lambda_{l}+\alpha\right) }\right) ^{p/2} & \leq\left( \sum_{k,l}\frac{\left\Vert
C_{\varepsilon}^{1/2}e_{k}\right\Vert ^{2}\left\Vert
C_{\varepsilon}^{1/2}e_{l}\right\Vert ^{2}}{\left( \lambda_{k}+\alpha\right)
\left( \lambda_{l}+\alpha\right) }\right) ^{p/2}=\left( \sum_{k}\frac{
\left\Vert C_{\varepsilon}^{1/2}e_{k}\right\Vert ^{2}}{\lambda_{k}+\alpha}
\right) ^{p}.
\end{align*}
Finally :
\begin{equation*}
\sum_{k}\frac{\left\Vert C_{\varepsilon}^{1/2}e_{k}\right\Vert ^{2}}{
\lambda_{k}+\alpha}=\sum_{k}\frac{\left\Vert
C_{\varepsilon}^{1/2}e_{k}\right\Vert ^{2}}{\lambda_{k}}\frac{\lambda_{k}}{
\lambda_{k}+\alpha}\leq\left\Vert C^{-1/2}C_{\varepsilon}^{1/2}\right\Vert
_{\infty}^{2}\sum _{k}\frac{\lambda_{k}}{\lambda_{k}+\alpha}.
\end{equation*}

With the help of Lemma \ref{lem.conc.mom} and taking into account that $
\mathbb{E}\left\Vert C^{\dagger1/2}\rho X_{i-1}\right\Vert ^{p}\leq\left(
a_{2}^{\prime}\right) ^{p}\mathbb{E}\left\Vert X_{1}\right\Vert ^{p}<+\infty$
we obtain finally :
\begin{align*}
\mathbb{E}\left\Vert \Delta_{i}\right\Vert _{2}^{p} & \leq2^{p}\left[ 
\mathbb{E}\left\Vert C^{\dagger1/2}\varepsilon_{1}\right\Vert ^{p}\left(
a_{2}^{\prime}\right) ^{p}\mathbb{E}\left\Vert X_{1}\right\Vert ^{p}+\mathbb{
E}\left\Vert C^{\dagger1/2}\varepsilon_{1}\right\Vert ^{2p}+\left(
a_{3}^{\prime}\right) ^{2p}\left( \sum_{k}\frac{\lambda_{k}}{
\lambda_{k}+\alpha}\right) ^{p}\right] . \\
& \leq2^{p}\left[ \xi_{p/2}^{\prime}\left( \frac{1+\eta}{\eta}m\right)
^{p/2}\left( a_{2}^{\prime}\right) ^{p}\mathbb{E}\left\Vert X_{1}\right\Vert
^{p}+\xi_{p}^{\prime}\left( \frac{1+\eta}{\eta}m\right) ^{p}+\left( \left(
a_{3}^{\prime}\right) ^{2}\frac{1+\eta}{\eta}m\right) ^{p}\right] \\
& \leq\left( 2\frac{1+\eta}{\eta}m\right) ^{p}\left[ \xi_{p/2}^{\prime}
\xi_{p/2}\left( \left( a_{2}^{\prime}\right) ^{2}\sqrt{\frac{\mathrm{tr}C}{m}
\frac{\eta}{1+\eta}}\right) ^{p}+\xi_{p}^{\prime}+\left( a_{3}^{\prime
}\right) ^{2p}\right] \leq\mathbf{c}_{p}^{\prime\prime}m^{p}
\end{align*}
with $\mathbf{c}_{p}^{\prime\prime}=\left( 2\frac{1+\eta}{\eta}\right) ^{p}\left[
\xi_{p/2}^{\prime}\xi_{p/2}\left( \left( a_{2}^{\prime}\right) ^{2}\sqrt{\eta
\mathrm{tr}C}\right) ^{p}+\xi_{p}^{\prime}+\left( a_{3}^{\prime }\right)
^{2p}\right]$.

In order to conclude we have to put together the previous different steps. Remind that we started from :
\[\mathbb{E}\left\Vert R_{n,\alpha}\right\Vert _{2}^{p}  \leq\frac{2^{p-1}}{
n^{p}t_{0}^{p}}\left\{ 2^{p}\mathbb{E}\left\Vert \sum_{i=1}^{n-1}\Delta
_{i}\right\Vert _{2}^{p}+\mathbb{E}\left\Vert
C^{\dagger1/2}Z_{n}C^{\dagger1/2}\right\Vert _{2}^{p}\right\}. \]

From (\ref{CZC}) and (\ref{burk-rosen}) we  get :
\begin{align*}
\mathbb{E}\left\Vert R_{n,\alpha}\right\Vert _{2}^{p}& \leq\left( \frac{m}{n}\right) ^{p/2}\frac{4^{p}}{t_{0}^{p}}\left[ \mathbf{c}_{p}\left\{ \mathbf{c}_{p}^{\prime}+\ell_p (m,n)\mathbf{c}_{p}^{\prime\prime}
\right\} +\left( \frac{m}{n}\right) ^{p/2}\left( \frac{1+\eta}{\eta}\right)
^{p}\left( \frac{\xi_{p}+\min^{p/2}\left( 1,\eta\right) }{2}\right) \right]
\\
& \leq\left( \frac{m}{n}\right) ^{p/2}\mathbf{c}_{p}^{\ast}
\end{align*}

with
\begin{equation*}
\mathbf{c}_{p}^{\ast}\leq\frac{4^{p}}{t_{0}^{p}}\left[ \mathbf{c}_p\left\{ \mathbf{c}
_{p}^{\prime}+\mathbf{c}_{p}^{\prime\prime}\right\} +\left( \frac{1+\eta}{
\eta}\right) ^{p}\left( \frac{\xi_{p}+\min^{p/2}\left( 1,\eta\right) }{2}
\right) \right] .
\end{equation*}
whenever $\ell_p (m,n)=n(m/n)^{p/2} \leq 1$ which is true when $p \geq 2+2/(1+\eta)$. This finishes the proof of Proposition \ref{conc.mom}.

The next result was invoked in the proof of the previous Proposition as a corollary even it is not completely.

\begin{corollary}
\label{corr-conc-U_n}Under assumption (\ref{A3m}) we have for all $p\leq
\mathbf{s}$ :
\begin{align*}
\mathbb{E}\left\Vert U_{n}C^{\dagger1/2}\right\Vert _{2}^{2p} & \leq\mathbf{c
}_{2p}\xi_{p}\mathbf{V}_{n}^{p}\left\{ 1+\frac{1}{n^{p-1}}\left( \frac{1}{
1-a_{1}^{2}}\right) ^{p}\xi_{p}^{\prime}\right\} \leq\mathcal{C}_{s}\mathbf{V
}_{n}^{p} \\
\mathcal{C}_{\mathbf{s}} & \mathbf{=}2\mathbf{s}\max_{p\leq\mathbf{s}
}\xi_{p}\left\{ 1+\left( \frac{1}{1-a_{1}^{2}}\right) ^{\mathbf{s}
}\max_{p\leq\mathbf{s}}\xi_{p}^{\prime}\right\}
\end{align*}
\end{corollary}

Since the derivation of the Corollary
is very similar but much simpler than above, the proof is omitted. The reader has just to notice that $U_{n}C^{\dagger1/2}=\dfrac{1}{n}\sum_{k=1}^{n-1}C^{\dagger
1/2}X_{k}\otimes\varepsilon_{k+1}$ is a Hilbert-Schmidt-valued martingale
array. A direct application of the Burkholder-Rosenthal
inequality (used within the proof of Proposition \ref{conc.mom} but in a more intricate version above) leads to the result.

\subsection{Series evaluation\label{sect-series}}

\begin{lemma}
\label{series}Let $m=\min\left\{ p:\lambda_{p}\leq\alpha\right\} $. Assume
that $\mathbf{A}_{3} $ holds. For all $K\in\mathbb{N}
^{\ast}$ :
\begin{align*}
\frac{m}{2^{K}} & \leq\sum_{p=1}^{+\infty}\left( \frac{\lambda_{p}}{
\lambda_{p}+\alpha}\right) ^{K}\leq\left( 1+\frac{1}{\eta}\right) m, \\
\frac{1}{4}\sum_{p=1}^{m}\frac{\left\Vert \rho e_{p}\right\Vert ^{2}}{
\lambda_{p}} & \leq\sum_{p=1}^{+\infty}\frac{\lambda_{p}}{\left(
\lambda_{p}+\alpha\right) ^{2}}\left\Vert \rho e_{p}\right\Vert ^{2}\leq
\sum_{p=1}^{m}\frac{\left\Vert \rho e_{p}\right\Vert ^{2}}{\lambda_{p}}+
\frac{1}{\xi}\frac{\left\Vert \rho e_{m}\right\Vert ^{2}}{\lambda_{m}}m
\end{align*}
where the constants $\eta$ and $\xi$ appears in $\mathbf{A}_{3}$ and specifiy the smoothness of $X$ and $\rho$. When $\mathbf{A}'_{3}$ holds instead of $\mathbf{A}_{3}$, the inequalities above are preserved except that constants $1/\eta$ and $1/\xi$ in both right hand sides are replaced by constants $c_\eta$ and $c_\xi$ explicited in the proof below.
\end{lemma}

\textbf{Proof of Lemma \ref{series} :} The following bounds are trivial :
\begin{equation*}
\frac{m}{2^{K}}\leq\sum_{k=1}^{m}\left( \frac{\lambda_{k}}{\lambda_{k}+\alpha
}\right) ^{K}\leq m\quad\sum_{k=m+1}^{+\infty}\left( \frac {\lambda_{k}}{
\lambda_{k}+\alpha}\right) ^{K}\leq\sum_{k=m+1}^{+\infty}\frac{\lambda_{k}}{
\lambda_{k}+\alpha}.
\end{equation*}

Besides when $\mathbf{A}_{3}$ holds:
\begin{align*}
\sum_{k=m+1}^{+\infty}\frac{\lambda_{k}}{\lambda_{k}+\alpha} & \leq\frac {1}{
\alpha}\sum_{k=m+1}^{+\infty}\lambda_{k}\leq\frac{\lambda_{m}}{\alpha}
\sum_{k=m+1}^{+\infty}\frac{\lambda_{k}}{\lambda_{m}}\leq\frac{\lambda_{m}}{
\alpha}\sum_{k=m+1}^{+\infty}\left( \frac{m}{k}\right) ^{1+\eta} \\
& \leq\frac{1}{\eta}\frac{\lambda_{m}}{\alpha}m\leq\frac{1}{\eta}m.
\end{align*}
Then for the second part of Lemma \ref{series} :
\begin{equation*}
\sum_{p=1}^{+\infty}\frac{\lambda_{p}}{\left( \lambda_{p}+\alpha\right) ^{2}}
\left\Vert \rho e_{p}\right\Vert ^{2}\geq\sum_{p=1}^{m}\frac{\lambda_{p}}{
\left( \lambda_{p}+\alpha\right) ^{2}}\left\Vert \rho e_{p}\right\Vert
^{2}\geq\frac{1}{4}\sum_{p=1}^{m}\frac{\left\Vert \rho e_{p}\right\Vert ^{2}
}{\lambda_{p}}.
\end{equation*}
And on the other hand :
\begin{equation*}
\sum_{p=1}^{+\infty}\frac{\lambda_{p}}{\left( \lambda_{p}+\alpha\right) ^{2}}
\left\Vert \rho e_{p}\right\Vert ^{2}\leq\sum_{p=1}^{+\infty}\frac{
\left\Vert \rho e_{p}\right\Vert ^{2}}{\lambda_{p}+\alpha}\leq\sum _{p=1}^{m}
\frac{\left\Vert \rho e_{p}\right\Vert ^{2}}{\lambda_{p}}+\sum_{p=m+1}^{+
\infty}\frac{\left\Vert \rho e_{p}\right\Vert ^{2}}{\lambda_{p}+\alpha}
\end{equation*}
with :
\begin{equation*}
\sum_{p=m+1}^{+\infty}\frac{\left\Vert \rho e_{p}\right\Vert ^{2}}{\lambda
_{p}+\alpha}\leq\frac{\left\Vert \rho e_{m}\right\Vert ^{2}}{\alpha}
\sum_{p=m+1}^{+\infty}\frac{\left\Vert \rho e_{p}\right\Vert ^{2}}{
\left\Vert \rho e_{m}\right\Vert ^{2}}\leq\frac{\left\Vert \rho
e_{m}\right\Vert ^{2}}{\lambda_{m}}\sum_{p=m+1}^{+\infty}\left( \frac{m}{k}
\right) ^{1+\xi}\leq\frac{1}{\xi}\frac{\left\Vert \rho e_{m}\right\Vert ^{2}
}{\lambda_{m}}m.
\end{equation*}
where we used $\lambda_{m}\leq\alpha$ and assumption $\mathbf{A}_{3}$.

When $\mathbf{A}'_{3}$ holds we just have to inspect the remainder terms in the series, e.g. $\sum_{p=m+1}^{+\infty}\lambda_{p}$. We have to show for instance that $\sup_m \left\lbrace \sum_{p=m}^{+\infty}\lambda_{p}/(m \lambda_m)  \right\rbrace <+\infty $. Like above in section \ref{bias-study} we switch from sequences to functions and to proving equivalently that $\sup_{x \geq 1} \left\lbrace \int_{s=x}^{+\infty}\lambda(s) ds/(x \lambda(x))  \right\rbrace <+\infty $. But Karamata's Theorem (see e.g. \cite{Bingham1987} Theorem 1.5.11 p.28) ensures that :
\[
\lim _{x \to +\infty} \frac{\int_{s=x}^{+\infty}\lambda(s) ds}{x \lambda(x)}=\frac{1}{\eta}.
\]
When $\eta \in (0,+\infty ]$ this always imply that $c_{\eta}=\sup_{x \geq 1} \left\lbrace \int_{s=x}^{+\infty}\lambda(s) ds/(x \lambda(x))  \right\rbrace <+\infty$. The same method gives a similar result for the sequence $\left\Vert \rho e_{p}\right\Vert ^{2}$. This finishes the proof of the Lemma.

%

\begin{proposition}
\label{AB12}Let $\eta,\beta>0$, $\eta+1\geq\beta$ and $\eta^{\prime},\beta^{
\prime}>0$, $\eta^{\prime}\geq\beta^{\prime}$. The table below gives, up to constants, the value of $\sum_{p=1}^{+\infty}\left\Vert \rho e_{p}\right\Vert
^{2}/(\alpha+\lambda_{p})$ (hence of the left hand side in equation (\ref{A2ter})) as a function of $\alpha$ in the four scenarii considered in Corollary \ref{cor-examples} :

\begin{equation*}
\begin{tabular}{|l|l|l|}
\hline
Case & $I:\left\Vert \rho e_{p}\right\Vert ^{2}\asymp p^{-1-\beta}$ & $II:
\left\Vert \rho e_{p}\right\Vert ^{2}\asymp\exp\left( -\beta^{\prime
}p\right) $ \\ \hline
$A:\lambda_{p}\asymp p^{-1-\eta}$ & $\alpha^{-1+\frac{\beta}{1+\eta}}$
& constant \\ \hline
$\lambda_{p}\asymp\exp\left( -\eta^{\prime}p\right) $ & $
\alpha^{-1}\log^{-\beta}\left( 1/\alpha\right) $ & $\alpha ^{-1+\frac{
\beta^{\prime}}{\eta^{\prime}}}$ \\ \hline
\end{tabular}
\ 
\end{equation*}

\end{proposition}

\textbf{Proof of the Proposition :}

Clearly in Case A-II, $\sum_{p=1}^{+\infty}\left\Vert \rho e_{p}\right\Vert
^{2}/\lambda_{p}<+\infty$. In Case A-I and B-II a simple Riemann integral
computation gives the result. Case B-I deserves a slight development for obtaining similar upper and lower bounds. Notice
that here $m=\left\lfloor 1/\eta^{\prime}\log\left( 1/\alpha\right)
\right\rfloor $. From :
\begin{equation*}
\sum_{p=1}^{+\infty}\frac{
\left\Vert \rho e_{p}\right\Vert ^{2}}{\lambda_{p}+\alpha}
\leq\int_{1}^{+\infty}\frac{\exp\left( \eta^{\prime}x\right) }{
1+\alpha\exp\left( \eta^{\prime}x\right) }\frac{dx}{x^{1+\beta}}
\end{equation*}
we go on with a change of variable :
\begin{equation*}
\int_{1}^{+\infty}\frac{\exp\left( \eta^{\prime}x\right) }{1+\alpha
\exp\left( \eta^{\prime}x\right) }\frac{dx}{x^{1+\beta}}=\eta^{\prime\beta
}\int_{\exp\left( \eta^{\prime}\right) }^{+\infty}\frac{1}{1+\alpha s}\frac{
ds}{\log^{1+\beta}\left( s\right) }
\end{equation*}
which is split in two pieces :
\begin{equation*}
\int_{\exp\left( \eta^{\prime}\right) }^{+\infty}\frac{1}{1+\alpha s}\frac{ds
}{\log^{1+\beta}\left( s\right) }=\int_{\exp\left( \eta^{\prime }\right)
}^{1/\alpha}\frac{1}{1+\alpha s}\frac{ds}{\log^{1+\beta}\left( s\right) }
+\int_{1/\alpha}^{+\infty}\frac{1}{1+\alpha s}\frac{ds}{\log^{1+\beta}\left(
s\right) }
\end{equation*}
where :
\begin{align*}
\frac{1}{2\alpha}\int_{1/\alpha}^{+\infty}\frac{ds}{s\log^{1+\beta}\left(
s\right) } & \leq\int_{1/\alpha}^{+\infty}\frac{\alpha s}{1+\alpha s}\frac{ds
}{\alpha s\log^{1+\beta}\left( s\right) }\leq\frac{1}{\alpha}
\int_{1/\alpha}^{+\infty}\frac{ds}{s\log^{1+\beta}\left( s\right) } \\
\frac{1}{2}\frac{1}{\alpha\beta\log^{\beta}\left( 1/\alpha\right) } &
\leq\int_{1/\alpha}^{+\infty}\frac{1}{1+\alpha s}\frac{ds}{\log^{1+\beta
}\left( s\right) }\leq\frac{1}{\alpha\beta\log^{\beta}\left( 1/\alpha
\right) }.
\end{align*}
Besides :
\begin{align*}
\int_{\exp\left( \eta^{\prime}\right) }^{1/\alpha}\frac{1}{1+\alpha s}\frac{
ds}{\log^{1+\beta}\left( s\right) } & =\int_{\exp\left( \eta^{\prime}\right)
}^{1/\left[ \alpha\log^{\beta}\left( 1/\alpha\right) \right] }\frac{1}{
1+\alpha s}\frac{ds}{\log^{1+\beta}\left( s\right) }+\int_{1/\left[
\alpha\log^{\beta}\left( 1/\alpha\right) \right] }^{1/\alpha}\frac{1}{
1+\alpha s}\frac{ds}{\log^{1+\beta}\left( s\right) } \\
& \leq\frac{1}{\eta^{\prime 1+\beta}\alpha\log^{\beta}\left( 1/\alpha\right) }+
\frac{1}{\alpha}\frac{1}{\log^{1+\beta}\left( \frac{1}{\alpha\log^{\beta
}\left( 1/\alpha\right) }\right) }
\end{align*}
where :
\begin{equation*}
\frac{1}{\alpha}\frac{1}{\log^{1+\beta}\left( \frac{1}{\alpha\log^{\beta
}\left( 1/\alpha\right) }\right) }=\frac{1}{\alpha\log^{\beta}\left(
1/\alpha\right) }\frac{\log^{-1}\left( \frac{1}{\alpha}\right) }{\left(
1-\beta\frac{\log\log\left( 1/\alpha\right) }{\log\left( \frac{1}{\alpha }
\right) }\right) ^{1+\beta}}
\end{equation*}
It is simple algebra to show that for $s\geq1$ $s^{-1}\log s\leq2/\left( e
\sqrt{s}\right) $ hence for $t\in\left( 0,1\right) $ $1-\beta\frac {
\log\log\left( 1/\alpha\right) }{\log\left( \frac{1}{\alpha}\right) }>t$
whenever $\alpha<\exp\left[ -4\beta^{2}/e^{2}\left( 1-t\right) ^{2}\right] $
and choosing finally $t=1/2$ :
\begin{equation*}
\frac{\log^{-1}\left( \frac{1}{\alpha}\right) }{\left( 1-\beta\frac {
\log\log\left( 1/\alpha\right) }{\log\left( \frac{1}{\alpha}\right) }\right)
^{1+\beta}}\leq\frac{1}{t^{1+\beta}}\left( \frac{e\left( 1-t\right) }{2\beta}
\right) ^{2}=2^{1+\beta}\left( \frac{e}{4\beta}\right) ^{2}\leq\frac{
2^{\beta}}{\beta^{2}}
\end{equation*}
so that :
\begin{equation*}
\sum_{p=1}^{+\infty}\frac{\left\Vert \rho e_{p}\right\Vert ^{2}}{ \lambda_{p}+\alpha}
\leq\frac{1}{\alpha\log^{\beta }\left(
1/\alpha\right) }\left( \frac{1}{\eta^{\prime\beta}}+\frac {2^{\beta}}{
\beta^{2}}\right) .
\end{equation*}
Now we deal with the lower bound left a bit sooner :
\begin{equation*}
\sum_{p=1}^{+\infty}\frac{\lambda_{p}}{\left( \lambda_{p}+\alpha\right) ^{2}}
\left\Vert \rho e_{p}\right\Vert ^{2}\geq\sum_{p=m}^{+\infty}\frac{
\lambda_{p}}{\left( \lambda_{p}+\alpha\right) ^{2}}\left\Vert \rho
e_{p}\right\Vert ^{2}\geq\frac{1}{2}\sum_{p=m}^{+\infty}\frac{\left\Vert
\rho e_{p}\right\Vert ^{2}}{\lambda_{p}+\alpha}.
\end{equation*}
Copying the method above for the  lower bound :

\begin{equation*}
\sum_{p=1}^{+\infty}\frac{\left\Vert \rho e_{p}\right\Vert ^{2}}{\lambda
_{p}+\alpha}\geq\int_{1/\eta^{\prime}\log\left( 1/\alpha\right) }^{+\infty }
\frac{\exp\left( \eta^{\prime}x\right) }{1+\alpha\exp\left( \eta^{\prime
}x\right) }\frac{dx}{x^{1+\beta}}=\int_{1/\alpha}^{+\infty}\frac{1}{1+\alpha
s}\frac{ds}{\log^{1+\beta}\left( s\right) }\geq\frac{1}{2}\frac{1}{
\alpha\beta\log^{\beta}\left( 1/\alpha\right) }
\end{equation*}
which leads to the announced result in Case B-I.\bigskip

\bigskip

\begin{proposition}
\label{closed-op}When (\ref{A1}) holds $1<\left\Vert
C_{\varepsilon}^{-1/2}C^{1/2}\right\Vert _{\infty}<+\infty$, besides 
$C_{\varepsilon }^{-1/2}C^{1/2}$ is invertible with :
\begin{equation*}
\left\Vert
C^{-1/2}C_{\varepsilon}^{1/2}\right\Vert _{\infty}<1, \quad \left\Vert
C^{-1/2}\rho\right\Vert _{\infty}<+\infty.
\end{equation*}
In addition denote $\widetilde{C}
_{\alpha}=C^{\dagger 1/2}\rho\left( C+\alpha\mathbf I\right) ^{1/2}$, take any $
t_{0}\in\left( a_{3}^{-2},1\right) $ and $\alpha<\left(
t_{0}-a_{3}^{-2}\right) /\left( a_{2}^{\prime}\right) ^{2}=\alpha_{0}$ then $\left\Vert \widetilde{C}
_{\alpha}\right\Vert _{\infty}<\sqrt{1-t_{0}}$. 
\end{proposition}

\textbf{Proof of the Proposition :}

From (\ref{cov}) the following equation is valid on $\mathcal{D}
\left( C_{\varepsilon}^{-1/2}\right) $:
\begin{equation*}
C_{\varepsilon}^{-1/2}CC_{\varepsilon}^{-1/2}=C_{\varepsilon}^{-1/2}\rho
C\rho^{\ast}C_{\varepsilon}^{-1/2}+C_{\varepsilon}^{-1/2}C_{\varepsilon
}C_{\varepsilon}^{-1/2}.
\end{equation*}
The operator $C_{\varepsilon}^{-1/2}C_{\varepsilon}C_{\varepsilon}^{-1/2}$
is a bounded densely defined operator whose extension is the Identity. And $
C_{\varepsilon}^{-1/2}\rho C\rho^{\ast}C_{\varepsilon}^{-1/2}$ is bounded densely defined as
well due to the assumption $\left\Vert C_{\varepsilon}^{-1/2}\rho\right\Vert
_{\infty}<+\infty$. Then so is $C_{\varepsilon}^{-1/2}CC_{
\varepsilon}^{-1/2}=B_{1}B_{1}^{\ast}$ with $B_{1}=C_{
\varepsilon}^{-1/2}C^{1/2}$ and the latter is also bounded. Besides $
C_{\varepsilon}^{-1/2}CC_{\varepsilon}^{-1/2}$ extends to $H$ and $\overline{
C_{\varepsilon}^{-1/2}CC_{\varepsilon}^{-1/2}}=I+\overline{
C_{\varepsilon}^{-1/2}\rho C\rho^{\ast}C_{\varepsilon}^{-1/2}}=I+T$ where $T$
is positive which implies first that $\left\Vert
C_{\varepsilon}^{-1/2}C^{1/2}\right\Vert _{\infty}>1$ and second that $
\overline{C_{\varepsilon}^{-1/2}CC_{\varepsilon}^{-1/2}}$ has a bounded
inverse namely $C_{\varepsilon}^{1/2}C^{-1}C_{
\varepsilon}^{1/2}=B_{2}B_{2}^{\ast} $ with $B_{2}^{\ast}=C^{-1/2}C_{
\varepsilon}^{1/2}.$ This operator $B_{2}^{\ast}$ has a bounded inverse too
: $C_{\varepsilon}^{-1/2}C^{1/2}$. With $I=\overline{C^{-1/2}\rho
C\rho^{\ast}C^{-1/2}}+\overline{C^{-1/2}C_{\varepsilon}C^{-1/2}}$ we deduce
that $\left\Vert B_{2}^{\ast}\right\Vert _{\infty}\leq1$. Now focus on the last sentence of the Proposition about $\widetilde{C}_{\alpha}$.

First $\left\Vert 
\widetilde{C}_{\alpha}\right\Vert _{\infty}\leq\left\Vert C^{-1/2}\rho\left(
C+\alpha\mathbf I\right) ^{1/2}\right\Vert _{\infty}$ because $\left\Vert
C^{\dagger}C\right\Vert _{\infty}\leq1$ so that we can take $\widetilde{C}_{\alpha}$ to be the operator $C^{-1/2}\rho\left(
C+\alpha\mathbf I\right) ^{1/2}$. Second : a sufficient condition here is to show the result with $\widetilde{C}_{\alpha}\left( \widetilde{C}
_{\alpha}\right) ^{\ast}$ instead of $\widetilde{C}_{\alpha}$. But $
\widetilde{C}_{\alpha}\left( \widetilde{C}_{\alpha}\right) ^{\ast }=
\overline{C^{-1/2}\rho\left( C+\alpha\mathbf I\right) \rho^{\ast}C^{-1/2}}$ which
may be split in two operators $\overline{C^{-1/2}\rho C\rho^{\ast}C^{-1/2}}$
and $\alpha\overline{C^{-1/2}\rho\rho^{\ast}C^{-1/2}}$. All what was done
above one may be copied to show that $C^{-1/2}\rho C\rho^{\ast}C^{-1/2}$ has
a bounded extension because $C^{-1/2}C_{\varepsilon}^{1/2}$ is bounded. More
precisely from :
\begin{align*}
C & =\rho C\rho^{\ast}+C_{\varepsilon} \\
C^{-1/2}CC^{-1/2} & =C^{-1/2}\rho
C\rho^{\ast}C^{-1/2}+C^{-1/2}C_{\varepsilon}C^{-1/2}
\end{align*}
we write for all $x\in\mathcal{D}\left( C^{-1/2}\right) $, $\left\Vert
x\right\Vert =1$ :
\begin{equation*}
\left\Vert C^{-1/2}\rho C^{1/2}x\right\Vert ^{2}=\left\langle C^{-1/2}\rho
C\rho^{\ast}C^{-1/2}x,x\right\rangle =1-\left\langle C^{-1/2}C_{\varepsilon
}C^{-1/2}x,x\right\rangle \leq1-\inf_{x\in\mathcal{D}\left( C^{-1/2}\right)
,\left\Vert x\right\Vert =1}\left\Vert
C_{\varepsilon}^{1/2}C^{-1/2}x\right\Vert ^{2}
\end{equation*}
but $\inf_{\left\Vert x\right\Vert =1,x\in\mathcal{D}\left( C^{-1/2}\right)
}\left\Vert C_{\varepsilon}^{1/2}C^{-1/2}x\right\Vert ^{2}\geq\left\Vert
C^{1/2}C_{\varepsilon}^{-1/2}\right\Vert _{\infty}^{-2}$ because $\left\Vert
x\right\Vert \leq\left\Vert T\right\Vert _{\infty}\left\Vert
T^{-1}x\right\Vert $ for all $x$ in $\mathcal{D}(T^{-1})$ and for bounded $T$ with densely
defined $T^{-1}$. We sum up our result : 
\begin{equation*}
\sup_{\left\Vert x\right\Vert =1,x\in\mathcal{D}\left( C^{-1/2}\right)
}\left\Vert C^{1/2}\rho C^{-1/2}x\right\Vert ^{2}\leq1-\left\Vert
C^{1/2}C_{\varepsilon}^{-1/2}\right\Vert _{\infty}^{-2}<1
\end{equation*}
because we proved above that $\left\Vert
C^{1/2}C_{\varepsilon}^{-1/2}\right\Vert _{\infty}>1$. With property (F1)
above at hand this gives : $<1-\left\Vert
C^{1/2}C_{\varepsilon}^{-1/2}\right\Vert _{\infty}^{-2}$Collecting the
previous results we get :
\begin{equation*}
\left\Vert \widetilde{C}_{\alpha}\left( \widetilde{C}_{\alpha}\right)
^{\ast}\right\Vert _{\infty}\leq\left\Vert \overline{C^{-1/2}\rho
C\rho^{\ast }C^{-1/2}}\right\Vert _{\infty}+\left\Vert \alpha\overline{
C^{-1/2}\rho \rho^{\ast}C^{-1/2}}\right\Vert _{\infty}\leq1-\left\Vert 
\overline {C^{1/2}C_{\varepsilon}^{-1/2}}\right\Vert
_{\infty}^{-2}+\alpha\left\Vert C^{-1/2}\rho\right\Vert _{\infty}^{2}
\end{equation*}
Then $\left\Vert \widetilde{C}_{\alpha}\right\Vert _{\infty}<\sqrt{1-t_{0}}$
when $\alpha<\left( t_{0}-a_{3}^{-2}\right) /\left( a_{2}^{\prime}\right)
^{2}.$

\bigskip \bigskip

\bibliography{bib-ARH-Minimax}

\end{document}